\documentclass[oneside,english]{amsart}
\usepackage[T1]{fontenc}
\usepackage[utf8]{inputenc}
\usepackage{lmodern}
\usepackage{amsthm,xcolor,hyperref}
\usepackage{amstext}
\usepackage{amssymb}
\usepackage[a4paper]{geometry} 
\usepackage[textwidth=2cm, textsize=scriptsize]{todonotes}

\makeatletter

\newtheorem{thm}{Theorem}
\newtheorem{definition}[thm]{Definition}
\newtheorem{lem}[thm]{Lemma}
\newtheorem{prop}[thm]{Proposition}
\newtheorem{rmk}[thm]{Remark}
\newtheorem{cor}[thm]{Corollary}

\newcommand*{\ensembledenombres}{\mathbb}
\newcommand*{\R}{\ensembledenombres{R}}
\newcommand*{\C}{\ensembledenombres{C}}
\newcommand*{\N}{\ensembledenombres{N}}

\renewcommand*{\P}{\mathbb{P}}
\newcommand*{\esp}{\mathbb{E}}
\newcommand{\1}{1\!\!{\sf I}}
\DeclareMathOperator{\Var}{Var}

\DeclareMathOperator{\Tr}{Tr}

\renewcommand{\Im}{\mathcal{I}}
\renewcommand{\Re}{\mathcal{R}}

\begin{document}

\title[Fluctuations of the Stieltjes transform of ESD of selfadjoint polynomials]{Fluctuations of the Stieltjes transform of the empirical spectral distribution of selfadjoint polynomials in Wigner and deterministic diagonal matrices
}
\author{Serban Belinschi}
\address{Institut de Math\'ematiques de Toulouse: UMR5219; Universit\'e de Toulouse; CNRS; UPS, F-31062 Toulouse, France.}
\email{serban.belinschi@math.univ-toulouse.fr}

\author{Mireille Capitaine}
\address{Institut de Math\'ematiques de Toulouse: UMR5219; Universit\'e de Toulouse; CNRS; UPS, F-31062 Toulouse, France.}
\email{mireille.capitaine@math.univ-toulouse.fr}

\author{Sandrine Dallaporta}
\address{Laboratoire de Mathématiques et Applications, UMR7348, Université de Poitiers, Téléport 2-BP30179, Boulevard Marie et Pierre Curie, 86962 Chasseneuil, France.}
\email{sandrine.dallaporta@math.univ-poitiers.fr}

\author{Maxime Fevrier}
\address{Universit\'e Paris-Saclay, CNRS, Laboratoire de math\'ematiques d’Orsay, 91405, Orsay, France.}
\email{maxime.fevrier@universite-paris-saclay.fr}

\thanks{This research was supported in part by funding from the Simons Foundation and the Centre de Recherches Mathématiques, through the Simons-CRM scholar-in-residence program.}

\begin{abstract}
We investigate the fluctuations around the mean of the Stieltjes transform of the empirical spectral distribution of any selfadjoint  noncommutative polynomial in a Wigner matrix and a deterministic diagonal matrix. We obtain the convergence in distribution to a centred complex Gaussian process whose covariance is expressed in terms of operator-valued subordination functions.
\end{abstract}

\maketitle

\section{Introduction}

From the pioneering work of Wigner to the latest developments, properties of eigenvalues of Wigner matrices have been a major subject in Random Matrix Theory. The celebrated Wigner Theorem states that the empirical spectral distribution of a Wigner matrix -- which means a complex Hermitian or real symmetric matrix whose entries are centered, with variance $\sigma^2$, and independent up to the symmetry condition -- weakly converges in probability to the semicircle law $\mu_{\sigma^2}$ with density $(2\pi \sigma^2)^{-1}\sqrt{4\sigma^2-x^2}\mathbf{1}_{[-2\sigma,2\sigma]}(x)$. It is then straightforward to deduce the convergence of linear statistics $N^{-1}\sum_{i=1}^Nf(\lambda_i)$ of eigenvalues $\lambda_1,\ldots ,\lambda_N$ of $N\times N$ Wigner matrices associated to bounded continuous test functions $f:\mathbb{R}\to\mathbb{R}$ towards $\int_{\R}fd\mu_{\sigma^2}$.

Among the questions that have been addressed, fluctuations of linear statistics have attracted some attention.
Initiated in the mid-nineties by investigations 
on traces of resolvents of real Wigner matrices, central limit theorems for linear spectral statistics of Wigner matrices progressively emerged.
Sinai and Soshnikov \cite{SiSo98b}, by the method of moments, and Bai and Yao \cite{BaiYao05} (see also Bao and Xie \cite{BaoXie16}), by applying a central limit theorem for martingale differences to the trace of the resolvent, obtained the fluctuations of linear spectral statistics associated to analytic test functions. These central limit theorems have been progressively extended to functions with less regularity: by Pastur and Lytova \cite{LytPas09a,LytPas09b} using Fourier analysis, for functions with sufficiently fast decaying Fourier transform, by Bai, Wang and Zhou \cite{BaWaZh09} for $\mathcal{C}^4$ functions, making use of Bernstein polynomials, by Shcherbina \cite{Shcherbina11}, Sosoe and Wong \cite{SosWon13} for $\mathcal{H}^{s}$ functions by a density argument, and Kopel \cite{Kopel} by precise computations on complex Gaussian Wigner matrices. Recently, Bao and He \cite{BaHe21} provided a near optimal rate of convergence for these central limit theorems in Kolmogorov-Smirnov distance. 

Gaussian fluctuations with different scale, mean and variance also hold for linear spectral statistics when the entries of the Wigner matrix have an infinite fourth moment 
(\cite{BGMal16}; see also \cite{BGGuiMal14} for the case of non square integrable entries, in which case Wigner's Theorem fails to hold \cite{BAGui08}). 
When entries of the Wigner matrix are not identically distributed in such a way that 
their variances differ (these matrices are called band matrices or sometimes Wigner matrices with variance profile), 
fluctuations of linear spectral statistics have also been described (see \cite{AdhJanSah} and references therein). 
Fluctuations of linear spectral statistics were also investigated at the mesoscopic scale. In this type of study, the object of interest is $\sum_{i=1}^Nf\big(N^{\alpha}(E-\lambda_i)\big)$, where $\alpha \in (0,1)$ and $E \in (-2\sigma,2\sigma)$, see \cite{BoKh99,LoSi15,HeKn17} for Wigner matrices and \cite{LiXu21} for Wigner matrices with variance profile.

When a Wigner matrix is deformed by an additive (random or deterministic) perturbation, it is a natural question to characterize the effect of the perturbation on fluctuations of linear spectral statistics. This question was raised very early by Khorunzhy \cite{Khorunzhy94}, who proved in the case of deformed real Gaussian Wigner matrices that the fluctuations were still Gaussian, but without providing an explicit covariance kernel. After contributions by Dembo, Guionnet and Zeitouni \cite{DeGuZe03} still in the Gaussian case by a dynamical approach, and by Su \cite{Su13} in the case of a random diagonal perturbation on another scale, the topic has been recently revived. Motivated by the analysis of spherical Sherrington-Kirkpatrick model or by the problem of statistical detection of noisy signals, the case of a deterministic rank one perturbation has been considered by Baik and Lee \cite{BaikLee17}, Baik, Lee and Wu \cite{BaLeWu18}, Chung and Lee \cite{ChLe18}, Jung, Chung and Lee \cite{ChJuLe20,ChJuLe21}. Diagonal perturbations with general rank were further investigated by Ji and Lee \cite{JiLee}, Dallaporta and Fevrier \cite{DaFe19}. Recently, Li, Schnelli and Xu provided the fluctuations of the linear spectral statistics for deformed Wigner matrices at mesoscopic scale \cite{LiScXu19}.

These papers make naturally use of free probability theory. Indeed, in an influential paper, Voiculescu gave evidence that the noncommutative probability theory he had previously introduced, called free probability theory, was a convenient framework for dealing with the convergence of the process of traces of noncommutative polynomials in several complex Gaussian Wigner matrices \cite{Vo91}. Dykema proved then that polynomials in more general Wigner and deterministic matrices also fit in this framework \cite{Dy93}. See also the book \cite{MS} and the paper \cite{BC}. An analog of the free probability framework for dealing with fluctuations of the process of traces of noncommutative polynomials in Gaussian Wigner and deterministic matrices was designed in a series of papers \cite{MiSp06,MiSnSp07,CMSS07}, building on the work of Mingo and Nica \cite{MiNi04} (see also the dynamical approach to a close question by Cabanal-Duvillard and Guionnet \cite{Cabanal01,Guionnet02,CaGu01,DeGuZe03}). However, this so-called second order free probability theory does not seem to be the relevant framework to describe fluctuations of the process of traces of noncommutative polynomials in more general Wigner and deterministic matrices, as witnessed by the recent combinatorial analysis by Male, Mingo, Péché and Speicher \cite{MaMiPeSp20}.

In this paper, we tackle the slightly different question of fluctuations of the Stieltjes transform of the empirical spectral distribution of general polynomials in a Wigner matrix and a diagonal deterministic matrix. Less is known on this question beyond the case of linear polynomials, which is equivalent to that of deformed Wigner matrices. Using a well-known linearization trick to convert our initial general noncommutative polynomial with complex coefficients into a linear polynomial with matrix coefficients and then adapting the strategy used by Bai and Yao \cite{BaiYao05} for a single Wigner matrix and upgraded independently by Ji and Lee \cite{JiLee} and Dallaporta and Fevrier \cite{DaFe19} to deal with deformed Wigner matrices, we establish a central limit theorem for the analytic process of traces of the resolvent. Since complex coefficients are replaced by matrix coefficients,
one has to rely on the operator-valued version of free probability to express and analyze the limiting covariance kernel. The latter involves the logarithm of an operator and a highly non-trivial  first task is to prove that this logarithm is well defined, making use of the contractivity of analytic self-maps on hyperbolic domains.
Moreover, to adapt the strategy previously  used for deformed Wigner matrices  to general polynomials, many commutativity issues arise and require very technical preliminary results and new approaches. A typical example of these difficulties is the study of the second and third terms in the so-called  hook process (see Sections \ref{sec_2nd_term} and \ref{sec_3rd_term}): it requires  a  new trick consisting in writing these terms as the trace of the sum of images of matrix tensors by fit operators and in proving that 
each of these matrix tensors satisfies an approximated fixed point equation;  then, a non-trivial study of spectral radius of operators is still necessary.
 
Besides the Introduction, the paper is organized as follows. Section \ref{Model} introduces the random matrices considered in this work whereas Section \ref{mainresults} presents our main results. Section \ref{sec:preliminary_results} is devoted to basic background on linearization trick and operator-valued free probability theory that are central in our approach. In Section \ref{defobjetlimite}, we prove that the limiting covariance kernel involved in our central limit theorem is well defined and  Section \ref{prelimi}  gathers numerous preliminary results (bounds, concentration bounds, convergence results...) used in its proof.
The proof of the convergence of finite-dimensional distributions of the process of Stieltjes transforms is detailed in Section \ref{CLT}.  
In Section \ref{tightanalytic}, we establish the tightness of this sequence of random analytic functions.
Three appendices
conclude the paper: the first one gathers tools from elementary linear algebra, random quadratic forms, 
martingale theory and  complex analysis  used in the proofs; the second one establishes results on moments and norm of Wigner matrices that are used throughout the paper;
the last one details the truncation argument allowing to assume that entries of Wigner matrices considered in this paper are almost surely bounded by a sequence slowly converging to 0.

\section{Presentation of the model} \label{Model}

The complex algebra $\mathbb{C}\langle t_1,\ldots ,t_n\rangle$ of polynomials with complex coefficients in $n$ noncommuting indeterminates $t_1,\ldots ,t_n$ becomes a $*$-algebra by anti-linear extension of
$$(t_{i_1}t_{i_2}\cdots t_{i_l})^* = t_{i_l}\cdots t_{i_2}t_{i_1},\quad i_1, \ldots, i_l=1,\ldots,n,\, l\in \mathbb{N}.$$
We consider, on a probability space, a sequence of random matrices $$X_N:=P(W_N,D_N),\quad N\in \mathbb{N},$$ where : 
\begin{enumerate}
\item\label{hyp:polynomial} $P\in \mathbb{C}\langle t_1,t_2\rangle$ is a selfadjoint polynomial in two noncommuting indeterminates;
\item\label{hyp:indep} entries $\{W_{ij}\}_{1\leq i\leq j\leq N}$ of the $N\times N$ Hermitian matrix $W_N$ are independent random variables;
\item\label{hyp:offdiagonal} off-diagonal entries $\{W_{ij}\}_{1\leq i< j\leq N}$ of $W_N$ are identically distributed complex random variables such that, for some $\varepsilon >0$, $\mathbb{E}[|\sqrt{N}W_{ij}|^{6(1+\varepsilon)}]\leq C_6$. We assume that $\mathbb{E}[W_{ij}]=0$ and that
$$\sigma_N^2:=\mathbb{E}[|W_{ij}|^2]\geq 0,\quad \theta_N:=\mathbb{E}[W_{ij}^2]\in \mathbb{C},\quad \kappa_N:=\mathbb{E}[|W_{ij}|^4]-2\sigma_N^4-|\theta_N|^2\in \mathbb{R}.$$
satisfy 
$$\lim_{N\to +\infty}N\sigma_N^2=\sigma^2> 0,\quad \lim_{N\to +\infty}N\theta_N=\theta\in \mathbb{R},\quad \lim_{N\to +\infty}N^2\kappa_N=\kappa\in \mathbb{R}.$$
The assumption $\theta\in \mathbb{R}$ 
means that correlations between the real and imaginary parts of off-diagonal entries of $W_N$ are negligible.
\item\label{hyp:diagonal} diagonal entries $\{W_{ii}\}_{1\leq i\leq N}$ of $W_N$ are identically distributed real random variables such that, for some $\varepsilon >0$, $\mathbb{E}[|\sqrt{N}W_{ii}|^{4(1+\varepsilon)}]\leq C_4$. We assume that $\mathbb{E}[W_{ii}]=0$ and that
$\tilde{\sigma}_N^2:=\mathbb{E}[W_{ii}^2]\geq 0$ satisfies $\lim_{N\to +\infty}N\tilde{\sigma}_N^2=\tilde{\sigma}^2> 0$;
\item\label{hyp:deformation} $D_N$ is a $N\times N$ deterministic real diagonal matrix. We assume that $\sup_{N\in \mathbb{N}}\|D_N\|<\infty$ and, for some Borel probability measure $\nu$ on $\mathbb{R}$,
$$\nu_N:=\frac{1}{N}\sum_{\lambda \in \text{sp}(D_N)}\delta_{\lambda} \Rightarrow \nu.$$
\end{enumerate}
Here and below, we use the notation $\text{sp}(A)$ for the (multi)set of eigenvalues (counted with their algebraic multiplicity) of a square matrix $A$.
We will also assume that all entries of $W_N$ are almost surely bounded by $\delta_N$, 
where $(\delta_N)_{N\in \mathbb{N}}$ is a sequence of positive numbers slowly converging to $0$ 
(at rate less than $N^{-\epsilon}$ for any $\epsilon >0$); 
this may be assumed without loss of generality, as proved in Appendix \ref{Truncation}. 
We will use the notation $m_N:=\esp[|W_{ij}|^4]=\kappa_N+2\sigma_N^4+|\theta_N|^2$.
In Assumptions \eqref{hyp:offdiagonal} and \eqref{hyp:diagonal}, we ask the entries to be identically distributed. This assumption does not seem to be necessary for our main result to hold, but leads to a simplification of truncation-centering-homogeneization arguments. Therefore, for the readibility of the paper, we will not pursue the task to relax this assumption. 

We are interested in the empirical spectral measure $\mu_N$ of $X_N$, defined by: 
$$\mu_N:=\frac{1}{N}\sum_{\lambda \in \text{sp}(X_N)}\delta_{\lambda}.$$
More precisely, we study the fluctuations of the Stieltjes transform $\mathbb{C}\setminus \mathbb{R}\ni z\mapsto \int_{\mathbb{R}}(z-x)^{-1}\mu_N(dx)$ of $\mu_N$ around its mean.

\section{Main result} \label{mainresults}

Before stating our main result, we introduce the necessary material. 

\subsection{Definitions}

\subsubsection{Free probability}

Let $\mathcal{A}$ be a complex algebra with a unit $1_{\mathcal{A}}$ and $\varphi : \mathcal{A} \to \mathbb{C}$ be a linear functional satisfying $\varphi(1_{\mathcal{A}})=1$. One usually calls $(\mathcal{A},\varphi)$ a noncommutative probability space and its elements noncommutative random variables. 
	
We say that a noncommutative random variable $a\in (\mathcal{A},\varphi)$ is distributed according to a Borel probability measure $\mu$ on $\mathbb{R}$ when $\varphi(a^n)=\int_{\mathbb{R}}t^n\mu(dt)$ holds for all $n\in \mathbb{N}$. In particular, a semicircular element with mean $0$ and variance $\sigma^2$ is a noncommutative random variable distributed according to the absolutely continuous probability measure with density $(2\pi\sigma^2)^{-1}\sqrt{4\sigma^2-t^2}\mathbf{1}_{[-2\sigma;2\sigma]}(t)$ in some noncommutative probability space.

Two noncommutative random variables $s,d$ in a noncommutative probability space $(\mathcal{A},\varphi)$ are said to be freely independent if the following holds : for each $n\in \mathbb{N}$ and any polynomials with complex coefficients $P_1,\ldots ,P_n,Q_1,\ldots ,Q_n\in \mathbb{C}[t]$, 
\[\varphi((P_1(s)-\varphi(P_1(s))1_{\mathcal{A}})(Q_1(d)-\varphi(Q_1(d))1_{\mathcal{A}})\cdots (P_n(s)-\varphi(P_n(s))1_{\mathcal{A}})(Q_n(d)-\varphi(Q_n(d))1_{\mathcal{A}}))=0.\]

\subsubsection{Linearization}

A powerful tool to deal with noncommutative polynomials in random matrices or in operators is the so-called ``linearization trick'' that 
goes back to Haagerup and Thorbj{\o}rnsen \cite{HT05,HT06}  in the context of operator algebras and random matrices (see \cite{MS}). 
We use the procedure introduced in \cite[Proposition 3]{A}.\\

Given a polynomial $P\in\mathbb C\langle t_1,\ldots ,t_n\rangle$, we call {\it linearization} of $P$ any
$$L_P := \begin{pmatrix} 0 & u\\v & Q \end{pmatrix} \in M_m(\mathbb{C}) \otimes \mathbb{C} \langle t_1,\ldots ,t_n\rangle,$$
where
\begin{enumerate}
	\item $m \in \mathbb{N}$,
	\item $Q \in M_{m-1}(\mathbb{C})\otimes \mathbb{C} \langle t_1,\ldots ,t_n\rangle$ is invertible,
	\item $u$ is a row vector and $v$ is a column vector, both of size $m-1$ with
	entries in $\mathbb{C}\langle t_1,\ldots ,t_n\rangle$,
	\item the polynomial entries in $Q, u$ and $v$ all have degree $\leq 1$,
	\item $P=-uQ^{-1}v$.
\end{enumerate}

Given a selfadjoint polynomial $P\in\mathbb C\langle t_1,\ldots ,t_n\rangle$, it is described in Section 4 of \cite{BC}, from Anderson \cite{A} (see also \cite{MaiThesis}), how to build a particular selfadjoint linearization $L_P \in M_m(\mathbb{C}) \otimes \mathbb{C} \langle t_1,\ldots ,t_n\rangle$ of $P$. We call this particular linearization the {\it canonical} linearization of $P$.	

\subsection{Statement of results}

Let $P\in \mathbb{C}\langle t_1,t_2\rangle$ be the selfadjoint polynomial in two noncommuting indeterminates involved in the definition of $X_N$ (see (1) in Section \ref{Model}). Let  $L_P=\gamma_0\otimes 1+
\gamma_1\otimes t_1+\gamma_2\otimes t_2$ be the canonical linearization of $P$ ;  $\gamma_0$, $\gamma_1$, $\gamma_2$
are selfadjoint  matrices in ${ M}_m(\mathbb{C})$.

Let $s$ be a semicircular element with mean $0$ and variance $\sigma^2$ which is freely independent from a noncommutative variable $d$ distributed according to the probability measure $\nu$ (see (5) in Section \ref{Model}) in some noncommutative probability space $(\mathcal{ A}, \varphi)$. We denote $\mathrm{id}_m:M_m(\mathbb{C})\to M_m(\mathbb{C})$ the identity map and define a map $\omega$ by
$$\omega(b)=b-\sigma^2\gamma_1 (\mathrm{id}_m \otimes \varphi )\left[\left(b\otimes 1_{\mathcal{ A}} - \gamma_1 \otimes s-\gamma_2 \otimes d \right)^{-1}\right]\gamma_1,$$ 
for any $b\in M_m(\mathbb{C})$ such that $b\otimes 1_{\mathcal{ A}} - \gamma_1 \otimes s-\gamma_2 \otimes d$ is invertible in $M_m(\mathbb{C})\otimes \mathcal{A}$. 
As explained in Lemma \ref{inversible} below, $\omega$ is well-defined on $\{ze_{11} -\gamma_0,z\in \mathbb{C}\setminus \mathbb{R}\}$. 
Section \ref{opfree} will describe the occurence of this so-called subordination map.

We denote, for $\beta_1,\beta_2\in \{\beta\in M_m(\mathbb{C}), \omega(\beta)\otimes I_m-\gamma_2\otimes d \text{ is invertible}\}$, by $T_{\{\beta_1,\beta_2\}}$ the operator defined on $M_m(\mathbb{C})\otimes M_m(\mathbb{C})$ by  $$T_{\{\beta_1,\beta_2\}}(x)=\int_{\mathbb{R}}((\omega(\beta_1)-t\gamma_2)^{-1}\gamma_1\otimes I_m)x(I_m\otimes \gamma_1 (\omega(\beta_2)-t\gamma_2)^{-1})d\nu(t),\quad x\in M_m(\mathbb{C})\otimes M_m(\mathbb{C}).$$
The limiting distribution of our central limit theorem involves the logarithms of some operators of the form $${\rm id}_m\otimes{\rm id}_m-T:M_m(\mathbb{C})\otimes M_m(\mathbb{C})\to M_m(\mathbb{C}) \otimes M_m(\mathbb{C}),$$
with $T$ some scalar multiples of $T_{\{z_1e_{11}-\gamma_0,z_2e_{11}-\gamma_0\}}$. Thus, our first result consists in proving that these logarithms are well-defined. 
Since $\log\left({\rm id}_m\otimes{\rm id}_m-T\right)$ is well defined by the convergent series expansion
\begin{equation}\label{logseries}
\log\left[{\rm id}_m\otimes{\rm id}_m-T\right]=-\sum_{k=1}^\infty \frac{1}{k}{T^k}
\end{equation}
as soon as the spectral radius of $T$ is less than $1$ (see (6.5.11) in \cite{HJ}), one proves the following proposition.

\begin{prop}\label{invertibility2} 
	For any $z_1,z_2\in \C\setminus \R$, the spectrum of the operator $$T_{\{z_1e_{11}-\gamma_0,z_2e_{11}-\gamma_0\}}:M_m(\mathbb{C})\otimes M_m(\mathbb{C})\to M_m(\mathbb{C}) \otimes M_m(\mathbb{C})$$ 
	is included in the open disk of radius $\sigma^{-2}$.
\end{prop}

Denote by $\mathcal H(\C \setminus \R)$ the space of complex analytic
functions on $\C \setminus \R$, endowed with the uniform topology on compact sets. The space $\mathcal H(\C \setminus \R)$ is equipped with the (topological) Borel
$\sigma$-field $\mathcal B(\mathcal H(\C \setminus \R))$. \\We are now in position to state our central limit theorem.
\begin{thm}\label{cvfdim}
Let $X_N$ be the random matrix introduced in Section \ref{Model}. For any $z\in \C \setminus \R$, set 
$$\xi_N(z):=\Tr\left((zI_N-X_N)^{-1}\right) -\mathbb{E}\left[\Tr\left((zI_N-X_N)^{-1}\right)\right].$$ The sequence of $\mathcal H(\C \setminus \R)$-valued random variables $(\xi_N)_{N\in \mathbb{N}}$ converges in  distribution to a centred complex Gaussian process $\{\mathcal G(z), z\in \mathbb{C}\setminus\mathbb{R}\}$ determined by $\overline{\mathcal G(z)}=\mathcal G(\bar z)$ and 
\begin{align*}
\mathbb{E}\left( \mathcal G(z_1) \mathcal G(z_2)\right)=\Gamma(z_1,z_2) & := \frac{\partial^2}{\partial z_1\partial z_2}\gamma(z_1,z_2),\quad z_1,z_2\in \mathbb{C}\setminus \mathbb{R},
\end{align*}
\begin{align*}
\gamma(z_1,z_2)& =-\Tr\otimes\Tr \left\{ \log \left[\mathrm{id}_m\otimes \mathrm{id}_m-\sigma^2T_{\{z_1e_{11}-\gamma_0,z_2e_{11}-\gamma_0\}}\right]  (I_m\otimes I_m)\right\}\\
& \quad -\Tr\otimes\Tr \left\{ \log \left[\mathrm{id}_m\otimes \mathrm{id}_m -\theta T_{\{z_1e_{11}-\gamma_0,z_2e_{11}-\gamma_0\}}\right]  (I_m\otimes I_m)\right\}\\
& \quad +(\tilde{\sigma}^2-\sigma^2-\theta) \Tr\otimes \Tr\{T_{\{z_1e_{11}-\gamma_0,z_2e_{11}-\gamma_0\}}(I_m\otimes I_m)\}\\
&\quad +\kappa/2\Tr\otimes \Tr\{T_{\{z_1e_{11}-\gamma_0,z_2e_{11}-\gamma_0\}}^2(I_m\otimes I_m)\}\end{align*}
\end{thm}

\section{Review of background}\label{sec:preliminary_results}

In this section, we gather properties which will be used several times in the sequel.

\subsection{Generalized resolvent} \label{sectionresolvent}

Let $\mathcal{M}$ be a unital $C^*$-algebra and $\mathcal{B}$ be a unital  $C^*$-subalgebra of $\mathcal{M}$. 
For $A\in\mathcal M$, we denote by $\Re A=(A+A^*)/2$ and $\Im A=(A-A^*)/2i$ the real and imaginary parts of $A$, so  $A=\Re A+i\Im A$. 
For a selfadjoint operator $A\in\mathcal M$, we write $A\ge 0$ if the spectrum  of $A$ is contained in $[0,+\infty)$ and $A>0$ if the spectrum  of $A$ is contained in $(0,+\infty)$. The operator upper half-plane of $\mathcal B$ is the set $\mathbb H^+(\mathcal B)=\{b\in\mathcal B\colon\Im b>0\}$. 

The generalized resolvent of an element $A\in \mathcal{M}$ in this context is the analytic map $R$ defined on the open subset $\{b\in \mathcal{B}, b-A \text{ is invertible in } \mathcal{M}\}$ of $\mathcal{B}$ by $R(b):=(b-A)^{-1}$. Note that the generalized resolvent of a selfadjoint element is in particular defined on the operator upper half-plane $\mathbb{H}^+(\mathcal{B})$ and satisfies 
\begin{equation}\label{majim}
\|R(b)\|\leq \|(\mathcal{I}b)^{-1}\|,\quad b\in \mathbb{H}^+(\mathcal{B}).
\end{equation}
In particular, when $\mathcal{M}=M_n({\mathbb{C}})$ and $\mathcal{B}=\mathbb{C}I_n\simeq \mathbb{C}$, the resolvent $R:z\mapsto (zI_n-A)^{-1}$ of a $n\times n$ Hermitian matrix with complex entries $A\in M_n(\mathbb{C})$ is defined on $\mathbb{C}\setminus \mathbb{R}$ and satisfies 
\begin{equation}\label{resolventbound}
\| R(z)\| \leq |\mathcal{I}z|^{-1},\quad z\in \mathbb{C}\setminus \mathbb{R}.
\end{equation}
The following Lemma is elementary but useful. 
\begin{lem}[Resolvent identity]\label{lem_resolvent_identity}
Let $A_1$ and $A_2$ be elements of $\mathcal{M}$ and denote by $R_1$ and $R_2$ their respective resolvents. Then, for all $b_1,b_2$ in the respective domains of $R_1$ and $R_2$, 
\[R_1(b_1)-R_2(b_2)=R_1(b_1)\big(b_2-b_1+A_1-A_2\big)R_2(b_2).\]
\end{lem}

\subsection{Linearization}\label{sec:lin}

In this Section, we collect a few properties of linearizations of polynomials in noncommuting indeterminates introduced above.

\begin{lem}\label{inversible}
Let  $P=P^*\in\mathbb{C}\langle t_1,\ldots,t_n\rangle$ and let
$L_P \in M_m(\mathbb{C})\otimes \mathbb{C}\langle t_1,\ldots,t_n\rangle$
be a linearization of P with the properties outlined above. 
Let $y = (y_1,\ldots, y_n)$  be a n-tuple of selfadjoint operators in a unital $C^*$-algebra ${\mathcal A}$. Then, for any $z\in \mathbb{C}$,  $ze_{11}\otimes 1_{\mathcal A}-L_P(y)$ is invertible if and only if $z 1_{\mathcal A}-P(y)$ is invertible and  we have \begin{equation}\label{coin}
\left(ze_{11}\otimes 1_{{\mathcal A}}-L_P(y)\right)^{-1}=\begin{pmatrix}
\left(z1_{\mathcal A}-P(y)\right)^{-1} & \star\\
\star & \star \end{pmatrix}.
\end{equation}
\end{lem}

Beyond the equivalence described above, we will use the following bound.

\begin{lem}\label{alta lemma}\cite{BBC} 
Let $z\in \mathbb C$ be such that $z1_{\mathcal A}-P(y)$ is invertible. There exist two polynomials $Q_1$ and $Q_2$ in $n$ commuting indeterminates, depending only on $L_P$, such that 
$$\left\|(ze_{11}\otimes 1_{\mathcal A}-L_P(y))^{-1}\right\| \leq Q_1\left(\|y_1\|,\dots,\|y_n\|\right) \left\|(z1_{\mathcal A}-P(y))^{-1}\right\|+ Q_2\left(\|y_1\|,\dots,\|y_n\|\right).$$
\end{lem}

\subsection{Operator-valued free probability theory}\label{opfree}

There exists an extension of free probability theory, operator-valued free probability theory, which still shares the basic properties of free probability but is much more powerful because of its wider domain of applicability. The concept of freeness with amalgamation and some of the relevant analytic transforms were introduced by Voiculescu in \cite{V1995}.

\begin{definition}
Let $\mathcal{M}$ be a unital complex algebra and $\mathcal B\subset \mathcal{M}$ be a unital subalgebra. A linear map $E: \mathcal{M} \rightarrow \mathcal B$ is a conditional expectation if $E(b)=b$ for all $b\in \mathcal B$ and $E(b_1Ab_2)=b_1E( A)b_2$ for all $A\in \mathcal{M}$ and $b_1,b_2\in \mathcal B$. Then $\left(\mathcal{M},E\right)$ is called a $\mathcal{B}$-valued probability space. If in addition $\mathcal{M}$ is a $C^*$-algebra (von Neumann algebra), $\mathcal B$ is a $C^*$-subalgebra (von Neumann subalgebra) of $\mathcal{M}$, then we have a $\mathcal{B}$-valued $C^*$-probability space ($\mathcal{B}$-valued $W^*$-probability space).
\end{definition}

In our applications, the algebra $\mathcal B$ is (isomorphic to) $M_m(\mathbb C)$ for some $m\in\mathbb N$. More precisely, let $\mathcal A$ be a von Neumann algebra endowed with a normal faithful tracial state $\tau$, and let $m\in\mathbb N$. Then $M_m(\mathbb C)$ can be identified with the subalgebra $M_m(\mathbb C)\otimes1_{\mathcal{A}}$ of $M_m({\mathcal A})=M_m(\mathbb C)\otimes \mathcal A$. Moreover, the von Neumann algebra $M_m({\mathcal A})$ is endowed with the normal faithful tracial state $
m^{-1}\Tr\otimes\tau$, and $\mathrm{id}_{m}\otimes\tau$ is the trace-preserving conditional expectation from $M_m({\mathcal A})$ to $M_m(\mathbb C)$. In other words, $(M_m({\mathcal A}),\mathrm{id}_{m}\otimes\tau)$ is a $M_m(\mathbb C)$-valued $W^*$-probability space.
As mentioned in Section 3 of \cite{BBC}, the distributional limits of random matrices as considered 
in our models are realized in II${}_1$-factors, with respect to their unique normal faithful tracial states, so that
there is no loss of generality in assuming $(\mathcal{A},\tau)$ to be a von Neumann algebra endowed with a normal faithful tracial state (also called $W^*$-probability space). 


\begin{definition}
Let $\left( {\mathcal M}, E\right)$ be a $\mathcal{B}$-valued probability space. The ${ \mathcal{B}}$-valued distribution of a noncommutative random variable $A\in {\mathcal M}$ is given by all ${ \mathcal{B}}$-valued moments $E(Ab_1Ab_2\cdots b_{n-1}A)$, $n \in \mathbb{N}, b_1, \ldots, b_{n-1}\in \mathcal{B}$.
\end{definition}

If $s$ is a (scalar-valued) semicircular element with mean $0$ and variance $\sigma^2$ in some $W^*$-probability space $(\mathcal A, \tau)$, then, for any Hermitian matrix $\gamma\in M_m(\C)$,  $\gamma\otimes s\in M_m({\mathcal A})$ is a $M_m(\mathbb C)$-valued semicircular element in the $M_m(\mathbb C)$-valued probability space $(M_m({\mathcal A}),\mathrm{id}_{m}\otimes\tau)$, in the sense of \cite{SMem}. The ${\mathcal{B}}$-valued distribution of a general centred $\mathcal{B}$-valued semicircular element $S\in \mathcal{M}$ is uniquely determined by its operator-valued variance $\eta\colon b\mapsto E(SbS)$ ; a characterization in terms of moments and cumulants via $\eta$ is provided by Speicher in \cite{SMem}. The operator-valued variance of the $M_m(\mathbb C)$-valued semicircular element $\gamma\otimes s$ is $\eta\colon b\mapsto \sigma^2 \gamma b\gamma$.

As in scalar-valued free probability, one defines \cite{V1995} {\em freeness with amalgamation}
over ${ \mathcal{B}}$ via an algebraic relation similar to {freeness}, but involving $E$ and noncommutative polynomials with coefficients in ${ \mathcal{B}}$.

\begin{definition}
Let $\left( {\mathcal M}, E\right)$ be a $\mathcal{B}$-valued probability space. Let $(\mathcal{A}_i)_{i\in I}$ be a family  of subalgebras with ${\mathcal B}\subset \mathcal{A}_i$ for all $i\in I$. The subalgebras $(\mathcal{A}_i)_{i\in I}$ are free with respect to E or free with amalgamation over ${\mathcal B}$ if $E(A_1\cdots A_n)=0$ whenever $A_j \in \mathcal{A}_{i_j}$, $i_j \in I$, $i_1 \neq i_2 \neq \cdots \neq i_n$ and $E(A_j)=0$, for all $j$.\\
Noncommutative random variables in $\mathcal M$ or subsets of $\mathcal M$ are free with amalgamation over $\mathcal{B}$ if the algebras generated by $\mathcal{B}$ and the variables or the algebras generated by $\mathcal{B}$ and the subsets, respectively, are so.
\end{definition}

The following result from \cite{NSS} explains why the  particular case  ${\mathcal B}=M_m(\mathbb C)$, $\mathcal M=M_m({\mathcal A})$, $E={\rm id}_{m}\otimes \tau$, where $(\mathcal A, \tau)$ is a $W^*$-probability space,
 is relevant in our work using linearizations of polynomials.

\begin{prop}\label{fr}
Let $(\mathcal{A},\varphi)$ be a noncommutative probability space, let $a_1,\ldots, a_n\in (\mathcal{A},\varphi)$ be freely independent noncommutative random variables and let $m\in \mathbb{N}$. Then the map ${\rm id}_m\otimes\varphi \colon M_m(\mathcal A)\to M_m(\mathbb C)$ is a conditional expectation, and $\alpha_1\otimes a_1, \ldots, \alpha_n \otimes a_n$  are free with amalgamation over $M_m(\mathbb C)$ for any $\alpha_1,\ldots ,\alpha_n\in M_m(\mathbb C)$.
\end{prop} 

The analytic subordination phenomenon for free convolutions was first noted by Voiculescu and Biane in the scalar case and later approached from an abstract coalgebra point of view by Voiculescu in \cite{V2000} and this approach extends the results to the operator-valued case. In \cite{BMS}, Belinschi, Mai and Speicher  develop an analytic theory. 

\begin{prop}\label{subop}\cite{V2000},\cite{BMS}(see Theorem 5 p 259 \cite{MS}) 
Let $\left(\mathcal M,E\right)$ be a $\mathcal{B}$-valued $C^*$-probability space. Let $A_1,A_2\in\mathcal M$ be selfadjoint noncommutative random variables which are {free with amalgamation over ${\mathcal B}$}. \\
There exists a unique pair of Fr\'echet  analytic maps {$\omega_1, \omega_2\colon\mathbb H^+({\mathcal{B}})\to\mathbb H^+({\mathcal B})$}
such that, for all $b\in\mathbb H^+({\mathcal B})$,
\begin{enumerate}
\item $\Im \omega_{j}(b)\ge\Im  b,\ \omega_j(b^*)=\omega_j(b)^*, \; j=1,2$;
\item $E[(b-(A_1+A_2))^{-1}]=E[(\omega_1(b)-A_1)^{-1}]=E[(\omega_2(b)-A_2)^{-1}]$;
\item $E[(b-(A_1+A_2))^{-1}]^{-1}+b=\omega_1(b)+\omega_2(b)$.
\end{enumerate}
Moreover, if $b \in \mathbb H^+({ \mathcal{B}})$, then $ \omega_1(b)$ is the unique fixed point of the map $f_b:\mathbb H^+({\mathcal B}) \rightarrow \mathbb H^+({\mathcal B})$ defined by $f_b(\omega)=h_{A_2}(h_{A_1}(\omega)+b)+b$,
where $h_{A_i}(b)=E[(b-A_i)^{-1}]^{-1}-b$
and $\omega_1(b)= \lim_{k\rightarrow +\infty} f_b^{\circ k}(w)$, for any $\omega \in\mathbb H^+({\mathcal B})$.
\end{prop}

Moreover, if $\mathcal M$ is a $W^*$-probability space and $\mathcal{B}\subset \mathcal{D}\subset \mathcal{M}$ are  
von Neumann subalgebras (hence unital by definition) such that $A_2\in \mathcal{D}$ and $\mathcal{D}$ is free with amalgamation over $\mathcal{B}$ from $A_1$ (with respect to the trace-preserving conditional expectation from $\mathcal M$ onto $\mathcal B$), then the following strengthened result holds:
\begin{equation*}
E_{\mathcal D}[\left(b-(A_1+A_2)\right)^{-1}]=\left(\omega_2(b)-A_2\right)^{-1},\quad b\in\mathbb H^+({\mathcal B}),
\end{equation*}
where $E_{\mathcal{D}}$ is the trace-preserving conditional expectation from $\mathcal M$ onto $\mathcal D$.

If, in Proposition \ref{subop}, $A_1$ is a centred $\mathcal{B}$-valued semicircular element with operator-valued variance $\eta$, the subordination function $\omega_2$ has a more explicit form (see \cite[Chapter 9]{MS} and the end of the proof of Theorem 8.3 in \cite{ABFN}): 
\begin{equation}\label{six}
\omega_2(b)=b-\eta(E[\left(b-(A_1+A_2)\right)^{-1}]),\quad b\in \mathbb{H}^+(\mathcal{B}).
\end{equation} 
It follows that $\omega_2$ may be analytically extended to the open subset $\{b\in \mathcal{B}, b-(A_1+A_2) \text{ is invertible in }\mathcal{A}\}$. Moreover, for $b$ in the connected component of $\{b\in \mathcal{B}, b-(A_1+A_2) \text{ is invertible in }\mathcal{A}\}$ containing $\mathbb H^+(\mathcal{B})$, 
\begin{equation}\label{subforte}
E_{\mathcal D}[\left(b-(A_1+A_2)\right)^{-1}]=\left(\omega_2(b)-A_2\right)^{-1}
\end{equation}
holds. Note also that $\omega_2(b)$ satisfies the fixed point equation
\begin{equation}\label{fixed}
\omega_2(b)=b-\eta(E[\left(\omega_2(b)-A_2\right)^{-1}]).
\end{equation}

\section{Definition of the limiting object: proof of Proposition \ref{invertibility2}}\label{defobjetlimite}

Our strategy of proof for Proposition \ref{invertibility2} is the following : in Section \ref{equalityspectra}, one proves that the operator $\sigma^2T_{\{z_1e_{11}-\gamma_0,z_2e_{11}-\gamma_0\}}:M_m(\mathbb C)\otimes M_m(\mathbb C)\to M_m(\mathbb C)\otimes M_m(\mathbb C)$ has the same eigenvalues as the operator $u_{z_1e_{11}-\gamma_0,z_2e_{11}-\gamma_0}$, where $u_{\beta_1,\beta_2}\colon M_m(\mathbb C)\to M_m(\mathbb C)$ is defined by 
$$b\mapsto\sigma^2({\rm id}_m\otimes\tau)\left[(\omega(\beta_1)\otimes1_{\mathcal{A}}-\gamma_2\otimes d )^{-1}((\gamma_1b\gamma_1)\otimes1_{\mathcal{A}})(\omega(\beta_2)\otimes1_{\mathcal{A}}-\gamma_2\otimes d)^{-1}\right].$$
The key point  of this section is then to prove the following Proposition \ref{welldefined3}.

\begin{prop}\label{welldefined3}
Let $s,d\in (\mathcal A,\tau)$ be freely independent selfadjoint noncommutative random variables in a $W^*$-probability space $(\mathcal A,\tau)$. Assume that $s$ is a semicircular element with mean $0$ and variance $\sigma^2$. Let also $\gamma_0,\gamma_1,\gamma_2$ be non-zero 
selfadjoint matrices in $M_{m}(\mathcal \C)$ such that the lower right $(m-1)\times(m-1)$ corner of $\gamma_0\otimes 1_{\mathcal{A}}+\gamma_1\otimes s+\gamma_2\otimes d$ is invertible in $M_{m-1}(\mathcal A)$. For any $(z_1, z_2)\in \mathbb C^2$ such that $(z_je_{11}-\gamma_0)\otimes 1_{\mathcal{A}}-\gamma_1\otimes s-\gamma_2\otimes d$ is invertible in $M_m(\mathcal A),j=1,2$, the spectrum of the linear operator $u_{z_1e_{11}-\gamma_0, z_2e_{11}-\gamma_0}\colon M_m(\mathbb C)\to M_m(\mathbb C)$ defined by 
$$b\mapsto\sigma^2({\rm id}_m\otimes\tau)\left[(\omega(z_1e_{11}-\gamma_0)\otimes1_{\mathcal{A}}-\gamma_2\otimes d )^{-1}((\gamma_1b\gamma_1)\otimes1_{\mathcal{A}})(\omega(z_2e_{11}-\gamma_0)\otimes1_{\mathcal{A}}-\gamma_2\otimes d)^{-1}\right]$$
is included in the open unit disk. Recall that $$\omega(\beta)=\beta-\sigma^2\gamma_1({\rm id}_m\otimes\tau)\left[\left(\beta\otimes1_{\mathcal{A}}-\gamma_1\otimes s-\gamma_2\otimes d\right)^{-1}\right]\gamma_1.$$
\end{prop}


Note that the conclusion of Proposition \ref{welldefined3} holds for any $(z_1, z_2)\in (\mathbb C\setminus \mathbb{R})^2$ as explained at the beginning of the proof.

The first step in the proof of Proposition \ref{welldefined3}, detailed in Section \ref{firststep}, consists in proving that the spectrum of $u_{\beta_1,\beta_2}$ is included in the open unit disk when $\beta_1, \beta_2\in M_m(\C)$ have positive-definite or negative-definite imaginary parts. Then,
in a second step in Section \ref{secondstep}, by using the maximum principle for plurisubharmonic functions, we deduce from the first step the result for $\beta_1=z_1e_{11}-\gamma_0$, $\beta_2=z_2e_{11}-\gamma_0$, as required.

\subsection{An equality of spectra}\label{equalityspectra}

The aim of this section is to prove that the operators $\sigma^2T_{\{z_1e_{11}-\gamma_0,z_2e_{11}-\gamma_0\}}$ and $u_{z_1e_{11}-\gamma_0,z_2e_{11}-\gamma_0}$ have the same eigenvalues (not counting multiplicities).

\begin{prop}
The operators $\sigma^2T_{\{z_1e_{11}-\gamma_0,z_2e_{11}-\gamma_0\}}:M_m(\mathbb C)\otimes M_m(\mathbb C)\to M_m(\mathbb C)\otimes M_m(\mathbb C)$ and $u_{z_1e_{11}-\gamma_0,z_2e_{11}-\gamma_0}:M_m(\mathbb C)\to M_m(\mathbb C)$ have the same eigenvalues.
\end{prop}

\begin{proof}
We will use the following identifications between algebras $M_n(\mathbb C)\otimes M_n(\mathbb C)$ and $\mathcal{L}(M_n(\mathbb C))$ : define an isomorphism of algebras  $M_n(\mathbb C)\otimes M_n(\mathbb C)\to \mathcal{L}(M_n(\mathbb C))$ by requiring that the image of $A\otimes B$ is the operator $X\mapsto AXB^T$ for any $A,B\in M_n(\mathbb C)$. 
Using these identifications, one may observe that the following equalities hold in $M_m(\mathbb{C})^{\otimes 4}$ : 
\begin{align*}
 \sigma^2 & T_{\{z_1e_{11}-\gamma_0,z_2e_{11}-\gamma_0\}}\\
&=\sigma^2\int_{\mathbb{R}}(\omega(z_1e_{11}-\gamma_0)-t\gamma_2)^{-1}\gamma_1\otimes I_m\otimes I_m\otimes  (\omega(z_2e_{11}-\gamma_0)-t\gamma_2)^{-1}\gamma_1 d\nu(t)\\
&=F\otimes F(I_m \otimes \sigma^2\int_{\mathbb{R}}(\omega(z_1e_{11}-\gamma_0)-t\gamma_2)^{-1}\gamma_1\otimes (\omega(z_2e_{11}-\gamma_0)-t\gamma_2)^{-1}\gamma_1d\nu(t)\otimes I_m)\\
&=F\otimes F(I_m \otimes u_{z_1e_{11}-\gamma_0,z_2e_{11}-\gamma_0}\otimes I_m),
\end{align*}
where $F:M_m(\mathbb{C})\otimes M_m(\mathbb{C})\to M_m(\mathbb{C})\otimes M_m(\mathbb{C})$ is the automorphism of the algebra $M_m(\mathbb{C})\otimes M_m(\mathbb{C})$ determined by $F(A\otimes B)=B\otimes A$. It follows that $\sigma^2T_{\{z_1e_{11}-\gamma_0,z_2e_{11}-\gamma_0\}}:M_m(\mathbb C)\otimes M_m(\mathbb C)\to M_m(\mathbb C)\otimes M_m(\mathbb C)$ on the one hand and $u_{z_1e_{11}-\gamma_0,z_2e_{11}-\gamma_0}:M_m(\mathbb C)\to M_m(\mathbb C)$ on the second hand have the same minimal polynomial, hence the same eigenvalues (but not with the same multiplicities though).
\end{proof}

\subsection{First step  in the proof of Proposition \ref{welldefined3}} \label{firststep}
Consider an arbitrary $C^*$-algebra $\mathcal B$, a completely positive map $\eta\colon\mathcal B\to\mathcal B$, and a centred $\mathcal B$-valued semicircular element $S$ with operator-valued variance $\eta$. {Recall that a completely positive map $\eta$ is automatically completely bounded: $\|\eta\|_{\rm cb}:=\sup_{m\in\mathbb N}\|{\rm id}_m
\otimes\eta\|<\infty$.}
Assume that $D=D^*$ is free from $S$ with amalgamation over $\mathcal B$ with respect to the conditional expectation $E$. As in \eqref{six}, we may write 
$$
\omega(b)=b-\eta\left(E\left[(b-S-D)^{-1}\right]\right)=b-\eta\left(E\left[(\omega(b)-D)^{-1}\right]\right).
$$
\begin{prop}\label{welldefined1}
For any $\beta_1, \beta_2$ in $\mathcal B$ such that either $\{\beta_1,\beta_2\in\mathbb H^+(\mathcal B)\}$, or $\{-\beta_1,-\beta_2\in\mathbb H^+(\mathcal B)\}$, or $\{\beta_1,-\beta_2\in\mathbb H^+(\mathcal B)\}$, or $\{-\beta_1,\beta_2\in\mathbb H^+(\mathcal B)\}$, 
the spectrum of the linear operator $u_{\beta_1,\beta_2}$ on $\mathcal B$ defined by 
$$
v\mapsto
E\left[(\omega(\beta_1)-D)^{-1}\eta(v)(\omega(\beta_2)-D)^{-1}\right]$$
is included in the open unit disk. 
\end{prop}

The context that will be of interest in our paper corresponds to $\mathcal B=M_m(\mathbb C)$, $S=\gamma_1\otimes s, D=\gamma_2\otimes d$ for an arbitrary selfadjoint noncommutative random variable $d$, freely independent from the semicircular element $s$ with mean $0$ and variance $\sigma^2$ in some $W^*$-probability space $(\mathcal A,\tau)$.
In that case $\eta:b\mapsto \sigma^2\gamma_1b\gamma_1$ and $E={\rm id}_m\otimes\tau$. The proof of Proposition \ref{welldefined1} given above  would 
benefit little in terms of simplification from the assumption that $\mathcal B$ is finite dimensional, so there is no reason not to give it in full generality.
The idea
of the proof is to make use of the contractivity of analytic self-maps on hyperbolic domains.
\begin{proof}
The cases $\{\beta_1,\beta_2\in\mathbb H^+(\mathcal B)\}$ and $\{-\beta_1,-\beta_2\in\mathbb H^+(\mathcal B)\}$ are covered in \cite[Proposition 4.1]{CAOT}.
Thus, we will focus exclusively on the case $\{-\beta_1,\beta_2\in\mathbb H^+(\mathcal B)\}$, the case $\{\beta_1,-\beta_2\in\mathbb H^+(\mathcal B)\}$ being 
identical. However, the reader will find that the methods we employ in our proof below cover the other two cases with virtually no modification.

We assume without loss of generality that $\mathcal I\beta_1<0,\mathcal I\beta_2>0$. Consider the convex set 
$$
\mathcal H=\left\{\begin{bmatrix} w_1 & v \\ 0 & w_2 \end{bmatrix}\colon\begin{bmatrix} -w_1 & v \\ 0 & w_2 \end{bmatrix}\in\mathbb H^+(M_2(\mathcal B))\right\}.
$$
Note that for any  $a=a^*$, $b$ and $c=c^*$ in $\mathcal B$,
$$\begin{pmatrix} a&b\\b^*&c\end{pmatrix}=\begin{pmatrix} 1&bc^{-1}\\0&1\end{pmatrix}
\begin{pmatrix} (a-bc^{-1}b^*)&0\\0&c\end{pmatrix} \begin{pmatrix} 1&0\\c^{-1}b^*&1\end{pmatrix}$$
holds. Thus, $\begin{pmatrix} a&b\\b^*&c\end{pmatrix}>0$ if and only if $c>0$ and $a-bc^{-1}b^*>0$, that is if and only if $c>0$, $a>0$,  and $a^{-1/2}bc^{-1}b^*a^{-1/2}<1$.\\
It is interesting to note that $\Im\begin{bmatrix} w_1 & v \\ 0 & w_2 \end{bmatrix}=\begin{bmatrix}\Im w_1 &\frac{v}{2i}\\\left(\frac{v}{2i}\right)^*&\Im w_2 \end{bmatrix}$
is invertible in $M_2(\mathcal B)$. Indeed, the two diagonal entries are assumed to be invertible, and the Schur complement formula tells us that as long as we can guarantee
that $\Im w_1-\frac{v}{2i}(\Im w_2)^{-1}\left(\frac{v}{2i}\right)^*$ is invertible, we are done. But $\Im w_1<0$, so that $0>\Im w_1\ge\Im w_1-\frac14v(\Im w_2)^{-1}v^*$
makes $\Im w_1-\frac{v}{2i}(\Im w_2)^{-1}\left(\frac{v}{2i}\right)^*=\Im w_1-\frac14v(\Im w_2)^{-1}v^*$ invertible {\em regardless of the size of} $v$! Trivially, so is
any element $\mathfrak w\in\mathcal H$.
The maps
$$
({\rm id}_2\otimes E)\!\!\left[\!\left(\!\begin{bmatrix} w_1 & v \\ 0 & w_2 \end{bmatrix}\!-\!\begin{bmatrix} S\!+\!D & 0 \\ 0 & S\!+\!D \end{bmatrix}\!\right)^{-1}\!\right]\!=\!
\begin{bmatrix}\!E\left[(w_1\!-\!S\!-\!D)^{-1}\right] & \!\!- E\left[(w_1\!-\!S\!-\!D)^{-1}v(w_2\!-\!S\!-\!D)^{-1}\right] \\ 0 & E\left[(w_2\!-\!S\!-\!D)^{-1}\right] \end{bmatrix}
$$
and 
$$
({\rm id}_2\otimes E)\!\!\left[\!\left(\!\begin{bmatrix} w_1 & v \\ 0 & w_2 \end{bmatrix}\!-\!\begin{bmatrix} D & 0 \\ 0 & D \end{bmatrix}\!\right)^{-1}\!\right]\!=\!
\begin{bmatrix}\!E\left[(w_1\!-\!D)^{-1}\right] & \!- E\left[(w_1\!-\!D)^{-1}v(w_2\!-\!D)^{-1}\right] \\ 0 & E\left[(w_2\!-\!D)^{-1}\right] \end{bmatrix}
$$
are well-defined (this is a trivial observation) and moreover are the unique extensions through the set $\{b\in M_2(\mathcal B)\colon\|b^{-1}\|<\frac{1}{\|S\|+\|D\|}\}$
of the usual operator-valued Cauchy transforms defined on $\mathbb H^\pm(M_2(\mathcal B))$. Of course, $\mathcal H$ is not open in $M_2(\mathcal B)$. However, 
as the set of invertible bounded operators on a Hilbert space is open, it is clear that for each $\begin{bmatrix} w_1 & v \\ 0 & w_2 \end{bmatrix}\in\mathcal H$  one finds a
norm-neighborhood $V$ of this point in $M_2(\mathcal B)$ (and depending on this point) such that both $\mathfrak w-\begin{bmatrix} S\!+\!D & 0 \\ 0 & S\!+\!D \end{bmatrix}$
and $\mathfrak w-\begin{bmatrix} D & 0 \\ 0 & D \end{bmatrix}$ are invertible for all $\mathfrak w\in V$. Thus, the above defined extensions are indeed unique by the 
identity principle for analytic functions. 

While not open in $M_2(\mathcal B)$, the space $\mathcal H$ is nevertheless an analytic space, open in the complex algebra of upper triangular matrices in $M_2(\mathcal B)$,
so that we may define analytic functions on it and apply analytic function theory results to them. To begin with, observe that both 
$$
({\rm id}_2\otimes E)\!\!\left[\!\left(\!\begin{bmatrix} w_1 & v \\ 0 & w_2 \end{bmatrix}\!-\!\begin{bmatrix} S\!+\!D & 0 \\ 0 & S\!+\!D \end{bmatrix}\!\right)^{-1}\!\right],\quad
({\rm id}_2\otimes E)\!\!\left[\!\left(\!\begin{bmatrix} w_1 & v \\ 0 & w_2 \end{bmatrix}\!-\!\begin{bmatrix} D & 0 \\ 0 & D \end{bmatrix}\!\right)^{-1}\!\right]
$$
send $\mathcal H$ to $-\mathcal H$. Indeed, for any selfadjoint noncommutative random variable $Y$ 
\begin{align*}({\rm id}_2\otimes E)\!\!\left[\!\left(\!
\begin{bmatrix} w_1 & v \\ 0 & w_2 \end{bmatrix}\!-\!\begin{bmatrix} Y & 0 \\ 0 & Y \end{bmatrix}\!\right)^{-1}\!\right]&=
({\rm id}_2\otimes E)\!\!\begin{bmatrix} (w_1-Y)^{-1} & -(w_1-Y)^{-1}v(w_2-Y)^{-1} \\ 0 & (w_2-Y)^{-1} \end{bmatrix}\!\\&=
\begin{bmatrix} E[(w_1-Y)^{-1}] & -E[(w_1-Y)^{-1}v(w_2-Y)^{-1}] \\ 0 & E[(w_2-Y)^{-1}] \end{bmatrix}.
\end{align*} 
Now $\Im w_1<0\implies\Im E[(w_1-Y)^{-1}]>0,$ and $
\Im w_2>0\implies\Im E[(w_2-Y)^{-1}]<0$. On the other hand, 
\begin{align*}
\Im ({\rm id}_2\otimes E)\left[\left(\begin{bmatrix} -w_1 & v \\ 0 & w_2 \end{bmatrix}-\begin{bmatrix} -Y & 0 \\ 0 & Y \end{bmatrix}\right)^{-1}\!\right]&=\Im
({\rm id}_2\otimes E)\!\!\begin{bmatrix} -(w_1-Y)^{-1} & (w_1-Y)^{-1}v(w_2-Y)^{-1} \\ 0 & (w_2-Y)^{-1} \end{bmatrix}\!\\&=
\Im\begin{bmatrix} -E[(w_1-Y)^{-1}] & E[(w_1-Y)^{-1}v(w_2-Y)^{-1}] \\ 0 & E[(w_2-Y)^{-1}] \end{bmatrix}\\
& <0.
\end{align*}
This shows that 
$({\rm id}_2\otimes E)\!\!\left[\!\left(\!\begin{bmatrix} w_1 & v \\ 0 & w_2 \end{bmatrix}\!-\!\begin{bmatrix} Y & 0 \\ 0 & Y \end{bmatrix}\!\right)^{-1}\!\right]\in
-\mathcal H$. \\In addition, $\eta$ is completely positive, so $({\rm id}_2\otimes\eta)(\mathcal H)\subseteq\mathcal H$. Thus, the map
$$
f\colon\mathcal H\times\mathcal H\to\mathcal H,\quad
f_{\boldsymbol{\beta}}(\mathfrak w)={\boldsymbol{\beta}}
-({\rm id}_2\otimes \eta)\circ({\rm id}_2\otimes E)\left[\left(\mathfrak w-\begin{bmatrix} D & 0 \\ 0 & D \end{bmatrix}\!\right)^{-1}\right]
$$
is a well-defined map, and for each ${\boldsymbol{\beta}}\in\mathcal H,$ $f_{\boldsymbol{\beta}}$ is an analytic self-map of $\mathcal H$. 

Considering the level one relation from \eqref{six} for our given $\beta_1,\beta_2$, we automatically have
\begin{align*}
({\rm id}_2\otimes E)\!\!\left[\!\left(\!\begin{bmatrix} \beta_1 & 0 \\0 & \beta_2 \end{bmatrix}\!-\!\begin{bmatrix} S\!+\!D & 0 \\ 0 & S\!+\!D \end{bmatrix}\!\right)^{-1}\!\right]
&= \begin{bmatrix} E\left[(\beta_1-S-D)^{-1}\right] & 0 \\ 0 & E\left[(\beta_2-S-D)^{-1}\right] \end{bmatrix}\\
&=\begin{bmatrix} E\left[(\omega(\beta_1)-D)^{-1}\right] & 0 \\ 0 & E\left[(\omega(\beta_2)-D)^{-1}\right] \end{bmatrix}\\
&= ({\rm id}_2\otimes E)\!\!\left[\!\left(\!\begin{bmatrix}\omega(\beta_1)&0\\0&\omega(\beta_2)\end{bmatrix}\!-\!\begin{bmatrix} D & 0\\0 & D \end{bmatrix}\!\right)^{-1}\!\right].
\end{align*}
This guarantees that, for ${\boldsymbol{\beta}}=\begin{bmatrix} \beta_1 & 0 \\0 & \beta_2 \end{bmatrix}$, we have
\begin{align*}
f_{\boldsymbol{\beta}}\left(\!\begin{bmatrix}\omega(\beta_1)&0\\0&\omega(\beta_2)\end{bmatrix}\!\right)
&= {\boldsymbol{\beta}}-\begin{bmatrix}(\eta\circ E)\left[(\omega(\beta_1)-D)^{-1}\right] & 0 \\ 0 & (\eta\circ E)\left[(\omega(\beta_2)-D)^{-1}\right] \end{bmatrix}\\
&=\begin{bmatrix} \beta_1-(\eta\circ E)\left[(\omega(\beta_1)-D)^{-1}\right] & 0 \\ 0 & \beta_2-(\eta\circ E)\left[(\omega(\beta_2)-D)^{-1}\right] \end{bmatrix}\\
&= \begin{bmatrix}\omega(\beta_1)&0\\0&\omega(\beta_2)\end{bmatrix}, 
\end{align*}
so that $\begin{bmatrix}\omega(\beta_1)&0\\0&\omega(\beta_2)\end{bmatrix}$ is a fixed point of $f_{\boldsymbol{\beta}}$. 
Let us limit ourselves now to the subset 
$$
\mathcal H_{\beta_1,\beta_2,R}=\left\{\mathfrak w=\begin{bmatrix} w_1 & v \\ 0 & w_2 \end{bmatrix}\in\mathcal H\colon\mathcal I
\begin{bmatrix} -w_1 & v \\ 0 & w_2 \end{bmatrix}>\frac12\mathcal I\begin{bmatrix} -\beta_1 & 0 \\ 0 & \beta_2 \end{bmatrix},\|\mathfrak w\|<R\right\}.
$$
It is clear that for $R>0$ sufficiently large, $\mathcal H_{\beta_1,\beta_2,R}$ is a nonempty open connected subset of $\mathcal H$ (in fact convex). {According to 
what precedes}, we have
\begin{align*}
\mathcal I\left(\!\begin{bmatrix}-\beta_1&\!0\\0&\!\beta_2\end{bmatrix}\!-\!
\begin{bmatrix}-(\eta\!\circ\!E)\!\left[(w_1\!-\!D)^{-1}\right]&\!\!-(\eta\!\circ\!E)\!\left[(w_1\!-\!D)^{-1}v(w_2\!-\!D)^{-1}\right]\\0&\!(\eta\!\circ\!E)\!\left[(w_2\!-\!D)^{-1}\right]
\end{bmatrix}\!\right)
& \geq \mathcal I\!\begin{bmatrix}-\beta_1&\!0\\0&\!\beta_2\end{bmatrix}\!\\
& > \frac12\!\mathcal I\!\begin{bmatrix}-\beta_1&\!0\\0&\!\beta_2\end{bmatrix}
\end{align*}
and 
\begin{align*}
\|f_{\boldsymbol\beta}(\mathfrak w)\|&=\left\|\begin{bmatrix} \beta_1 & 0 \\0 & \beta_2 \end{bmatrix}-
\begin{bmatrix}(\eta\!\circ\!E)\!\left[(w_1\!-\!D)^{-1}\right]&\!\!-(\eta\!\circ\!E)\!\left[(w_1\!-\!D)^{-1}v(w_2\!-\!D)^{-1}\right]\\0&\!(\eta\!\circ\!E)\!\left[(w_2\!-\!D)^{-1}\right]
\end{bmatrix}\right\|\\
&\le\left\|\begin{bmatrix} \beta_1 & 0 \\0 & \beta_2 \end{bmatrix}\right\|+\|\eta\|_{\rm cb}\left\|
\begin{bmatrix}(w_1\!-\!D)^{-1}&\!\!-\!(w_1\!-\!D)^{-1}v(w_2\!-\!D)^{-1}\\0&\!(w_2\!-\!D)^{-1}
\end{bmatrix}\right\|\\
&=\left\|\begin{bmatrix} \beta_1 & 0 \\0 & \beta_2 \end{bmatrix}\right\|+
\|\eta\|_{\rm cb}\left\|\begin{bmatrix}w_1-D&\!v\\0&\!w_2-D\end{bmatrix}^{-1}\right\|\\
&=\left\|\begin{bmatrix} \beta_1 & 0 \\0 & \beta_2 \end{bmatrix}\right\|+
\|\eta\|_{\rm cb}\left\|\left(\begin{bmatrix}w_1-D&\!v\\0&\!w_2-D\end{bmatrix}\begin{bmatrix} -1 & 0 & \\ 0 & 1 \end{bmatrix}\right)^{-1}\right\|\\
&=\left\|\begin{bmatrix} \beta_1 & 0 \\0 & \beta_2 \end{bmatrix}\right\|+
\|\eta\|_{\rm cb}\left\|\begin{bmatrix}-w_1+D&\!v\\0&\!w_2-D\end{bmatrix}^{-1}\right\|\\
&=\left\|\begin{bmatrix} \beta_1 & 0 \\0 & \beta_2 \end{bmatrix}\right\|+
\|\eta\|_{\rm cb}\left\|\left(\begin{bmatrix}-w_1&\!v\\0&\!w_2\end{bmatrix}-\begin{bmatrix} -D & 0 \\ 0 & D \end{bmatrix}\right)^{-1}\right\|\\
&\le\|\beta_1\|+\|\beta_2\|+\|\eta\|_{\rm cb}\left\|\left(\Im\begin{bmatrix}- w_1 & v\\ 0 & w_2\end{bmatrix}\right)^{-1}\right\|\\
&<\|\beta_1\|+\|\beta_2\|+2\|\eta\|_{\rm cb}\!\left\|\begin{bmatrix}-\Im\beta_1 & 0\\ 0 &\!\Im\beta_2\end{bmatrix}^{-1}\right\|\!\le\!\|\beta_1\|+\|\beta_2\|+2\|\eta\|_{\rm cb}
\!\left[\|(\Im\beta_1)^{-1}\|\!+\!\|(\Im\beta_2)^{-1}\|\right]\!.
\end{align*}
(We have used \eqref{majim} and the hypothesis on elements in $\mathcal H_{\beta_1,\beta_2,R}$.) Observe that the majorization of $\|f_{\boldsymbol\beta}(\mathfrak w)\|$
is independent of $R$ and $\mathfrak w$. Thus, if we choose $R=2\|\beta_1\|+\|\beta_2\|+2\|\eta\|_{\rm cb}
\!\left[\|(\Im\beta_1)^{-1}\|\!+\!\|(\Im\beta_2)^{-1}\|\right]$, then we are guaranteed that $f_{\boldsymbol\beta}(\mathcal H_{\beta_1,\beta_2,R})$
is at positive norm-distance from $\mathcal H\setminus\mathcal H_{\beta_1,\beta_2,R}$. Theorem \ref{Earle-Hamilton} guarantees that $f_{\boldsymbol\beta}$ has a unique
{\em attracting} fixed point in $\mathcal H$, which, unsurprisingly, belongs to $\mathcal H_{\beta_1,\beta_2,R}$, and that $f_{\boldsymbol\beta}$ is a strict contraction
in the hyperbolic metric on $\mathcal H$. Since $f_{\boldsymbol\beta}$ maps $\mathcal H_{\beta_1,\beta_2,R}$ strictly inside itself, of course $f_{\boldsymbol\beta}$
is a strict contraction in the hyperbolic metric of $\mathcal H_{\beta_1,\beta_2,R}$ itself, hence any hyperbolic ball around the fixed point
$\begin{bmatrix}\omega(\beta_1)&0\\0&\omega(\beta_2)\end{bmatrix}$ is mapped strictly inside itself. In particular, given an arbitrary finite-radius hyperbolic ball 
$B\subsetneq\mathcal H_{\beta_1,\beta_2,R}$ around the fixed point, $f_{\boldsymbol\beta}^{\circ n}(\mathfrak w)\to
\begin{bmatrix}\omega(\beta_1)&0\\0&\omega(\beta_2)\end{bmatrix}$ as $n\to\infty$ for all $w\in B$, uniformly on $B$. Since on any subset at positive distance 
from the complement of $\mathcal H_{\beta_1,\beta_2,R}$ the norm topology and the hyperbolic topology are equivalent, there exists $r>0$ such that 
$\begin{bmatrix}\omega(\beta_1)&v\\0&\omega(\beta_2)\end{bmatrix}\in B$ for all $v\in\mathcal B,\|v\|\le r$. 

Now assume towards contradiction that the spectral radius of the linear completely bounded map $U\colon\mathcal B\ni v\mapsto(\eta\circ E)\left[(\omega(\beta_1)-D)^{-1}
v(\omega(\beta_2)-D)^{-1}\right]\in\mathcal B$ is greater than or equal to one. According to the spectral radius formula, this forces 
$\displaystyle\lim_{n\to\infty}\|U^n\|^\frac1n=\inf_{n\ge1}\|U^n\|^\frac1n\ge1$. However, by direct computation, with
${\boldsymbol{\beta}}=\begin{bmatrix} \beta_1 & 0 \\0 & \beta_2 \end{bmatrix}$, we obtain
\begin{align*}
f_{\boldsymbol\beta}^{\circ 2}\!\left(\begin{bmatrix}\omega(\beta_1)&\!\!v\\0&\!\!\omega(\beta_2)\end{bmatrix}\right) & = 
f_{\boldsymbol\beta}\left(f_{\boldsymbol\beta}\left(\begin{bmatrix}\omega(\beta_1)&v\\0&\omega(\beta_2)\end{bmatrix}\right)\right) \hspace{7.5cm}
\end{align*}
\begin{align*}
&=f_{\boldsymbol\beta}\left(\begin{bmatrix}\beta_1\!-(\eta\circ\!E)\left[(\omega(\beta_1)-D)^{-1}\right]&(\eta\circ\!E)\left[(\omega(\beta_1)-D)^{-1}v(\omega(\beta_2)-D)^{-1}
\right]\\0&\beta_2-(\eta\circ E)\left[(\omega(\beta_2)-D)^{-1}\right]\end{bmatrix}\!\right)\\
& =  f_{\boldsymbol\beta}\left(\begin{bmatrix}\omega(\beta_1)&\!\!Uv\\0&\!\!\omega(\beta_2)\end{bmatrix}\right)=\begin{bmatrix}\omega(\beta_1)&\!\!U^2v\\0&\!\!\omega(\beta_2)\end{bmatrix}, \quad\text{ so that }\\
f_{\boldsymbol\beta}^{\circ n}\!\left(\begin{bmatrix}\omega(\beta_1)&\!\!v\\0&\!\!\omega(\beta_2)\end{bmatrix}\right)&=
\begin{bmatrix}\omega(\beta_1)&\!\!U^nv\\0&\!\!\omega(\beta_2)\end{bmatrix},\quad n\in\mathbb N.
\end{align*}
Since, as mentioned above, the norm and hyperbolic topologies coincide locally and $f_{\boldsymbol\beta}$ is a strict hyperbolic contraction, there exists an $n_0\in\mathbb N$
such that $f_{\boldsymbol\beta}^{\circ n}\!\left(\begin{bmatrix}\omega(\beta_1)&\!\!v\\0&\!\!\omega(\beta_2)\end{bmatrix}\right)=
\begin{bmatrix}\omega(\beta_1)&\!\!U^nv\\0&\!\!\omega(\beta_2)\end{bmatrix}\in\left\{\begin{bmatrix}\omega(\beta_1)&\!\!\xi\\0&\!\!\omega(\beta_2)\end{bmatrix}\colon
\|\xi\|\le r/2\right\}$ for all $v,\|v\|\le r,n\ge n_0$. Thus, $\|U^{n_0}v\|\le\frac{r}{2}$ for all $v$, $\|v\|\le r$, so that $\|U^{n_0}\|\le\frac12$.
Of course, this means that $\displaystyle\lim_{n\to\infty}\|U^n\|^\frac1n=\lim_{m\to\infty}\|U^{mn_0}\|^\frac{1}{mn_0}\le\lim_{m\to\infty}\|U^{n_0}\|^\frac{m}{mn_0}=
\frac{1}{\sqrt[n_0]{2}}<1$, contradicting our hypothesis. Thus, the spectral radius of $U$ is strictly less than one. 

Since $U$ is the composition of $\eta$ with $E\left[(\omega(\beta_1)-D)^{-1}\ \cdot \ (\omega(\beta_2)-D)^{-1}
\right]$ and $u_{\beta_1,\beta_2}$ is the composition of $E\left[(\omega(\beta_1)-D)^{-1}\ \cdot \ (\omega(\beta_2)-D)^{-1}
\right]$ with $\eta$, the spectral radius of $U$ coincides with the spectral radius of $u_{\beta_1,\beta_2}$. 
This concludes the proof of our proposition. 
\end{proof}

\begin{rmk}
\begin{trivlist}
\item[\,(1)] The proof given above works with almost no modification for the case $\{\beta_1,\beta_2\in\mathbb H^+(\mathcal B)\}$, case already covered by
\cite[Proposition 4.1]{CAOT}.
\item[\ \,(2)] Although not directly useful in our paper, we mention that another side benefit of Proposition \ref{welldefined1} is that it allows the subordination function 
$\omega({\boldsymbol\beta})$ to be extended to a neighborhood of $\mathcal H$ {\em as a fixed point} of $f_{\boldsymbol\beta}(\cdot)$.

\end{trivlist}

\end{rmk}

\subsection{Second step  in the proof of Proposition \ref{welldefined3}} \label{secondstep}
The following theorem will be useful in order to deduce Proposition \ref{welldefined3} from  Proposition \ref{welldefined1}.

\begin{thm}\label{pluri}
Assume that $\Omega\subset\mathbb C^d$ is an open connected set, $\mathcal A$ is a Banach algebra, and $f\colon\Omega\to\mathcal A$ is analytic. Denote by $\rho(x)$
the spectral radius of the element $x\in\mathcal A$. The function $\Omega\ni\mathbf z\mapsto\rho(f(\mathbf z))\in[0,+\infty)$ is plurisubharmonic.
\end{thm}
The result is well-known (see for instance \cite[Section 4.1]{Ch}), but we provide a brief sketch of a proof. First, recall that the function 
$\mathcal A\ni x\mapsto\rho(x)\in[0,+\infty)$ is upper semicontinuous \cite[I.1,Theorem 31]{VM}. Obviously, $\Omega\ni\mathbf z\mapsto\|f(\mathbf z)\|\in[0,+\infty)$ is 
continuous. We claim it is also plurisubharmonic. Indeed, in any Banach space, the norm of an element is equal to the 
value at it of a certain norm-one linear functional defined on the Banach space. In particular, for any $x\in\mathcal A$, there exists a norm-one continuous linear 
functional $x^\star\colon\mathcal A\to\mathbb C$ such that $\|x\|=x^\star(x)$. Thus, for any given $\mathbf z_0\in\Omega$, one may find a norm-one linear 
functional $\phi_{\mathbf z_0}\colon\mathcal A\to\mathbb C$ such that $\phi_{\mathbf z_0}(f(\mathbf z_0))=\|f(\mathbf z_0)\|$. The function
$\Omega\ni\mathbf z\mapsto\phi_{\mathbf z_0}(f(\mathbf z))\in\mathbb C$ is analytic because of the continuity of the linear functional $\phi_{\mathbf z_0}$ and
the assumption of analyticity imposed on $f$. According to \cite[Proposition 1.29]{GZ}, the map $\Omega\ni\mathbf z\mapsto|\phi_{\mathbf z_0}(f(\mathbf z))|^\alpha\in[0,+\infty)$
becomes then a continuous plurisubharmonic function satisfying $|\phi_{\mathbf z_0}(f(\mathbf z_0))|^\alpha=\|f(\mathbf z_0)\|^\alpha,|\phi_{\mathbf z_0}(f(\mathbf z))|^\alpha\le
\|f(\mathbf z)\|^\alpha$, $\mathbf z\in\Omega,\alpha>0$. If we let $\Omega^\star=\{\phi_{\mathbf z}\colon\mathbf z\in\Omega\}\subseteq\{\phi\colon\mathcal A\to\mathbb C\colon
\phi\text{ linear, continuous, }\|\phi\|=1\}$, then for any $\alpha>0$,
$$
\|f(\mathbf z)\|^\alpha=\max\{|\phi_\mathbf w(f(\mathbf z))|^\alpha\colon\phi_\mathbf w\in\Omega^\star,\mathbf w\in\Omega\},\quad \mathbf z\in\Omega,
$$
so that, according to \cite[Proposition 1.40]{GZ}, it is indeed plurisubharmonic.

Finally, as already mentioned, the spectral radius formula is given as
$$
\rho(f(\mathbf z))=\lim_{n\to\infty}\|f(\mathbf z)^n\|^\frac1n=\inf_{n\ge1}\|f(\mathbf z)^n\|^\frac1n,
$$
that is, the upper semicontinuous function $\Omega\ni\mathbf z\mapsto\rho(f(\mathbf z))\in[0,+\infty)$ is the infimum of a family of plurisubharmonic functions, hence 
itself plurisubharmonic (see \cite[Proposition 1.28.(2)]{GZ}).

The property useful for our purposes of plurisubharmonic functions is that they satisfy a maximum principle \cite[Corollary 1.37]{GZ}: if $\mathbf z_0$ is a local maximum for
$\rho(f(\mathbf z))$, then $\mathbf z\mapsto\rho(f(\mathbf z))$ is constant on a neighborhood of $\mathbf z_0$ in $\Omega$.

\begin{rmk}\label{psh}
Consider an open connected set $\Omega\subseteq\mathbb C^d$ for some integer $d\ge1$, a Banach algebra $\mathcal A$, an analytic function $f\colon\Omega\to
\mathcal A$, and a number $M>0$. Assume that $\rho(f(\mathbf z))\leq M$ for all $\mathbf z\in\Omega$. Then either $\rho(f(\mathbf z))\equiv M$  for all $\mathbf z\in\Omega$,
or $\rho(f(\mathbf z))<M$  for all $\mathbf z\in\Omega$. In the second case, there exists a sequence $\mathbf z_n\in\Omega$ such that $\mathbf z_n\to\infty$ as $n\to\infty$, 
$\rho(f(\mathbf z_n))<\rho(f(\mathbf z_{n+1})),n\in\mathbb N$, and $M\ge\sup\{\rho(f(\mathbf z))\colon\mathbf z\in\Omega\}=
\lim_{n\to\infty}\rho(f(\mathbf z_n))$. By $\mathbf z_n\to\infty$ we mean that for any compact $K\subset\Omega$ there exists $n_K\in\mathbb N$ such that $\mathbf z_n
\not\in K$ for all $n\ge n_K$.
\end{rmk}

Observe that our boundedness hypothesis automatically excludes the possibility that $f$ is {\em not} constant and simultaneously $\Omega=\mathbb C^d$.

\begin{proof}
The statements are trivial: the first one is a consequence of the maximum principle for plurisubharmonic functions \cite[Corollary 1.37]{GZ} and Theorem \ref{pluri},
and the second of the very definition of supremum.
\end{proof}

\begin{proof}[Proof of Proposition \ref{welldefined3}]
First observe that $(ze_{11}-\gamma_0)\otimes1_{\mathcal{A}}-\gamma_1\otimes s-\gamma_2\otimes d$ is invertible for all $z\in \mathbb C^+\cup \mathbb C^-\cup(\mathbb R\setminus\sigma((\gamma_0)_{1,1}\cdot1_{\mathcal{A}}+(\gamma_1)_{1,1}s+(\gamma_2)_{1,1}d+u^*Q^{-1}u))$. Indeed, splitting this element in four blocks 
$$(ze_{11}-\gamma_0)\otimes1_{\mathcal{A}}-\gamma_1\otimes s-\gamma_2\otimes d= 
\begin{bmatrix}
((ze_{11}-\gamma_0)\otimes1_{\mathcal{A}}-\gamma_1\otimes s-\gamma_2\otimes d)_{1,1} & u^*\\
u & Q
\end{bmatrix}
$$
as in Section \ref{sec:lin}, where $Q\in M_{m-1}(\mathcal A)$ is assumed to be invertible. By our choice, $u$ and $Q$ do not depend on $z$, but only on $s,d$, and the $\gamma$'s.
The Schur complement formula guarantees that the above is invertible in $M_m(\mathcal A)$ whenever
$((ze_{11}-\gamma_0)\otimes1_{\mathcal{A}}-\gamma_1\otimes s-\gamma_2\otimes d)_{1,1}-u^*Q^{-1}u$
is invertible in $\mathcal A$. Since $(u^*Q^{-1}u)^*=u^*Q^{-1}u$ and $(\gamma_0\otimes1_{\mathcal{A}}+\gamma_1\otimes s+\gamma_2\otimes d)_{1,1}=(\gamma_0)_{1,1}\cdot1_{\mathcal{A}}+(\gamma_1)_{1,1}s+(\gamma_2)_{1,1}d$ is also selfadjoint in $\mathcal A$, it follows that for any $z\in\mathbb C^+\cup \mathbb C^-\cup(\mathbb R\setminus\sigma((\gamma_0)_{1,1}\cdot1_{\mathcal{A}}+(\gamma_1)_{1,1}s+(\gamma_2)_{1,1}d+u^*Q^{-1}u))$ -- a connected set
which is also a neighborhood of infinity -- the random variable $(ze_{11}-\gamma_0)\otimes1_{\mathcal{A}}-\gamma_1\otimes s-\gamma_2\otimes d$ is invertible, as claimed.

Direct computation shows that points $(ze_{11}-\gamma_0)$, $z\in\mathbb C$, belong to the topological closure of $\mathbb H^+(M_m(\mathbb C))\cup
\mathbb H^-(M_m(\mathbb C))$. Since the set of invertible elements in a Banach algebra is open in the norm topology, it follows immediately that if 
$z$ is such that $(ze_{11}-\gamma_0)\otimes1_{\mathcal{A}}-\gamma_1\otimes s-\gamma_2\otimes d$ is invertible in $M_m(\mathcal A)$, then 
there is a small enough neighborhood $V$ in $M_m(\mathbb C)$  such that $w\otimes1_{\mathcal{A}}-\gamma_1\otimes s-\gamma_2\otimes d$ is invertible
in $M_m(\mathcal A)$ for all $w\in V$ and, of course, $V\cap(\mathbb H^+(M_m(\mathbb C))\cup
\mathbb H^-(M_m(\mathbb C)))\neq\varnothing.$ 

This guarantees in particular that $
\omega(w)= w-\sigma^2\gamma_1({\rm id}_m\otimes\tau)\left[\left(w\otimes1_{\mathcal{A}}-\gamma_1\otimes s-\gamma_2\otimes d\right)^{-1}\right]\gamma_1
$ extends analytically to a neighborhood of $\{ze_{11}-\gamma_0\colon z\in\mathbb C^+\cup
\mathbb C^-\cup(\mathbb R\setminus\sigma((\gamma_0)_{1,1}\cdot1_{\mathcal{A}}+(\gamma_1)_{1,1} s+(\gamma_2)_{1,1} d+u^*Q^{-1}u))\}$ in 
$M_m(\mathbb C)$. 
{We use that the spectral radius of operators on $M_m(\C)$ is continuous} to conclude thanks to Proposition \ref{welldefined1} that for any $z_1,z_2\in\mathbb C^+\cup
\mathbb C^-\cup(\mathbb R\setminus\sigma((\gamma_0)_{1,1}\cdot1_{\mathcal{A}}+(\gamma_1)_{1,1}s+(\gamma_2)_{1,1}d+u^*Q^{-1}u))$, we may write 
$$
1\ge\lim_{\beta_j\to0,\beta_j\in\mathbb H^{\bullet_j}(M_m(\mathbb C))}
\rho\left(u_{z_1e_{11}-\gamma_0+\beta_1, z_2e_{11}-\gamma_0+\beta_2}\right)=\rho\left(u_{z_1e_{11}-\gamma_0, z_2e_{11}-\gamma_0}\right),
$$
where $\bullet_j$ is the sign of the imaginary part of $z_j$, $j=1,2$, and if one or both of $z_j$ are real, then we agree to make the choice $\bullet_j=+$.
Since the correspondence $(z_1,z_2)\mapsto u_{z_1e_{11}-\gamma_0, z_2e_{11}-\gamma_0}$ is an analytic map from the open subset $\{\mathbb C^+\cup
\mathbb C^-\cup(\mathbb R\setminus\sigma((\gamma_0)_{1,1}\cdot1_{\mathcal{A}}+(\gamma_1)_{1,1}s+(\gamma_2)_{1,1} d+u^*Q^{-1}u))\}^2$ of $\mathbb C^2$
into the Banach algebra of (bounded) linear self-maps of $M_m(\mathbb C)$, it follows that the correspondence $(z_1,z_2)\mapsto\rho\left(u_{z_1e_{11}-\gamma_0,
z_2e_{11}-\gamma_0}\right)\in[0,+\infty)$ is plurisubharmonic on the same set, according to Theorem \ref{pluri}. Since plurisubharmonic functions satisfy the maximum principle, 
it follows (see Remark \ref{psh}) that either $\rho\left(u_{z_1e_{11}-\gamma_0, z_2e_{11}-\gamma_0}\right)\equiv1$ for all pairs $(z_1,z_2)$ in the above-described domain of this 
function, or that $\rho\left(u_{z_1e_{11}-\gamma_0, z_2e_{11}-\gamma_0}\right)<1$ for all pairs $(z_1,z_2)$ in this domain. Thus, in order to show that the second 
case holds, it is enough to find a single such pair in which this spectral radius is strictly less than one. The pair we focus on will be of the form $(z_1,z_2)=(y,y)$ for $y\in \mathbb R$ sufficiently large. Note that, for such a pair, $u_{z_1e_{11}-\gamma_0, z_2e_{11}-\gamma_0}\colon M_m(\mathbb C)\to M_m(\mathbb C)$ is completely positive. We apply to it Theorem \ref{real}. Assume towards contradiction that the completely positive map $u_{ye_{11}-\gamma_0,ye_{11}-\gamma_0}$ has spectral radius equal to one. 
This map is the composition of two completely positive maps, namely $v\mapsto\sigma^2\gamma_1v\gamma_1$ and 
$v\mapsto ({\rm id}_m\otimes\tau)\left[(\omega(ye_{11}-\gamma_0)\otimes1_{\mathcal{A}}-\gamma_2\otimes d)^{-1}(v\otimes1_{\mathcal{A}}) 
(\omega(ye_{11}-\gamma_0)\otimes1_{\mathcal{A}}-\gamma_2\otimes d)^{-1}\right]$. Since the spectral radius of $AB$ coincides with the spectral radius of $BA$ 
for any linear maps $A,B$ on a Banach space, it follows that $\rho(u_{ye_{11}-\gamma_0,ye_{11}-\gamma_0})=1$ if and only if
\begin{equation}\label{spectral}
\rho\left(\sigma^2\gamma_1({\rm id}_m\otimes\tau)\left[(\omega(ye_{11}-\gamma_0)\otimes1_{\mathcal{A}}-\gamma_2\otimes d)^{-1}(\ \cdot\ \otimes1_{\mathcal{A}}) 
(\omega(ye_{11}-\gamma_0)\otimes1_{\mathcal{A}}-\gamma_2\otimes d)^{-1}\right]\gamma_1\right)=1.
\end{equation}
Thus, according to  Theorem \ref{real}, there exists a matrix $v_0\neq0$ such that 
$$
\sigma^2\gamma_1({\rm id}_m\otimes\tau)\left[(\omega(ye_{11}-\gamma_0)\otimes1_{\mathcal{A}}-\gamma_2\otimes d)^{-1}(v_0 \otimes1_{\mathcal{A}}) 
(\omega(ye_{11}-\gamma_0)\otimes1_{\mathcal{A}}-\gamma_2\otimes d)^{-1}\right]\gamma_1=v_0.
$$
Recalling that $\omega(w)=w-\sigma^2\gamma_1({\rm id}_m\otimes\tau)\left[\left(\omega(w)\otimes1_{\mathcal{A}}-\gamma_2\otimes d\right)^{-1}\right]\gamma_1,$ it follows
that 
$$
\omega'(w)(c)=c+\sigma^2\gamma_1({\rm id}_m\otimes\tau)\left[\left(\omega(w)\otimes1_{\mathcal{A}}-\gamma_2\otimes d\right)^{-1}\omega'(w)(c)
\left(\omega(w)\otimes1_{\mathcal{A}}-\gamma_2\otimes d\right)^{-1}\right]\gamma_1,
$$
for all $w\in\mathbb H^\pm(M_m(\mathbb C))$. Extending this to our point $ye_{11}-\gamma_0$, we have two options: either $\omega'(ye_{11}-\gamma_0)$ is bijective,
and then there exists a $c_y\in M_m(\mathbb C)$ such that $\omega'(ye_{11}-\gamma_0)(c_y)=v_0$, or there is no such $c_y$ and then $\omega'(ye_{11}-\gamma_0)$
is not bijective, which in finite dimensional spaces means it has a nontrivial kernel. In the first situation, we obtain
\begin{align*}
v_0 &=\omega'(ye_{11}-\gamma_0)(c_y)\\
&=c_y\!+\!\sigma^2\gamma_1({\rm id}_m\!\otimes\!\tau)\!\left[\left(\omega(ye_{11}\!-\!\gamma_0)\!\otimes\!1_{\mathcal{A}}\!-\!\gamma_2\!\otimes\!d\right)^{-1}\!
\omega'(ye_{11}\!-\!\gamma_0)(c_y)\left(\omega(ye_{11}\!-\!\gamma_0)\!\otimes\!1_{\mathcal{A}}\!-\!\gamma_2\!\otimes\!d\right)^{-1}\right]\!\gamma_1\\
&=c_y\!+\!\sigma^2\gamma_1({\rm id}_m\!\otimes\!\tau)\!\left[\left(\omega(ye_{11}\!-\!\gamma_0)\!\otimes\!1_{\mathcal{A}}\!-\!\gamma_2\!\otimes\!d\right)^{-1}\!
v_0\left(\omega(ye_{11}\!-\!\gamma_0)\!\otimes\!1_{\mathcal{A}}\!-\!\gamma_2\!\otimes\!d\right)^{-1}\right]\!\gamma_1\\
&=c_y\!+\!v_0,
\end{align*}
which implies $c_y=0$, so that $0=\omega'(ye_{11}\!-\!\gamma_0)(c_y)=v_0\neq0$, which implies that the first situation cannot occur. If $0\neq c\in\ker(
\omega'(ye_{11}-\gamma_0))$, then $0=\omega'(ye_{11}\!-\!\gamma_0)(c)=c+0=c\neq0$, again a contradiction. Thus, it is impossible that \eqref{spectral} takes place.
This concludes the proof of Proposition \ref{welldefined3}.
\end{proof}

Unlike for the case of Proposition \ref{welldefined1}, here there were several points where the finite dimensionality of $M_m(\mathbb C)$ was used. However, we cannot think 
at this moment of a situation in which  Proposition \ref{welldefined3} would have an appropriate formulation involving an infinite-dimensional algebra of scalars.

The reader might be concerned by one element in our proof, namely the fact that we have not hesitated to extend analytically $\omega$ around $(ye_{11}\!-\!\gamma_0)$.
This problem has been essentially addressed in \cite{BBC}. However, the reader can find a simple argument for this extension by recalling Voiculescu's result \cite{V2000}, namely
$$\left(\omega(w)\otimes1_{\mathcal{A}}-\gamma_2\otimes d\right)^{-1}=E_{M_m(\mathbb C\langle d\rangle)}\left[
\left(w\otimes1_{\mathcal{A}}-\gamma_1\otimes s-\gamma_2\otimes d\right)^{-1}\right].$$ As $y\in \mathbb R$ is taken so that $(ye_{11}-\gamma_0)\otimes1_{\mathcal{A}}-
\gamma_1\otimes s-\gamma_2\otimes d$ is invertible, we are guaranteed that it will remain invertible on a neighborhood of $ye_{11}-\gamma_0$ in $M_m(\mathbb C)$.
On the one hand this guarantees (via \eqref{six}) the existence and analyticity of $\omega$ on this neighborhood, and on the other, the boundedness of the conditional
expectation of $\left(w\otimes1_{\mathcal{A}}-\gamma_1\otimes s-\gamma_2\otimes d\right)^{-1}$ on this neighborhood guarantees that $
\omega(w)\otimes1_{\mathcal{A}}-\gamma_2\otimes d$ stays invertible on the same neighborhood. This shows that taking the derivative of $\omega$ at $w=ye_{11}-\gamma_0
\in M_m(\mathbb C)$ is permissible. 

\section{Preliminary results}\label{prelimi}

\subsection{Notations}

We start by fixing some notations. 
We consider the canonical linearization 
$$L_P(t_1,t_2)=\gamma_0\otimes 1+\gamma_1\otimes t_1+\gamma_2\otimes t_2 \in M_m(\mathbb{C})\otimes \mathbb{C}\langle t_1,t_2\rangle$$ 
of the selfadjoint polynomial $P$ with the properties outlined in Section \ref{mainresults}.
By \eqref{coin}, $\Tr(zI_N-X_N)^{-1}$ is related to the generalized resolvent 
$$R_N(\beta):=(\beta\otimes I_N-\gamma_1\otimes W_N-\gamma_2\otimes D_N)^{-1},$$
by 
$$\Tr(zI_N-X_N)^{-1}=(\Tr\otimes \Tr)((e_{11}\otimes I_N)R_N(ze_{11}-\gamma_0)),\quad z\in \mathbb{C}\setminus\mathbb{R}.$$
Denoting by $(E_{ij})_{1\leq i,j\leq N}$ the canonical basis of $M_N(\C)$, we define the matrices $R_{ij}(\beta)$, $1\leq i,j\leq N$, by $$R_N(\beta)=\sum_{i,j=1}^N R_{ij}(\beta)\otimes E_{ij}.$$
Schur inversion formula (Proposition \ref{Schur} applied to $A=\beta\otimes I_N-\gamma_1\otimes W_N-\gamma_2\otimes D_N$ and $I^c=\{k, k+N,\ldots, k+(N-1)m\}$) relates $R_N(\beta)$ to its Schur complements 
$$\beta-W_{kk}\gamma_1-D_{kk}\gamma_2-\gamma_1\otimes C_k^{(k)*}R^{(k)}(\beta)\gamma_1\otimes C_k^{(k)},\quad k=1,\ldots, N,$$
where 
$$R^{(k)}(\beta)=(\beta\otimes I_{N-1}-\gamma_1\otimes W_N^{(k)}-\gamma_2\otimes D_N^{(k)})^{-1}.$$
Here $W_N^{(k)}, D_N^{(k)}$ denote the $(N-1)\times(N-1)$ matrices obtained respectively from $W_N, D_N$ 
by deleting the $k$-th row/column and $C_k^{(k)}$ is the $(N-1)$-dimensional vector obtained from the $k$-th column of $W_N$ by deleting its $k$-th component. 

An immediate consequence of Lemma \ref{lem_resolvent_identity} is the following relation between $R_N$ and the generalized resolvent 
$$R^{(ab)}(\beta)=(\beta\otimes I_N-\gamma_1\otimes W_N^{(ab)}-\gamma_2\otimes D_N)^{-1},$$
where $W_N^{(ab)}$ is obtained from $W_N$ by replacing its $(a,b)$ and $(b,a)$ entries by $0$:
\begin{equation}\label{remove}
R^{(ab)}(\beta)-R_N(\beta)=-R^{(ab)}(\beta)(1-\frac{1}{2}\delta_{ab})(\gamma_1 \otimes W_{ab}E_{ab}+\gamma_1 \otimes \overline{W_{ab}}E_{ba})R_N(\beta).
\end{equation}
One deduces from \eqref{remove} the following bound:
\begin{equation*}
\|R^{(ab)}(\beta)-R_N(\beta)\|\leq 2\delta_N\|\gamma_1\|\|R^{(ab)}(\beta)\|\|R_N(\beta)\|.
\end{equation*}
Analogously, we denote by $R^{(kab)}$ the generalized resolvent 
$$R^{(kab)}(\beta)=(\beta\otimes I_N-\gamma_1\otimes W_N^{(kab)}-\gamma_2\otimes D_N^{(k)})^{-1},$$
where $W_N^{(kab)}$ is obtained from $W_N^{(k)}$ by replacing its $(a,b)$ and $(b,a)$ entries by $0$; and by $R^{(kabcd)}$ the generalized resolvent 
$$R^{(kabcd)}(\beta)=(\beta\otimes I_N-\gamma_1\otimes W_N^{(kabcd)}-\gamma_2\otimes D_N^{(k)})^{-1},$$
where $W_N^{(kabcd)}$ is obtained from $W_N^{(k)}$ by replacing its $(a,b),(b,a),(c,d)$ and $(d,c)$ entries by $0$.

Martingales appearing in this paper will be with respect to the filtration 
$$(\mathcal{F}_k:=\sigma(W_{ij},1\leq i\leq j\leq k))_{k\geq 1};$$
$\mathbb{E}_{\leq k}$ denotes the conditional expectation on the sigma-field $\mathcal{F}_k$ 
and $\mathbb{E}_{k}$ the expectation with respect to the $k$-th column $\{W_{ik},1\leq i\leq N\}$ of $W_N$.

We will consider, for each $N\in \mathbb{N}$, a $W^*$-probability space $(\mathcal{A}_N, \tau_N)$, a semicircular element $s_N\in \mathcal{A}_N$ of mean $0$ and variance $N\sigma_N^2$ and a von Neumann subalgebra $\mathcal{D}_N\subset \mathcal{A}_N$ isomorphic to the algebra of $N\times N$ diagonal matrices with complex entries and freely independent from $s_N$. As we have seen, $\gamma_1\otimes s_N\in M_m(\mathcal{A}_N)$ is a centered $M_m(\mathbb{C})$-valued semicircular element of variance $\eta_N: b\mapsto N\sigma_N^2 \gamma_1 b\gamma_1$ which is free with amalgamation over $M_m(\mathbb{C})$ from $M_m(\mathcal{D}_N)$ in the $M_m(\mathbb{C})$-valued $W^*$-probability space $(M_m(\mathcal{A}_N),{\rm id}_m\otimes \tau_N)$. For $D_N\in \mathcal{D}_N\simeq \mathcal{D}_N(\mathbb{C})$, the generalized resolvent
$$r_N(\beta):=(\beta\otimes1_{\mathcal{A}_N}-\gamma_1\otimes s_N-\gamma_2\otimes D_N)^{-1}$$ 
and the subordination map 
$$\omega_N(\beta):=\beta-\eta_N\big(({\rm id}_m\otimes \tau_N)[r_N(\beta)]\big)$$
are related according to \eqref{subforte} by 
\begin{equation}\label{defhatR}
\hat{R}(\beta):=E_{M_m(\mathcal D_N)}[r_N(\beta)]=\left(\omega_N(\beta)\otimes 1_{\mathcal{A}_N}-\gamma_2\otimes D_N\right)^{-1}.\end{equation}
By analogy, one defines 
\begin{equation}\label{defomegatilde}
\Omega_N(\beta):=\beta-\eta_N\big(\mathbb{E}[({\rm id}_m\otimes N^{-1}\Tr)(R_N(\beta))]\big).\end{equation}

In the sequel, we will use the notation $O(v_N)$ when a quantity depending on $N\in \mathbb{N}$, $z\in \C \setminus \R$, and sometimes on $k\in \{1,\ldots ,N\}$, satisfies the following : for any compact subset $K$ of $\mathbb{C}\setminus \mathbb{R}$, there exists $N_K\in \mathbb{N}$ and $C_K>0$ such that for any $N\geq N_K$, for any $z\in K$, (for any $k\in \{1,\ldots ,N\}$,) this quantity is bounded by $C_Kv_N$.
	

Throughout the paper, $C$, $c$ denote some positive constants and $Q$ denotes some deterministic  polynomial in one or several commuting indeterminates 
; they can depend on $m$, $\gamma_0, \gamma_1, \gamma_2$ and they may vary from line to line.

\subsection{Free probability bounds and convergences}

For $\beta \in \mathbb{H}^+(M_m(\mathbb{C}))$ and $k=1,\ldots ,N$, as observed in Section \ref{sectionresolvent}, 
\begin{equation}\label{majrN}
\Vert r_N(\beta)\Vert \leq \Vert(\Im \beta)^{-1}\Vert
\end{equation}
and
$$\Vert (\omega_N(\beta)\otimes 1_{\mathcal{A}_N}-\gamma_2\otimes D_N)^{-1}\Vert \leq \Vert (\Im\omega_N(\beta))^{-1}\Vert\leq \Vert (\Im \beta)^{-1}\Vert .$$
It is a consequence of Lemma \ref{inversible} that $r_N$ is also defined on $\{ze_{11}-\gamma_0, z\in \mathbb{C}\setminus \mathbb{R}\}$ and it follows from Lemma \ref{alta lemma} with $y=(s_N,D_N)$, Assumptions \ref{hyp:offdiagonal}, \ref{hyp:deformation} and \eqref{resolventbound} that, for $z\in \mathbb{C}\setminus \mathbb{R}$, 
\begin{equation}\label{petitrN}
\Vert r_N(ze_{11}-\gamma_0)\Vert \leq Q_1(2N^{1/2}\sigma_N,\Vert D_N\Vert)\Vert(zI_N-P(s_N,D_N))^{-1}\Vert+Q_2(2N^{1/2}\sigma_N,\Vert D_N\Vert)=O(1).
\end{equation}

\begin{lem}\label{Qbound} 
\begin{align*}
 \Vert (\omega_N(ze_{11}-\gamma_0)\otimes 1_{\mathcal{A}_N}-\gamma_2\otimes D_N)^{-1}\Vert  & = O(1);\\
 \left\Vert \frac{\partial}{\partial z}(\omega_N(ze_{11}-\gamma_0)\otimes 1_{\mathcal{A}_N}-\gamma_2\otimes D_N)^{-1}\right\Vert & = O(1).
\end{align*}
\end{lem}

\begin{proof}
For $z\in \mathbb{C}\setminus\mathbb{R}$, it follows from \eqref{subforte} that 
$$E_{M_m(\mathcal{D}_N)}\left[\left((ze_{11}-\gamma_0)\otimes 1_{\mathcal{A}_N}-\gamma_1\otimes s_N-\gamma_2\otimes D_N\right)^{-1}\right]=\left(\omega_N(ze_{11}-\gamma_0)\otimes 1_{\mathcal{A}_N}-\gamma_2\otimes D_N\right)^{-1}$$
holds and therefore 
$$\Vert (\omega_N(ze_{11}-\gamma_0)\otimes 1_{\mathcal{A}_N}-\gamma_2\otimes D_N)^{-1}\Vert \leq \Vert (ze_{11}\otimes 1_{\mathcal{A}_N}-L_P(s_N,D_N))^{-1}\Vert .$$
Since moreover $$\left\Vert e_{11}+\eta_N\left(({\rm id}_m\otimes \tau_N)\left[\left(ze_{11}\otimes 1_{\mathcal{A}_N}-L_P(s_N,D_N)\right)^{-1}(e_{11}\otimes1_{\mathcal{A}_N})\left(ze_{11}\otimes 1_{\mathcal{A}_N}-L_P(s_N,D_N)\right)^{-1}\right]\right)\right\Vert$$
$$\leq 1+N\sigma_N^2 \Vert \gamma_1\Vert^2\Vert\left(ze_{11}\otimes 1_{\mathcal{A}_N}-L_P(s_N,D_N)\right)^{-1}\Vert^2,$$
it follows that 
\begin{align*}
\Vert\frac{\partial}{\partial z} & (\omega_N(ze_{11}-\gamma_0)\otimes1_{\mathcal{A}_N}-\gamma_2\otimes D_N)^{-1}\Vert\\
&\leq\Vert (ze_{11}\otimes 1_{\mathcal{A}_N}-L_P(s_N,D_N))^{-1}\Vert^2(1+N\sigma_N^2 \Vert \gamma_1\Vert^2\Vert\left(ze_{11}\otimes 1_{\mathcal{A}_N}-L_P(s_N,D_N)\right)^{-1}\Vert^2).
\end{align*}
Using Lemma \ref{alta lemma} with $y=(s_N,D_N)$,
$$\Vert (ze_{11}\otimes 1_{\mathcal{A}_N}-L_P(s_N,D_N))^{-1}\Vert \leq Q_1(N^{1/2}\sigma_N,\Vert D_N\Vert)\vert \Im z\vert^{-1}+Q_2(N^{1/2}\sigma_N,\Vert D_N\Vert)$$
and we are done.
\end{proof}

For $\beta\in M_m(\mathbb{C})$ such that $\beta=ze_{11}-\gamma_0$ with $z\in\mathbb{C}\setminus\mathbb R$, define
\begin{equation}\label{defhatRk}
\hat R_k(\beta)=(\beta-N\sigma_N^2\gamma_1({\rm id}_m\otimes\tau)(r_N(\beta))\gamma_1-D_{kk}\gamma_2)^{-1}=(\omega_N(\beta)-D_{kk}\gamma_2)^{-1}.
\end{equation} 
It readily follows from Lemma \ref{Qbound} that 
\begin{equation}\label{majorationunifhatRk}
\Vert \hat R_k(ze_{11}-\gamma_0)\Vert 
=O(1),\end{equation}
\begin{equation}\label{majhatdR}
\left\Vert \frac{\partial}{\partial z}\hat R_k(ze_{11}-\gamma_0)\right\Vert 
=O(1),\end{equation}
and that
\begin{equation}\label{bornemanq} 
\sup_{ t \in \text{supp}(\nu_N)} \|  (\omega_N(ze_{11}-\gamma_0)-t\gamma_2 )^{-1}\|=O(1).\end{equation}

\begin{lem}\label{bound} The map $\mathbb C^\pm\times\mathbb R\ni(z,t)\mapsto(\omega(ze_{11}-\gamma_0)-t\gamma_2)^{-1}$ is well-defined 
and analytic. Thus, it is bounded on any compact subset of $\mathbb C^\pm\times\mathbb R$.
\end{lem}

\begin{proof}
Fix $M>\|d\|$ and a compact $K\subset\mathbb C^+$.
We use the equality $\omega(ze_{11}-\gamma_0)=ze_{11}-\gamma_0-\gamma_1G(ze_{11}-\gamma_0)\gamma_1,$ where $G$ denotes the generalized Cauchy transform of
$\gamma_1\otimes s+\gamma_2\otimes d$ (i.e. $G(b)=(\textrm{id}_m\otimes\tau)(r(b))$), together with the structure of the linearization and normal families.
The boundedness statement is obvious if $t\in\sigma(d)$, by Equation \eqref{subforte} and the boundedness of evaluation maps.
We show that, given $M>0$ (which we assume for convenience sufficiently large so that $\sigma(d)\subset[-M,M]$), there exists an $\mathfrak r>0$ (depending on it)
so that $\omega_N(ze_{11}-\gamma_0)-t\gamma_2=ze_{11}-\gamma_0-\gamma_1G(ze_{11}-\gamma_0)\gamma_1-t\gamma_2$ is invertible for all $|z|>\mathfrak r$.
We recall the shape of $G(ze_{11}-\gamma_0)\in M_m(\mathbb C)$:
\begin{align*}
G(& ze_{11}-\gamma_0)\\
& = \begin{bmatrix} \tau\left((z-u^*Q^{-1}u)^{-1}\right) & -({\rm id}_{1\times(m-1)}\otimes\tau)\left((z-u^*Q^{-1}u)^{-1}u^*Q^{-1}\right)\\ 
-({\rm id}_{(m-1)\times1}\otimes\tau)\left(Q^{-1}u(z-u^*Q^{-1}u)^{-1}\right) & ({\rm id}_{m-1}\otimes\tau)\left(Q^{-1}+Q^{-1}u(z-u^*Q^{-1}u)^{-1}u^*Q^{-1}\right)\end{bmatrix},
\end{align*}
where $u,Q,u^*$ are the (obvous size) constituents of the linearization of $P$, evaluated in $(s,d)$. It is useful to note that the above matrix can be re-written as
\begin{align*}
G( ze_{11}&-\gamma_0)\\
&=\begin{bmatrix} 0 & 0\\ 
0 & ({\rm id}_{m-1}\otimes\tau)\left(Q^{-1}\right)\end{bmatrix}+({\rm id}_{m}\otimes\tau)\left[\begin{bmatrix} 1 \\ -Q^{-1}u\end{bmatrix}\!\!\begin{array}{l}
\left(z-u^*Q^{-1}u\right)^{-1}\begin{bmatrix} 1 & -u^*Q^{-1}\end{bmatrix}\\ 
\ 
\end{array}\right].
\end{align*}

It is shown in \cite[Lemma 4.2]{BBC} that there exist permutation matrices
$T_1,T_2\in M_{m-1}(\mathbb C)$ and a strictly lower triangular matrix $N\in M_{m-1}(\mathbb C\langle s,d\rangle)$ such that $Q^{-1}=-T_1(I_{m-1}+N)T_2$.
We multiply $G(ze_{11}-\gamma_0)$ left with $\begin{bmatrix} 1 & 0 \\ 0 & T_1^{-1} \end{bmatrix}$ and right with $\begin{bmatrix} 1 & 0 \\ 0 & T_2^{-1} \end{bmatrix}$
to get 
\begin{align*}
\begin{bmatrix} 0 & 0\\ 
0 & -I_{m-1}-({\rm id}_{m-1}\otimes\tau)(N)\end{bmatrix}& +\\
& ({\rm id}_{m}\otimes\tau)
\left[\begin{bmatrix} 1 \\ (I_{m-1}+N)T_2u\end{bmatrix}\!\!\begin{array}{l}
\left(z-u^*Q^{-1}u\right)^{-1}\begin{bmatrix} 1 & u^*T_1(I_{m-1}+N)\end{bmatrix}\\ 
\ 
\end{array}\right]\\
& \hspace{2cm} =\begin{bmatrix} 0 & 0\\ 
0 & -I_{m-1}-({\rm id}_{m-1}\otimes\tau)\left(N\right)\end{bmatrix}+{\rm O}\left(\frac1z\right).
\end{align*}
(The ${\rm O}\left(\frac1z\right)$ part can be easily made precise: it is $\frac1z({\rm id}_{m}\otimes\tau)\left[\begin{bmatrix} 1 \\ (I_{m-1}+N)T_2u\end{bmatrix}
\begin{bmatrix} 1 & u^*T_1(I_{m-1}+N)\end{bmatrix}\right]+{\rm O}\left(\frac{1}{z^2}\right)$.)
We recall that $Q=P^*_{m-1}(\gamma_0-\gamma_1-\gamma_2)P_{m-1}$, where we have denoted by $P_{m-1}$ the operator that embeds
$\mathbb C^{m-1}\hookrightarrow(0,\mathbb C^{m-1})\subset \mathbb C^{m}$. It is known that $T_2P^*_{m-1}\gamma_0P_{m-1}T_1=-I_{m-1},$ while $T_1P^*_{m-1}
(\gamma_1+\gamma_2)P_{m-1}T_1$ is strictly lower triangular (nilpotent). Moreover, the Schur product of $\gamma_1$ and $\gamma_2$ is known to have all but possibly
the $(1,1)$ entry equal to zero, and $T_1P^*_{m-1}\gamma_1P_{m-1}T_1$, $T_1P^*_{m-1}\gamma_2P_{m-1}T_1$ are both strictly lower triangular. In particular,
neither the structure of $N$ nor the presence of $I_{m-1}$ is affected if one replaces $\gamma_2$ with $t\gamma_2$ for an arbitrary $t\in\mathbb R$.
Thus,
$$
\begin{bmatrix} 1 & 0 \\ 0 & T_2 \end{bmatrix}(ze_{11}-\gamma_0-\gamma_1-t\gamma_2)\begin{bmatrix} 1 & 0 \\ 0 & T_1 \end{bmatrix}
=\begin{bmatrix} z+\text{const}_t & \mathfrak x_t^*T_1 \\ T_2\mathfrak x_t & -I_{m-1}+\tilde{N}_t \end{bmatrix},
$$
where $\tilde{N}_t$ is a strictly lower triangular matrix, some of whose entries might depend linearly of $t\in\mathbb R$, and the same holds for the column vector $\mathfrak x_t$.
The constant $\text{const}_t$ might be affine in $t$. It is trivial that the above matrix is invertible for all $z\not\in\mathbb R$, and equally obvious that $-I_{m-1}+\tilde{N}_t $
is invertible. Writing the Schur complement wrt the ($1,1$) entry, we obtain that $z+\text{const}_t-\mathfrak x_t^*T_1(\tilde{N}_t-I_{m-1})^{-1}T_2\mathfrak x_t$ is invertible.
We study how adding $\gamma_1-\gamma_1G(ze_{11}-\gamma_0)\gamma_1$ to $ze_{11}-\gamma_0-\gamma_1-t\gamma_2$ influences this invertibility. We write
\begin{align*}
\lefteqn{\begin{bmatrix} 1 & 0 \\ 0 & T_2 \end{bmatrix}(\gamma_1-\gamma_1G(ze_{11}-\gamma_0)\gamma_1)\begin{bmatrix} 1 & 0 \\ 0 & T_1 \end{bmatrix}}\\
& = \begin{bmatrix} 1 & 0 \\ 0 & T_2 \end{bmatrix}\gamma_1\begin{bmatrix} 1 & 0 \\ 0 & T_1 \end{bmatrix}
-\begin{bmatrix} 1 & 0 \\ 0 & T_2 \end{bmatrix}\gamma_1\begin{bmatrix} 1 & 0 \\ 0 & T_1 \end{bmatrix}\begin{bmatrix} 1 & 0 \\ 0 & T_1^{-1} \end{bmatrix}
G(ze_{11}-\gamma_0)\begin{bmatrix} 1 & 0 \\ 0 & T_2^{-1}\end{bmatrix}\begin{bmatrix} 1 & 0 \\ 0 & T_2 \end{bmatrix}\gamma_1\begin{bmatrix} 1 & 0 \\ 0 & T_1 \end{bmatrix}.
\end{align*}
As seen above, the lower right $(m-1)\times(m-1)$ corner of $\begin{bmatrix} 1 & 0 \\ 0 & T_2 \end{bmatrix}\gamma_1\begin{bmatrix} 1 & 0 \\ 0 & T_1 \end{bmatrix}$
is a sub-matrix of $\tilde{N}_t$, namely the constant entries; the same holds for the vectors above and to the left of this corner, and the $(1,1)$ entry is the constant
part of the affine map $t\mapsto$const${}_t$, all those with a minus in front.
It was seen above that $\begin{bmatrix} 1 & 0 \\ 0 & T_1^{-1} \end{bmatrix}
G(ze_{11}-\gamma_0)\begin{bmatrix} 1 & 0 \\ 0 & T_2^{-1}\end{bmatrix}=\begin{bmatrix} 0 & 0\\ 
0 & -I_{m-1}-({\rm id}_{m-1}\otimes\tau)\left(N\right)\end{bmatrix}+{\rm O}\left(\frac1z\right)$. It follows that 
\begin{align*}
\lefteqn{\begin{bmatrix} 1 & 0 \\ 0 & T_2 \end{bmatrix}\gamma_1\begin{bmatrix} 1 & 0 \\ 0 & T_1 \end{bmatrix}\begin{bmatrix} 1 & 0 \\ 0 & T_1^{-1} \end{bmatrix}
G(ze_{11}-\gamma_0)\begin{bmatrix} 1 & 0 \\ 0 & T_2^{-1}\end{bmatrix}\begin{bmatrix} 1 & 0 \\ 0 & T_2 \end{bmatrix}\gamma_1\begin{bmatrix} 1 & 0 \\ 0 & T_1 \end{bmatrix}}\\
& = \begin{bmatrix} \star & \star \\ \star & \mathfrak n \end{bmatrix}
\left(\begin{bmatrix} 0 & 0\\ 
0 & -I_{m-1}-({\rm id}_{m-1}\otimes\tau)(N)\end{bmatrix}\right.\\
& \quad +({\rm id}_{m}\otimes\tau)\left.
\left[\begin{bmatrix} 1 \\ (I_{m-1}+N)T_2u\end{bmatrix}\!\!\begin{array}{l}
\left(z-u^*Q^{-1}u\right)^{-1}\begin{bmatrix} 1 & u^*T_1(I_{m-1}+N)\end{bmatrix}\\ 
\ 
\end{array}\right]\right)\begin{bmatrix} \star & \star \\ \star & \mathfrak n \end{bmatrix}.
\end{align*}
As already mentioned, $\mathfrak n$ in the above is a constant sub-matrix of $\tilde{N}_t$, hence strictly lower triangular, and the stars stand for constant/constant vectors, whose
precise identity is irrelevant for our purposes. The second summand in the above is still O$\left(\frac1z\right)$, and the lower right corner of the first summand is 
$-\mathfrak n(I_{m-1}+N)\mathfrak n=-\mathfrak n^2-\mathfrak nN\mathfrak n$, a sum of products of strictly lower triangular matrices, is strictly lower triangular (in fact guaranteed 
to have also all entries $(i,i-1)$ equal to zero as well). It follows immediately that
\begin{eqnarray*}
\begin{bmatrix} 1 & 0 \\ 0 & T_2 \end{bmatrix}
(\omega(ze_{11}-\gamma_0)-t\gamma_2)\begin{bmatrix} 1 & 0 \\ 0 & T_1\end{bmatrix}
& = & \begin{bmatrix} z+\text{const}_t+\star & \mathfrak x_t^*T_1+\star \\ T_2\mathfrak x_t+\star & -I_{m-1}+\tilde{N}_t-\mathfrak n^2-\mathfrak nN\mathfrak n \end{bmatrix}
+{\rm O}\left(\frac1z\right).
\end{eqnarray*}
Since the lower right corner is invertible for all $z,t$, and the other $t$-dependent terms only depend affinely of it, it follows by the Schur complement that for $|z|$ sufficiently 
large the first term in the right-hand side is invertible. Since the second is ${\rm O}\left(\frac1z\right)$, the same Schur complement guarantees that, by slightly increasing
$|z$| if necessary, the invertibility statement remains valid. As the invertibility of $\omega(ze_{11}-\gamma_0)-t\gamma_2$ is equivalent to the invertibility of 
$\begin{bmatrix} 1 & 0 \\ 0 & T_2 \end{bmatrix}
(\omega(ze_{11}-\gamma_0)-t\gamma_2)\begin{bmatrix} 1 & 0 \\ 0 & T_1\end{bmatrix}$, we have established the invertibilty of $\omega(ze_{11}-\gamma_0)-t\gamma_2$ for
$|z$| sufficiently large. 

Since $\Im\omega(ze_{11}-\gamma_0)\ge0$ for all $z\in\mathbb C^+$ and $\gamma_2=\gamma_2^*$, it follows that $\Im(\omega(ze_{11}-\gamma_0)-t\gamma_2)^{-1}\le0$
whenever $\omega(ze_{11}-\gamma_0)-t\gamma_2$ is invertible. The argument employed in \cite[Lemma 5.5]{BBC} guarantees that $z\mapsto 
(\omega(ze_{11}-\gamma_0)-t\gamma_2)^{-1}$ extends analytically to $\mathbb C^+\cup\mathbb C^-$, in addition to the neighborhood of infinity on which we have
already shown it is well-defined. Since $M>\|d\|$ is arbitrary, the map $\mathbb C^\pm\times\mathbb R\ni(z,t)\mapsto(\omega(ze_{11}-\gamma_0)-t\gamma_2)^{-1}$ is well-defined 
and analytic. Thus, it is bounded on any compact subset of $\mathbb C^\pm\times\mathbb R$, and in particular on $K\times[-M,M]$.
\end{proof}

\begin{lem}\label{cvunif}
For any compact subset $K$ of $\C \setminus \R$,
$$\lim_{N\to +\infty}\sup_{z\in K}\sup_{t\in \text{supp}(\nu_N)}\|(\omega_N(ze_{11}-\gamma_0)-t\gamma_2)^{-1}-(\omega(ze_{11}-\gamma_0)-t\gamma_2)^{-1}\|=0.$$
\end{lem}
\begin{proof}
Let  $K\subset\mathbb C^+$ be a compact set. Let  $M>0$ be such that for all large $N$, $\mbox{support}(\nu_N)\subset[-M,M]$.
\begin{align}
\lefteqn{ \sup_{z\in K} \sup_{ t \in \text{supp}(\nu_N)}\|(\omega_N(ze_{11}-\gamma_0)-t\gamma_2)^{-1}-(\omega(ze_{11}-\gamma_0)-t\gamma_2)^{-1}\|}\label{uno}\\
&\leq \sup_{z\in K,t \in \text{supp}(\nu_N)}
\|(\omega_N(ze_{11}-\gamma_0)-t\gamma_2)^{-1}\|\cdot\sup_{z\in K,|t|\le M}\|(\omega(ze_{11}-\gamma_0)-t\gamma_2)^{-1}\|\nonumber\\
& \quad  \times \sup_{z\in K}\left\|\omega_N(ze_{11}-\gamma_0)-\omega(ze_{11}-\gamma_0)\right\|=0\nonumber.
\end{align}
Thus, Lemma \ref{cvunif} readily follows from Lemma \ref{bound}, \eqref{bornemanq}  and the uniform convergence of $\omega_N$ to $\omega$ on compact subsets of $\C^{\pm} e_{11} -\gamma_0$ (see Remark \ref{cvuniformeomegaN})
\end{proof}

\subsection{Concentration bounds on quadratic forms}

Applying Proposition \ref{Schur} to $\beta\otimes I_N-\gamma_1\otimes W_N-\gamma_2\otimes D_N$ 
leads to expressions involving random quadratic maps
$\gamma_1\otimes C_k^{(k)*}R^{(k)}(\beta)\gamma_1\otimes C_k^{(k)}$. 
It is easy to compute the expectation of such quadratic maps. 
(Recall that $\esp_k$ denotes the expectation with respect to $\{W_{ik},1 \leq i \leq N\}$.)
\[\mathbb{E}\Big[\gamma_1\otimes C_k^{(k)*}R^{(k)}(\beta)\gamma_1\otimes C_k^{(k)}\Big]=\mathbb{E}\Big[\mathbb{E}_k[\gamma_1\otimes C_k^{(k)*}R^{(k)}(\beta )\gamma_1\otimes C_k^{(k)}]\Big]=\eta_N(\mathbb{E}[({\rm id_m}\otimes N^{-1}\Tr)(R^{(k)}(\beta))]\big).\]
Their variance may be deduced from the following Lemma:

\begin{lem}\label{quadratic forms}
For $m(N-1)\times m(N-1)$ random matrices $A'=\sum_{i,j\neq k}\alpha_{ij}'\otimes E_{ij}, 
A''=\sum_{i,j\neq k}\alpha_{ij}''\otimes E_{ij}$, $m\times m$ random matrices $\beta', \beta'', \gamma',\gamma''$, 
all independent of $\{W_{ik},1\leq i\leq N\}$, and $h\in \{1,\ldots ,N\}$,
\begin{align*}
 \mathbb{E}_k \Big[\mathbb{E}_{\leq h}[\Tr\big(&(\gamma' \otimes C_k^{(k)*} A'\gamma' \otimes C_k^{(k)}-\gamma'({\rm id}_d\otimes \sigma_N^2\Tr)(A')\gamma')\beta'\big)] \\
  & \hspace{2cm} \times \mathbb{E}_{\leq h}[\Tr\big((\gamma'' \otimes C_k^{(k)*}A''\gamma'' \otimes C_k^{(k)}-\gamma'' ({\rm id}_d \otimes \sigma_N^2\Tr)(A'')\gamma'')\beta''\big)]\Big]\\
 & =\sigma_N^4\sum_{i,j\neq k\leq h}\mathbb{E}_{\leq h}[\Tr(\gamma' \alpha_{ij}'\gamma' \beta')]\mathbb{E}_{\leq h}[\Tr(\gamma'' \alpha_{ji}''\gamma'' \beta'')]\\
 & +|\theta_N|^2\sum_{i,j\neq k\leq h} \mathbb{E}_{\leq h}[\Tr(\gamma' \alpha_{ij}'\gamma' \beta')]\mathbb{E}_{\leq h}[\Tr(\gamma'' \alpha_{ij}''\gamma'' \beta'')]\\
 & +\kappa_N\sum_{i\neq k\leq h}\mathbb{E}_{\leq h}[\Tr(\gamma' \alpha_{ii}'\gamma' \beta')]\mathbb{E}_{\leq h}[\Tr(\gamma'' \alpha_{ii}''\gamma'' \beta'')] +\varepsilon_k,
\end{align*}
where
\[\varepsilon_k= \sum_{i\neq j\neq k\leq h}\Big(\mathbb{E}[\overline{W_{ki}}^2]\mathbb{E}[W_{kj}^2]-|\theta_N|^2\Big)\mathbb{E}_{\leq h}[\Tr(\gamma' \alpha_{ij}' \gamma' \beta')]\mathbb{E}_{\leq h}[\Tr(\gamma'' \alpha_{ij}''\gamma'' \beta'')].\]
Moreover, if $A'$ and $A''$ are bounded in $L^2$ and $\beta', \beta'', \gamma',\gamma''$ are bounded deterministic, then
\[\sum\limits_{k=1}^N\varepsilon_k \overset{\P}{\underset{N \to +\infty}{\longrightarrow}} 0.\]
\end{lem}

\begin{proof}[Proof of Lemma \ref{quadratic forms}]
It is sufficient to prove the result for $A'=\alpha'\otimes M', A''=\alpha''\otimes M''$.
Since
\begin{align*}
 \mathbb{E}_{\leq h}& \Big[\Tr\big((\gamma' \otimes C_k^{(k)*}\alpha'\otimes M'\gamma' \otimes C_k^{(k)}-\gamma'({\rm id}_d\otimes \sigma_N^2\Tr)(\alpha'\otimes M')\gamma' )\beta'\big)\Big]\\
& =\mathbb{E}_{\leq h}[\Tr\big((C_k^{(k)*}M'C_k^{(k)}-\sigma_N^2\Tr(M'))\gamma' \alpha' \gamma' \beta'\big)]\\
& =\sum_{i'\neq j'\neq k\leq h}\overline{W_{ki'}}W_{kj'}\mathbb{E}_{\leq h}[\Tr(\gamma' M_{i'j'}'\alpha' \gamma' \beta')]+\sum_{i'\neq k\leq h}(|W_{ki'}|^2-\sigma_N^2)\mathbb{E}_{\leq h}[\Tr(\gamma' M_{i'i'}'\alpha' \gamma' \beta')]\\
& =\sum_{i'\neq j'\neq k\leq h}\overline{W_{ki'}}W_{kj'}\mathbb{E}_{\leq h}[\Tr(\gamma' \alpha_{i'j'}' \gamma' \beta')]+\sum_{i'\neq k\leq h}(|W_{ki'}|^2-\sigma_N^2)\mathbb{E}_{\leq h}[\Tr(\gamma' \alpha_{i'i'}' \gamma' \beta')]
\end{align*}
and
\begin{align*}
 \mathbb{E}_{\leq h}& \Big[\Tr\big((\gamma'' \otimes C_k^{(k)*}\alpha''\otimes M''\gamma'' \otimes C_k^{(k)}-\gamma'' (id\otimes \sigma_N^2\Tr)(\alpha''\otimes M'')\gamma'')\beta''\big)\Big]\\
 & =\sum_{i''\neq j''\neq k\leq h}\overline{W_{ki''}}W_{kj''}\mathbb{E}_{\leq h}[\Tr(\gamma'' \alpha_{i''j''}'' \gamma'' \beta'')]+\sum_{i''\neq k\leq h}(|W_{ki''}|^2-\sigma_N^2)\mathbb{E}_{\leq h}[\Tr(\gamma'' \alpha_{i''i''}'' \gamma'' \beta'')],
\end{align*}
it follows that 
\begin{align*}
 \mathbb{E}_k\Big[\mathbb{E}_{\leq h}[\Tr\big(& (\gamma' \otimes C_k^{(k)*}A'\gamma' \otimes C_k^{(k)}-\gamma'({\rm id}_d\otimes \sigma_N^2\Tr)(A')\gamma')\beta'\big)]\\
 & \hspace{2cm} \times \mathbb{E}_{\leq h}[\Tr\big((\gamma'' \otimes C_k^{(k)*}A''\gamma'' \otimes C_k^{(k)}-\gamma'' ({\rm id}_d\otimes \sigma_N^2\Tr)(A'')\gamma'')\beta''\big)]\Big]\\
& =\sum_{i\neq j\neq k\leq h}\mathbb{E}[|W_{ki}|^2]\mathbb{E}[|W_{kj}|^2]\mathbb{E}_{\leq h}[\Tr(\gamma' \alpha_{ij}' \gamma' \beta')]\mathbb{E}_{\leq h}[\Tr(\gamma'' \alpha_{ji}''\gamma'' \beta'')]\\
& +\sum_{i\neq j\neq k\leq h}\mathbb{E}[\overline{W_{ki}}^2]\mathbb{E}[W_{kj}^2]\mathbb{E}_{\leq h}[\Tr(\gamma' \alpha_{ij}' \gamma' \beta')]\mathbb{E}_{\leq h}[\Tr(\gamma'' \alpha_{ij}''\gamma'' \beta'')]\\
& +\sum_{i\neq k\leq h}\mathbb{E}\big[(|W_{ki}|^2-\sigma_N^2)^2\big]\mathbb{E}_{\leq h}[\Tr(\gamma' \alpha_{ii}' \gamma' \beta')]\mathbb{E}_{\leq h}[\Tr(\gamma'' \alpha_{ii}''\gamma'' \beta'')].
\end{align*}
Note that, for $i \neq j$, $\mathbb{E}[\overline{W_{ki}}^2]\mathbb{E}[W_{kj}^2]=|\theta_N|^2$ if $i$ and $j$ are both smaller than $k$ or greater than $k$. If it is not the case, this term equals $\theta_N^2$ or $\bar{\theta}_N^2$, according to whether $i<j$ or $i>j$. By Assumption \ref{hyp:offdiagonal} and the fact that $\lim N\theta_N \in \R$, $\Im \theta_N=o(N^{-1})$ and $\theta_N^2=|\theta_N|^2+o(N^{-2})$ (and so does $\bar{\theta}_N^2$). Therefore
\begin{align*}
\mathbb{E}_k\big[\mathbb{E}_{\leq h}[\Tr \big((\gamma' \otimes C_k^{(k)*}A'\gamma' & \otimes C_k^{(k)} -\gamma'(id\otimes \sigma_N^2\Tr)(A')\gamma')\beta'\big)]\\
& \times \mathbb{E}_{\leq h}[\Tr\big((\gamma'' \otimes C_k^{(k)*}A''\gamma'' \otimes C_k^{(k)}-\gamma'' (id\otimes \sigma_N^2\Tr)(A'')\gamma'')\beta''\big)]\big]\\
& \hspace{1cm} = \sigma_N^4\sum_{i,j\neq k\leq h}\mathbb{E}_{\leq h}[\Tr(\gamma' \alpha_{ij}' \gamma' \beta')]\mathbb{E}_{\leq h}[\Tr(\gamma'' \alpha_{ji}''\gamma'' \beta'')]\\
& \hspace{1cm} + |\theta_N|^2\sum_{i,j\neq k\leq h}\mathbb{E}_{\leq h}[\Tr(\gamma' \alpha_{ij}' \gamma' \beta')]\mathbb{E}_{\leq h}[\Tr(\gamma'' \alpha_{ij}''\gamma'' \beta'')]\\
& \hspace{1cm} + \kappa_N\sum_{i\neq k\leq h}\mathbb{E}_{\leq h}[\Tr(\gamma' \alpha_{ii}' \gamma' \beta')]\mathbb{E}_{\leq h}[\Tr(\gamma'' \alpha_{ii}''\gamma'' \beta'')] + \varepsilon_k,
\end{align*}
where $\varepsilon_k$ has the expected expression.
		
In order to show that $\sum_{k=1}^N\varepsilon_k \to 0$ in probability, we consider the $L^1$ norm $\|\varepsilon_k\|_1$. We denote by $\varepsilon_{ijk}$ the quantity $\mathbb{E}[\overline{W_{ki}}^2]\mathbb{E}[W_{kj}^2]-|\theta_N|^2$.
\begin{align*}
\|\varepsilon_k\|_1
& \leq d^2\|\gamma'\|^2\|\gamma''\|^2\|\beta'\|\|\beta''\|\sup_{i\neq j\neq k\leq h}|\varepsilon_{ijk}|\sum_{i\neq j\neq k\leq h}\mathbb{E}\big[\|\alpha'_{ij}\|\|\alpha''_{ij}\|\big]\\
& \leq \frac{d^2}{2}\|\gamma'\|^2\|\gamma''\|^2\|\beta'\|\|\beta''\|\sup_{i\neq j\neq k\leq h}|\varepsilon_{ijk}|\sum_{i\neq j\neq k\leq h}\mathbb{E}\big[\|\alpha'_{ij}\|^2+\|\alpha''_{ij}\|^2\big]\\
& \leq \frac{d^3}{2}\|\gamma'\|^2\|\gamma''\|^2\|\beta'\|\|\beta''\|\sup_{i\neq j\neq k\leq h}|\varepsilon_{ijk}|(N-1)\mathbb{E}\big[\|A'\|^2+\|A''\|^2\big]\\
\end{align*}
using Lemma \ref{prelim}. Recall that $\|\gamma'\|$, $\|\gamma''\|$, $\|\beta'\|$, $\|\beta''\|$, $\mathbb{E}[\|A'\|^2]$ and $\mathbb{E}[\|A''\|^2]$ are bounded. Together with $\sup_{i\neq j\neq k\leq h}|\varepsilon_{ijk}|=o(N^{-2})$, it leads to $\|\varepsilon_k\|_1 =o(N^{-1})$, uniformly in $k$. Therefore, as claimed,
\[\sum_{k=1}^N\varepsilon_k \overset{\P}{\underset{N \to +\infty}{\longrightarrow}} 0.\]
\end{proof}

\subsection{Bounds on $R_N$ and $R^{(k)}$}

For $\beta \in \mathbb{H}^+(M_m(\mathbb{C}))$ and $k=1,\ldots ,N$, by a direct application of \eqref{majim}, 
\begin{equation}\label{majRNRk}
\Vert R_N(\beta)\Vert \leq \Vert(\Im \beta)^{-1}\Vert, \quad \Vert R^{(k)}(\beta)\Vert \leq \Vert(\Im \beta)^{-1}\Vert.
\end{equation}
It is a consequence of Lemma \ref{inversible} that $R_N$ and $R^{(k)}$ are also defined on $\{ze_{11}-\gamma_0, z\in \mathbb{C}\setminus \mathbb{R}\}$.

\begin{lem}\label{majRNRkLp}
For 
$p\geq 1$, $\mathbb{E}[\Vert R_N(ze_{11}-\gamma_0)\Vert^p]=O(1).$
\end{lem}

\begin{proof}
Using Lemma \ref{alta lemma} with $y=(W_N,D_N)$, one gets 
\[\Vert R_N(ze_{11}-\gamma_0)\Vert\leq Q_1(\Vert W_N\Vert, \Vert D_N\Vert)\Vert (zI_N-X_N)^{-1}\Vert+Q_2(\Vert W_N\Vert, \Vert D_N\Vert).\]
Then, using Assumption \ref{hyp:deformation} and the bound \eqref{resolventbound}, there is a polynomial $Q$ such that 
\[\Vert R_N(ze_{11}-\gamma_0)\Vert\leq (1+\vert \Im z\vert^{-1})Q(\Vert W_N\Vert).\]
It follows from Proposition \ref{boundnorm} that $Q(\Vert W_N\Vert)$
is bounded in all $L^p,\, p\geq 1$.
\end{proof}

\begin{rmk}\label{remmajRNRkLp}
The same argument with $y=(W_N^{(k)},D_N^{(k)})$ (and the observation that $\|W_N^{(k)}\|\leq \|W_N\|$ and $\|D_N^{(k)}\|\leq \|D_N\|$) proves that 
$\mathbb{E}[\Vert R^{(k)}(ze_{11}-\gamma_0)\Vert^p]=O(1).$
With $y=(W_N^{(kab)},D_N^{(k)})$ (and the observation that $\|W_N^{(kab)}\|\leq \|W_N\|+2\delta_N$ and $\|D_N^{(k)}\|\leq \|D_N\|$), we get that 
$\mathbb{E}[\Vert R^{(kab)}(ze_{11}-\gamma_0)\Vert^p]=O(1)$, uniformly in $a,b\neq k$.
\end{rmk}

By Proposition \ref{Schur}, for $z\in \mathbb{C}\setminus \mathbb{R}$, $ze_{11}-\gamma_0-W_{kk}\gamma_1-D_{kk}\gamma_2-\gamma_1\otimes C_k^{(k)*}R^{(k)}(ze_{11}-\gamma_0)\gamma_1\otimes C_k^{(k)}$ is invertible, 
$$(ze_{11}-\gamma_0-W_{kk}\gamma_1-D_{kk}\gamma_2-\gamma_1\otimes C_k^{(k)*}R^{(k)}(ze_{11}-\gamma_0)\gamma_1\otimes C_k^{(k)})^{-1}=R_{kk}(ze_{11}-\gamma_0)$$
and therefore 
\begin{equation}\label{InverseSchurComplementBound}
\Vert (ze_{11}-\gamma_0-W_{kk}\gamma_1-D_{kk}\gamma_2-\gamma_1\otimes C_k^{(k)*}R^{(k)}(ze_{11}-\gamma_0)\gamma_1\otimes C_k^{(k)})^{-1}\Vert \leq \Vert R_N(ze_{11}-\gamma_0)\Vert.
\end{equation}
Note that the same argument implies the bound
$$\Vert (\beta-W_{kk}\gamma_1-D_{kk}\gamma_2-\gamma_1\otimes C_k^{(k)*}R^{(k)}(\beta)\gamma_1\otimes C_k^{(k)})^{-1}\Vert \leq \Vert R_N(\beta)\Vert \leq \|(\Im \beta)^{-1}\|$$
for $\beta \in \mathbb{H}^+(M_m(\mathbb{C}))$.

\begin{lem}\label{estimfq}
For 
$p\in [2,4(1+\varepsilon)]$, 
\[\mathbb{E}[\Vert \gamma_1\otimes C_k^{(k)*}R^{(k)}(ze_{11}-\gamma_0)\gamma_1\otimes C_k^{(k)}-\sigma_N^2\gamma_1(\mathrm{id}_m\otimes \Tr R^{(k)}(ze_{11}-\gamma_0))\gamma_1]\Vert^p]=O(N^{-p/2}).\]
\end{lem}
\begin{proof}
The result easily follows from Lemmas \ref{lem_moment_quadratic_forms}, \ref{puissance} and Remark \ref{remmajRNRkLp}.
\end{proof}

\begin{lem}\label{compkLp}
For 
$p\geq 1$, $$\mathbb{E}\left[\|\mathrm{id}_m\otimes \Tr R^{(k)}(ze_{11}-\gamma_0)-\mathrm{id}_m\otimes \Tr R_N (ze_{11}-\gamma_0)\|^p\right]=O(1),$$
\begin{equation}\label{derive}\mathbb{E}\left[\|\frac{\partial}{\partial z}(\mathrm{id}_m\otimes \Tr R^{(k)}(ze_{11}-\gamma_0)-\mathrm{id}_m\otimes \Tr R_N (ze_{11}-\gamma_0))\|^p\right]=O(1).\end{equation}
\end{lem}

\begin{proof}
For $z\in \mathbb{C}\setminus\mathbb{R}$, by Proposition \ref{Schur}, since $\mathrm{id}_m\otimes \Tr R=\mathrm{id}_m\otimes\Tr(R_I + R_{I^c})$ with $I^c=\{k, k+N,\ldots, k+(N-1)m\}$,\\
$\mathrm{id}_m\otimes \Tr R_N (ze_{11}-\gamma_0)-\mathrm{id}_m\otimes \Tr R^{(k)}(ze_{11}-\gamma_0)$ $$=s_k^{-1}+\mathrm{id}_m\otimes \Tr (R^{(k)}(ze_{11}-\gamma_0)\gamma_1\otimes C_k^{(k)}s_k^{-1}\gamma_1\otimes C_k^{(k)*}R^{(k)}(ze_{11}-\gamma_0)),$$
where $s_k=ze_{11}-\gamma_0-W_{kk}\gamma_1-D_{kk}\gamma_2-\gamma_1\otimes C_k^{(k)*}R^{(k)}(ze_{11}-\gamma_0)\gamma_1\otimes C_k^{(k)}$.
Observe that 
\[\Vert\mathrm{id}_m\otimes \Tr (R^{(k)}(ze_{11}-\gamma_0)\gamma_1\otimes C_k^{(k)}s_k^{-1}\gamma_1\otimes C_k^{(k)*}R^{(k)}(ze_{11}-\gamma_0))\Vert \leq m^4\Vert \gamma_1\Vert^2\Vert R^{(k)}(ze_{11}-\gamma_0)\Vert^2\Vert C_k^{(k)}\Vert^2\Vert s_k^{-1}\Vert,\]
so that 
\begin{align*}
\|\mathrm{id}_m\otimes \Tr R^{(k)}(ze_{11}-\gamma_0) - \mathrm{id}_m & \otimes \Tr R_N (ze_{11}-\gamma_0)\|\\
&\leq \Vert s_k^{-1}\Vert(1+m^4\Vert \gamma_1\Vert^2\Vert R^{(k)}(ze_{11}-\gamma_0)\Vert^2\Vert C_k^{(k)}\Vert^2)\\
&\leq \Vert R_N(ze_{11}-\gamma_0)\Vert\Vert(1+m^4\Vert \gamma_1\Vert^2\Vert R^{(k)}(ze_{11}-\gamma_0)\Vert^2\Vert W_N\Vert^2)
\end{align*}
is bounded in all $L^p,\, p\geq 1$, by Lemma \ref{majRNRkLp}, Remark \ref{remmajRNRkLp} and Proposition \ref{boundnorm}. Similarly one can prove \eqref{derive}.
\end{proof}

\begin{lem}\label{crudevariancebound}
For 
any even integer $p\in \mathbb{N}$, 
\[\mathbb{E}[\| \mathrm{id}_m\otimes \Tr R_N (ze_{11}-\gamma_0)-\mathbb{E}[\mathrm{id}_m\otimes \Tr R_N (ze_{11}-\gamma_0)]\|^{p}]=O(N^{p-1}),\]
\[\mathbb{E}[\|\frac{\partial}{\partial z}\left[ \mathrm{id}_m\otimes \Tr R_N (ze_{11}-\gamma_0)-\mathbb{E}(\mathrm{id}_m\otimes \Tr R_N (ze_{11}-\gamma_0))\right]\|^{p}]=O(N^{p-1}).\]
\end{lem}

\begin{proof}
Observe that $M_k:=\mathbb{E}_{\leq k}[\mathrm{id}_m\otimes \Tr R_N (ze_{11}-\gamma_0)]$, $k\geq 0$, satisfies
\begin{eqnarray*}
M_k-M_{k-1}
&=&(\mathbb{E}_{\leq k}-\mathbb{E}_{\leq k-1})[\mathrm{id}_m\otimes \Tr R_N (ze_{11}-\gamma_0)]\\
&=&(\mathbb{E}_{\leq k}-\mathbb{E}_{\leq k-1})[\mathrm{id}_m\otimes \Tr R_N (ze_{11}-\gamma_0)-\mathrm{id}_m\otimes \Tr R^{(k)}(ze_{11}-\gamma_0)]
\end{eqnarray*} 
so that, using Jensen's inequality, 
\[\mathbb{E}[\Vert M_k-M_{k-1}\Vert^{p}]\leq 2^{p}\mathbb{E}[\Vert \mathrm{id}_m\otimes \Tr R_N (ze_{11}-\gamma_0)-\mathrm{id}_m\otimes \Tr R^{(k)}(ze_{11}-\gamma_0)\Vert^{p}]=O(1)\]
by Lemma \ref{compkLp}. 
Apply then Lemma \ref{martingalelp} to the martingale $(M_k=\mathbb{E}_{\leq k}[\mathrm{id}_m\otimes \Tr R_N (ze_{11}-\gamma_0)])_{k\geq 0}$ to obtain the first statement. Similarly, one can obtain the second statement by considering $M_k:=\mathbb{E}_{\leq k}[\frac{\partial}{\partial z}\mathrm{id}_m\otimes \Tr R_N (ze_{11}-\gamma_0)]$.
\end{proof}

\begin{rmk}\label{crudevarianceboundk}
Using Lemma \ref{compkLp}, one may deduce from the Lemma \ref{crudevariancebound} that, for 
any even integer $p\in \mathbb{N}$, 
\[\mathbb{E}[\| \mathrm{id}_m\otimes \Tr R^{(k)} (ze_{11}-\gamma_0)-\mathbb{E}[\mathrm{id}_m\otimes \Tr R^{(k)} (ze_{11}-\gamma_0)]\|^{p}]=O(N^{p-1}).\]
\end{rmk}

\subsection{Qualitative asymptotic freeness}

We first prove the asymptotic freeness of $W_N$ and $D_N$. \begin{lem}\label{asymptoticfreeness}
For any polynomial in two noncommuting indeterminates $H\in \mathbb{C}\langle t_1,t_2\rangle$, 
\begin{equation}\label{Cvpol}
\mathbb{E}[N^{-1}\Tr\left(H\left(W_N,D_{N}\right) \right)]\longrightarrow_{N \rightarrow \infty} \tau (H(s,d)). \end{equation}
\end{lem}

\begin{proof}
One may assume, without loss of generality, that $\sigma_N^2=\sigma^2N^{-1}$. Define a $N\times N$ random matrix $A_N$ by $A_{ij}=W_{ij}$ when $i\neq j$ and $A_{ii}=\sigma_N\tilde{\sigma}_N^{-1}W_{ii}$.
It is straightforward that
$$\|W_N-A_N\|\leq |1-\sigma_N\tilde{\sigma}_N^{-1}|\delta_N\rightarrow_{N\rightarrow +\infty} 0.$$
Since $\sup_N \|D_N\| <+\infty$ and, by Proposition \ref{boundnorm}, $(\|W_N\|)_N$ and $(\|A_N\|)_N$ are bounded in all $L^p,\, p\geq 1$, it follows that 
$$|\mathbb{E}[N^{-1}\Tr\left(H\left(W_N,D_{N}\right) \right)]-\mathbb{E}[N^{-1}\Tr\left(H\left(A_N,D_{N}\right) \right)]|$$
$$\leq \mathbb{E}[\|H(W_N,D_{N})-H(A_N,D_N)\| ]\longrightarrow_{N \rightarrow +\infty} 0.$$
Since $\mathbb{E}[A_{ij}]=0$,  
$\mathbb{E}[|A_{ij}|^2]=\sigma^2N^{-1}$ and, for any $m>2$,
$\sup_{i,j\leq N} \mathbb{E}[|A_{ij}|^m]=o(N^{-1})$,
one can apply Theorem 1 of \cite{Ry98}, and obtain that 
$\mathbb{E}[N^{-1}\Tr H(A_N,D_N)]$ converges towards $\tau(H(s,d))$.
\end{proof}

The following result is a consequence of the above asymptotic freeness.

\begin{lem}\label{asymptoticfreenessL1}For $z\in \mathbb{C}\setminus \mathbb{R}$, 
$$ \mathbb{E}[\left(\mathrm{id}_m \otimes N^{-1}\Tr\right)(R_{N}(ze_{11}-\gamma_0))]-\left(\mathrm{id}_m \otimes \tau_N\right)(r_N(ze_{11}-\gamma_0))\longrightarrow_{N \rightarrow +\infty} 0.$$
$$\frac{\partial}{\partial z}\left( \mathbb{E}[\left(\mathrm{id}_m \otimes N^{-1}\Tr\right)(R_{N}(ze_{11}-\gamma_0))]-\left(\mathrm{id}_m \otimes \tau_N\right)(r_N(ze_{11}-\gamma_0))\right)\longrightarrow_{N \rightarrow +\infty} 0.$$
\end{lem}

\begin{proof} 
Let $L_P = \begin{pmatrix} 0 & u\\v & Q \end{pmatrix}$ be the canonical linearization of $P$. Remember that the entries of  the row vector $u$, the column vector $v$ and the matrix $Q^{-1}$  are all polynomials and that $P=-uQ^{-1}v$.
It readily follows from Proposition \ref{Schur} applied to 
$z-L_P$
, that for each $p,q\in \{1,\ldots,m\}$, there exist two polynomials $H_1^{(p,q)}$ and $H_2^{(p,q)}$ such that, for any $z\in \mathbb{C} \setminus \mathbb{R}$,
for any $N$, the entry $(p,q)$ of the $m\times m$ matrix $\left(\mathrm{id}_m \otimes \tau_N\right) \left[r_N(ze_{11} - \gamma_0)  \right]$
(respectively of the $m\times m$ matrix $\left( \mathrm{id}_m \otimes N^{-1}\Tr\right) \left[R_N(ze_{11} - \gamma_0)  \right]$)
is equal to $\tau_N\left((z1_{\mathcal{A}_N}-P(s_N,D_N))^{-1} H_1^{(p,q)}(s_N,D_N) + H_2^{(p,q)}(s_N,D_N)\right)$ (respectively of the $m\times m$ matrix
$N^{-1}\Tr\left((zI_N-P(W_N,D_N))^{-1} H_1^{(p,q)}(W_N,D_N) + H_2^{(p,q)}(W_N,D_N)\right)$). 
		
We have for any selfadjoint operators $y_1$ and $y_2$, for any $z\in \mathbb{C} \setminus \mathbb{R}$, for any nonzero integer $r$, 
\begin{equation} \label{appol} (z-P(y_1,y_2))^{-1}= \sum_{k=0}^{r-1} z^{-1} ( z^{-1} P(y_1,y_2) )^k +    \left(z-P(y_1,y_2) \right)^{-1} (z^{-1} P(y_1,y_2))^r.\end{equation}
For any $K>0$, define $$\mathcal{O}_K = \{ z\in \mathbb{C} \setminus \mathbb{R}, \Im(z) > K \}.$$
Let $0<c<1$. For any $\kappa >0$, there exists  $K=K(\kappa, P,c)>0$  such that if $z \in {\mathcal O}_K$, for any $y_1$ and $y_2$ such that $\Vert y_1\Vert \leq \kappa$ and $\Vert y_2 \Vert \leq \kappa $ then  \begin{equation}\label{defK} \Vert  (z^{-1} P(y_1,y_2))\Vert \leq c,  \end{equation} 
so that
\begin{equation}\label{limr}\sup_{z \in {\mathcal O}_K} \left\|   \left(z-P(y_1,y_2) \right)^{-1} (z^{-1} P(y_1,y_2))^r H_1^{(p,q)}(y_1,y_2)\right\| \leq c'\frac{c^r}{K} \rightarrow_{r \rightarrow +\infty} 0.\end{equation}
\noindent Fix $K>0$ such that \eqref{defK} holds for  
$(y_1,y_2)=(s_N,D_N)$ and  $(y_1,y_2)=(W_N,D_{N})$ on $ E_{N}=\{\Vert W_N\Vert \leq C\}$ where $C$ is defined in Proposition \ref{boundnorm}. 
Using Proposition \ref{boundnorm}, it readily follows from \eqref{Cvpol} that 
for any polynomial $H$ in two noncommuting indeterminates, 
 \begin{equation}\label{Cvpolindicatrice}
\mathbb{E}[\1_{ E_{N}}N^{-1}\Tr\left(H\left(W_N,D_{N}\right) \right)]\longrightarrow_{N \rightarrow +\infty}\tau(H(s,d)). \end{equation}
By the convergence in noncommutative distribution of $(s_N,D_N)$ to $(s,d)$, we have that, for any polynomial $H$ in two noncommuting indeterminates,  \begin{equation}\label{Cvpol2}
\tau_N\left(H(s_N, d_{N})\right)\longrightarrow_{N \rightarrow +\infty}\tau\left(H(s,d)\right). \end{equation}
Using \eqref{appol}, \eqref{limr}, \eqref{Cvpolindicatrice},   \eqref{Cvpol2}, letting $N$ and then $r$ go to infinity, we obtain that for any $z\in  {\mathcal O}_K$, $$\mathbb{E}[\1_{ E_{N}}\left(\mathrm{id}_m \otimes N^{-1}\Tr \right)(R_{N}(ze_{11}-\gamma_0))]-\left(\mathrm{id}_m \otimes \tau_N\right)(r_N(ze_{11}-\gamma_0))\longrightarrow_{N \rightarrow +\infty} 0,$$
and then, using Lemma \ref{majRNRkLp}  and  $\mathbb{P}\left( E_N^c\right) \rightarrow_{N\rightarrow +\infty} 0,$ we readily deduce that, for any $z\in  {\mathcal O}_K$, $$
\mathbb{E}[\left(\mathrm{id}_m \otimes N^{-1}\Tr\right)(R_{N}(ze_{11}-\gamma_0))]-\left(\mathrm{id}_m \otimes \tau_N\right)(r_N(ze_{11}-\gamma_0))\longrightarrow_{N \rightarrow \infty} 0.$$
Functions $\Phi_N(z)=\mathbb{E}[\left(\mathrm{id}_m \otimes N^{-1}\Tr\right)(R_{N}(ze_{11}-\gamma_0))]-\left(\mathrm{id}_m\otimes \tau_N\right)(r_N(ze_{11}-\gamma_0)),\, N\in \mathbb{N},$ are holomorphic on $\mathbb{C}^+$. Moreover, using Lemma \ref{alta lemma} and Proposition \ref{boundnorm}, there exists a polynomial $Q$ 
such that, \begin{equation}\label{cvdomine} \; 
|| \Phi_N(z)|| \leq Q((\Im z)^{-1}),\quad z\in \mathbb{C}^+.\end{equation}
It readily follows that $(\Phi_N)_{N\in \mathbb{N}}$ is a bounded sequence in the set of analytic functions on $\mathbb{C}^+$ endowed with the uniform convergence on compact subsets. 
We can apply Vitali's theorem to conclude that the convergences of $(\Phi_N)_{N\in \mathbb{N}}$ and $(\partial_z\Phi_N)_{N\in \mathbb{N}}$ to $0$ hold on $\mathbb{C}^+$. Of course, this convergence similarly holds on $\mathbb{C}^-$. 
The proof of Lemma \ref{asymptoticfreenessL1} is complete.
\end{proof}

\begin{rmk}\label{cvuniformeomegaN}
Note that, following   the strategy of the proof of Lemma \ref{asymptoticfreenessL1}, we obtain the uniform convergence of $\omega_N$
to $\omega$ on compact subsets of $\C^{\pm} e_{11} -\gamma_0$.
\end{rmk}

As a consequence of Lemma \ref{crudevariancebound}  and  Lemma \ref{asymptoticfreenessL1}, we obtain the following
\begin{lem}\label{asymptoticfreenessLp}
For $z\in \mathbb{C}\setminus \mathbb{R}$ and $p\geq 1$, 
\begin{equation} \label{CVLpaf} \mathbb{E}[\|\mathrm{id}_m\otimes N^{-1}\Tr (R_N(ze_{11}-\gamma_0))-\mathrm{id}_m\otimes \tau_N(r_N(ze_{11}-\gamma_0))\|^p]\underset{N\to +\infty}{\longrightarrow}0,\end{equation}
\begin{equation} \label{CVLpaf2}\mathbb{E}[\|\frac{\partial}{\partial z} \left[ \mathrm{id}_m\otimes N^{-1}\Tr (R_N(ze_{11}-\gamma_0))-\mathrm{id}_m\otimes \tau_N(r_N(ze_{11}-\gamma_0))\right]\|^p]\underset{N\to +\infty}{\longrightarrow}0.\end{equation}
\end{lem}

\begin{lem}\label{boundomegatilde}
\begin{equation*}
\Vert (\Omega_N(ze_{11}-\gamma_0)\otimes I_N-\gamma_2\otimes D_N)^{-1}\Vert 
=O(1)\end{equation*}
\end{lem}

\begin{proof}
It follows from Lemma \ref{Qbound} that $\omega_N(ze_{11}-\gamma_0)\otimes I_N-\gamma_2\otimes D_N$ is invertible for every $z\in \mathbb{C}\setminus \mathbb{R}$ and
\begin{equation*}
\Vert (\omega_N(ze_{11}-\gamma_0)\otimes I_N-\gamma_2\otimes D_N)^{-1}\Vert =O(1).\end{equation*}
One deduces from Lemma 
\ref{asymptoticfreenessL1} that 
\[\Vert(\Omega_N(ze_{11}-\gamma_0)\otimes I_N-\gamma_2\otimes D_N)-(\omega_N(ze_{11}-\gamma_0)\otimes I_N-\gamma_2\otimes D_N)\Vert \underset{N\to +\infty}{\longrightarrow}0.\]
Note that this convergence is uniform on compact subsets of $\mathbb{C}\setminus\mathbb{R}$. Hence, it follows from Lemma \ref{inversionsomme} that, for any compact subset $K$ of $\mathbb{C}\setminus\mathbb{R}$, $\Omega_N(ze_{11}-\gamma_0)\otimes I_N-\gamma_2\otimes D_N,\, z\in K,$ are all invertible for large enough $N$ and  $$\Vert (\Omega_N(ze_{11}-\gamma_0)\otimes I_N-\gamma_2\otimes D_N)^{-1}\Vert=O(1).$$
\end{proof}

\begin{rmk}\label{remboundomegatilde}
By a similar argument based on Lemma \ref{inversionsomme}, one may deduce from Lemma \ref{boundomegatilde} and Lemma \ref{compkLp} 
that \[\Big\| \left( (ze_{11}-\gamma_0-\eta_N(\mathbb{E}[(\mathrm{id}_m\otimes N^{-1}\Tr)(R^{(k)}(ze_{11}-\gamma_0))])\otimes I_N-\gamma_2\otimes D_N\right)^{-1} \Big\|=O(1).\]
\end{rmk}

\begin{lem}\label{maxtendzero}
For $z\in \mathbb{C}\setminus\mathbb{R}$ and $p\geq 1$, 
\[\max_{k=1,\ldots ,N}\mathbb{E}[\| R_{kk}(ze_{11}-\gamma_0)-\hat{R}_k(ze_{11}-\gamma_0)\|^p]\underset{N\to +\infty}{\longrightarrow}0,\]
where $\hat{R}_k$ is defined by \eqref{defhatRk}.
\end{lem}

\begin{proof}
With the notation $\beta =ze_{11}-\gamma_0$,
and 
\begin{equation}\label{hatR}
\hat R_k(\beta)=(\omega_N(\beta) -D_{kk}\gamma_2)^{-1},
\end{equation} 
by Proposition \ref{Schur}, since for any $k \in \{1,\ldots,N\}$, $R_{kk}(\beta)$ is the submatrix of $R_N(\beta)$ corresponding to rows and columns indexed by $\{k, k+m, \ldots, k+m(N-1)\}$,
\begin{align*}
R_{kk}(\beta)
& = (\beta-W_{kk}\gamma_1-D_{kk}\gamma_2-\gamma_1\otimes C_k^{(k)*}R^{(k)}(\beta)\gamma_1\otimes C_k^{(k)})^{-1}\\
& = \hat R_{k}(\beta)+\hat R_{k}(\beta)(W_{kk}\gamma_1+\gamma_1\otimes C_k^{(k)*}R^{(k)}(\beta)\gamma_1\otimes C_k^{(k)}-\eta_N((\mathrm{id}_m\otimes \tau_N)(r_N(\beta))))R_{kk}(\beta).
\end{align*}
Observe that, for $n\geq 2$,
\begin{align*}
  3^{-n+1}\mathbb{E}& [\Vert W_{kk}\gamma_1+\gamma_1\otimes C_k^{(k)*}R^{(k)}(\beta)\gamma_1\otimes C_k^{(k)}-\eta_N((\mathrm{id}_m\otimes \tau_N)(r_N(\beta)))\Vert^{n}]\\
 & \leq \Vert \gamma_1\Vert^{n}\mathbb{E}[\vert W_{kk}\vert^{n}]+\mathbb{E}[\Vert \gamma_1\otimes C_k^{(k)*}R^{(k)}(\beta)\gamma_1\otimes C_k^{(k)}-\eta_N((\mathrm{id}_m\otimes N^{-1}\Tr)(R^{(k)}(\beta)))\Vert^{n}]\\
& +\mathbb{E}[\Vert \eta_N((\mathrm{id}_m\otimes N^{-1}\Tr)(R^{(k)}(\beta)))-\eta_N((\mathrm{id}_m\otimes \tau_N)(r_N(\beta)))\Vert^{n}].
\end{align*}
The first term is bounded by $\Vert \gamma_1\Vert^{n}\delta_N^{n-1}\tilde{\sigma}_N$ by assumption; the second term asymptotically vanishes by Lemma \ref{lem_moment_quadratic_forms}; the third term asymptotically vanishes by Lemma \ref{asymptoticfreenessLp} and Lemma \ref{compkLp}.
Then, choosing $q,r\geq1$ such that $\frac{1}{q}+\frac{1}{r}=1$ and $pq\geq 2$,
\begin{align*}
 & \mathbb{E}[\Vert R_{kk}(\beta)-\hat R_k(\beta)\Vert ^p]\\
 & \leq \Vert \hat R(\beta)\Vert ^p\mathbb{E}\big[\Vert W_{kk}\gamma_1+\gamma_1\otimes C_k^{(k)*}R^{(k)}(\beta)\gamma_1\otimes C_k^{(k)}-\eta_N((\mathrm{id}_m\otimes \tau_N)(r_N(\beta)))\Vert^{pq}\big]^{1/q}\mathbb{E}\big[\Vert R_N(\beta)\Vert ^{pr}\big]^{1/r}
\end{align*}
vanishes uniformly in $k$ by using Lemma \ref{majRNRkLp} and Lemma \ref{Qbound}. 
\end{proof}

\begin{rmk}\label{elemdiag}
Similarly, $\mathbb{E}[\Vert R_{ll}^{(k)}(\beta)-\hat R_l(\beta)\Vert ^p]$ vanishes uniformly in $1\leq k\neq l\leq N$. Indeed, using in this Remark only the notation $R^{(kl)}(\beta)=(\beta\otimes I_{N-2}-\gamma_1\otimes W_N^{(kl)}-\gamma_2\otimes D_N^{(kl)})^{-1}$, where $W_N^{(kl)}, D_N^{(kl)}$ denote the $(N-2)\times(N-2)$ matrices obtained respectively from $W_N, D_N$ by deleting the $k$-th and $l$-th rows/columns,
\begin{align*}
R_{ll}^{(k)}(\beta)
& = (\beta-W_{ll}\gamma_1-D_{ll}\gamma_2-\gamma_1\otimes C_l^{(kl)*}R^{(kl)}(\beta)\gamma_1\otimes C_l^{(kl)})^{-1}\\
& = \hat R_{l}(\beta)+\hat R_{l}(\beta)(W_{ll}\gamma_1+\gamma_1\otimes C_l^{(kl)*}R^{(kl)}(\beta)\gamma_1\otimes C_l^{(kl)}-\eta_N((\mathrm{id}_m\otimes \tau_N)(r_N^{(k)}(\beta))))R_{ll}^{(k)}(\beta).
\end{align*}
Observe that, for $n\geq 2$,
\begin{align*}
3^{-n+1}\mathbb{E}& \big[\Vert W_{ll}\gamma_1+\gamma_1\otimes C_l^{(kl)*}R^{(kl)}(\beta)\gamma_1\otimes C_l^{(kl)}-\eta_N((\mathrm{id}_m\otimes \tau_N)(r_N(\beta)))\Vert^{n}\big]\\
& \leq \Vert \gamma_1\Vert^{n}\mathbb{E}\big[\vert W_{ll}\vert^{n}\big]+\mathbb{E}\big[\Vert \gamma_1\otimes C_l^{(kl)*}R^{(kl)}(\beta)\gamma_1\otimes C_l^{(kl)}-\eta_N((\mathrm{id}_m\otimes N^{-1}\Tr)(R^{(kl)}(\beta)))\Vert^{n}\big]\\
& +\mathbb{E}\big[\Vert \eta_N((\mathrm{id}_m\otimes N^{-1}\Tr)(R^{(kl)}(\beta)))-\eta_N((\mathrm{id}_m\otimes \tau_N)(r_N(\beta)))\Vert^{n}\big]. 
\end{align*}
The first term is bounded by $\Vert \gamma_1\Vert^{n}\delta_N^{n-1}\tilde{\sigma}_N$ by assumption; the second term asymptotically vanishes by Lemma \ref{lem_moment_quadratic_forms}; the third term asymptotically vanishes by Lemma \ref{asymptoticfreenessLp} and Lemma \ref{compkLp}.
Then, choosing $q,r\geq1$ such that $\frac{1}{q}+\frac{1}{r}=1$ and $pq\geq 2$,
\begin{align*}
 & \mathbb{E}\big[\Vert R_{ll}^{(k)}(\beta)-\hat R_l(\beta)\Vert ^p\big]\\
 & \leq \Vert \hat R(\beta)\Vert ^p\mathbb{E}\big[\Vert W_{ll}\gamma_1+\gamma_1\otimes C_l^{(kl)*}R^{(kl)}(\beta)\gamma_1\otimes C_l^{(kl)}-\eta_N((\mathrm{id}_m\otimes \tau_N)(r_N(\beta)))\Vert^{pq}\big]^{1/q}\mathbb{E}\big[\Vert R^{(k)}(\beta)\Vert ^{pr}\big]^{1/r}
\end{align*}
vanishes uniformly in $l\neq k$ by using Remark \ref{remmajRNRkLp} and Lemma \ref{Qbound}.
\end{rmk}

\begin{rmk}\label{kilhat}
One may also prove that $\mathbb{E}[\Vert R_{ll}^{(kil)}(\beta)-\hat R_l(\beta)\Vert ^p]$ vanishes uniformly in $1\leq i\neq k\neq l\leq N$, using $\mathbb{E}[\Vert R_{ll}^{(k)}(\beta)-R_{ll}^{(kil)}(\beta)\Vert^p]=o(1)$.
\end{rmk}

\subsection{Concentration bounds on $R_N$ and $R^{(k)}$}

\begin{lem}\label{lem-varianceboundtrace}
For 
any $c>0$, 
\[\mathbb{E}[\| \mathrm{id}_m\otimes\Tr R_N(ze_{11}-\gamma_0)-\mathbb{E}[\mathrm{id}_m\otimes \Tr R_N (ze_{11}-\gamma_0)]\|^4]=O(N^{2+c}).\]
\end{lem}

\begin{proof}
Observe as in the proof of Lemma \ref{crudevariancebound} that $M_k:=\mathbb{E}_{\leq k}[\mathrm{id}_m\otimes \Tr R_N (ze_{11}-\gamma_0)]$, $k\geq 0$, satisfies
\begin{align*}
M_k-M_{k-1}
&=(\mathbb{E}_{\leq k}-\mathbb{E}_{\leq k-1})[\mathrm{id}_m\otimes \Tr R_N (ze_{11}-\gamma_0)]\\
&=(\mathbb{E}_{\leq k}-\mathbb{E}_{\leq k-1})[\mathrm{id}_m\otimes \Tr R_N (ze_{11}-\gamma_0)-\mathrm{id}_m\otimes \Tr R^{(k)}(ze_{11}-\gamma_0)]
\end{align*} and
write as in the proof of Lemma \ref{compkLp}: for $z\in\mathbb{C}\setminus\mathbb{R}$,
\begin{align*}
\mathrm{id}_m\otimes \Tr R_N(ze_{11} -\gamma_0) & -\mathrm{id}_m \otimes \Tr  R^{(k)}(ze_{11}-\gamma_0)\\
& = s_k^{-1}+\mathrm{id}_m\otimes \Tr(R^{(k)}(ze_{11}-\gamma_0)\gamma_1\otimes C_k^{(k)}s_k^{-1}\gamma_1\otimes C_k^{(k)*}R^{(k)}(ze_{11}-\gamma_0)),
\end{align*}
where $s_k=ze_{11}-\gamma_0-W_{kk}\gamma_1-D_{kk}\gamma_2-\gamma_1\otimes C_k^{(k)*}R^{(k)}(ze_{11}-\gamma_0)\gamma_1\otimes C_k^{(k)}$.
It follows from Remark \ref{remboundomegatilde} that for any compact subset $K$ of $\mathbb{C}\setminus \mathbb{R}$, $\mathbb{E}[s_k],\, z\in K,$ are all invertible for $N$ large enough and $\Vert \mathbb{E}[s_k]^{-1}\Vert=O(1)$.
Observe that, for such $N$,  
\[(\mathbb{E}_{\leq k}-\mathbb{E}_{\leq k-1})[\mathbb{E}[s_k]^{-1}+\mathbb{E}_k[\mathrm{id}_m\otimes \Tr(R^{(k)}(ze_{11}-\gamma_0)\gamma_1\otimes C_k^{(k)}\mathbb{E}[s_k]^{-1}\gamma_1\otimes C_k^{(k)*}R^{(k)}(ze_{11}-\gamma_0))]]=0\]
so that 
\begin{align}\label{dvpt}
 -(\mathbb{E}_{\leq k}&-\mathbb{E}_{\leq k-1}) [\mathrm{id}_m\otimes \Tr R^{(k)}(ze_{11}-\gamma_0)- \mathrm{id}_m\otimes \Tr R_N (ze_{11}-\gamma_0)]\nonumber\\
& =(\mathbb{E}_{\leq k}-\mathbb{E}_{\leq k-1})[s_k^{-1}-\mathbb{E}[s_k]^{-1}]\nonumber\\
&+(\mathbb{E}_{\leq k}-\mathbb{E}_{\leq k-1})[\mathrm{id}_m\otimes \Tr(R^{(k)}(ze_{11}-\gamma_0)\gamma_1\otimes C_k^{(k)}(s_k^{-1}-\mathbb{E}[s_k]^{-1})\gamma_1\otimes C_k^{(k)*}R^{(k)}(ze_{11}-\gamma_0))]\nonumber \\
&+(\mathbb{E}_{\leq k}-\mathbb{E}_{\leq k-1})[\mathrm{id}_m\otimes \Tr(R^{(k)}(ze_{11}-\gamma_0)\gamma_1\otimes C_k^{(k)}\mathbb{E}[s_k]^{-1}\gamma_1\otimes C_k^{(k)*}R^{(k)}(ze_{11}-\gamma_0))\nonumber \\
& -\mathbb{E}_k[\mathrm{id}_m\otimes \Tr(R^{(k)}(ze_{11}-\gamma_0)\gamma_1\otimes C_k^{(k)}\mathbb{E}[s_k]^{-1}\gamma_1\otimes C_k^{(k)*}R^{(k)}(ze_{11}-\gamma_0))]].
\end{align}
From
\begin{align*}
 & s_k^{-1}-\mathbb{E}[s_k]^{-1}\\
 & =s_k^{-1}(W_{kk}\gamma_1+\gamma_1\otimes C_k^{(k)*}R^{(k)}(ze_{11}-\gamma_0)\gamma_1\otimes C_k^{(k)}-\mathbb{E}[\gamma_1\otimes C_k^{(k)*}R^{(k)}(ze_{11}-\gamma_0)\gamma_1\otimes C_k^{(k)}])\mathbb{E}[s_k]^{-1},
\end{align*}  
H\"older's inequality, together with \eqref{InverseSchurComplementBound}, Lemmas \ref{alta lemma}, \ref{puissance}, \ref{estimfq}, \ref{majRNRkLp}, Remarks \ref{remmajRNRkLp} and \ref{remboundomegatilde}, deduce that, for large $N$, for any $k$, for small enough $t>0$,
\begin{align}
 \mathbb{E}[ \Vert & s_k^{-1} -\mathbb{E}[s_k]^{-1}\Vert^{4(1+t)}]\nonumber\\
&\leq \mathbb{E}[\Vert s_k^{-1}\Vert^{4(1+t)}\Vert W_{kk}\gamma_1+\gamma_1\otimes C_k^{(k)*}R^{(k)}(ze_{11}-\gamma_0)\gamma_1\otimes C_k^{(k)}\nonumber\\
 & \hspace{5cm }-\mathbb{E}[\gamma_1\otimes C_k^{(k)*}R^{(k)}(ze_{11}-\gamma_0)\gamma_1\otimes C_k^{(k)}]\Vert^{4(1+t)}]\Vert \mathbb{E}[s_k]^{-1}\Vert^{4(1+t)}\nonumber \\
&\leq O(1)\mathbb{E}[\Vert s_k^{-1}\Vert^{4(1+t)(2+1/t)}]^{t/(1+2t)}\Big\{(\| \gamma_1 \|\sigma_N)^{8(1+2t)}\nonumber\\
& \hspace{4cm} \times \mathbb{E}\Big[ \big\| \mathrm{id}_m \otimes \Tr R^{(k)}(ze_{11}-\gamma_0)-\mathbb{E} [\mathrm{id}_m \otimes \Tr R^{(k)}(ze_{11}-\gamma_0)] \big\|^{4(1+2t)}\Big]\nonumber\\
& \hspace{10.5cm}+O(N^{-2(1+2t)}))\Big\}^{(1+t)/(1+2t)}\nonumber\\
&\leq O(1)\mathbb{E}[\Vert R_N(ze_{11}-\gamma_0)\Vert^{4(1+t)(2+1/t)}]^{t/(1+2/t)}\Big\{(\| \gamma_1 \|\sigma_N)^{8(1+2t)} \nonumber\\
& \hspace{1cm} \times \mathbb{E}\Big[(2N\Vert R^{(k)}(ze_{11}-\gamma_0)\Vert)^{9t}
\big\| \mathrm{id}_m \otimes \Tr R^{(k)}(ze_{11}-\gamma_0)-\mathbb{E}\big( \mathrm{id}_m \otimes \Tr R^{(k)}(ze_{11}-\gamma_0)\big) \big\|^{4-t}\Big] \nonumber\\
 &\hspace{10.5cm} +O(N^{-2(1+2t)})\Big\}^{(1+t)/(1+2t)}\nonumber\\
&\leq O(1)\Big\{ N^{9t -4(1+2t)} \mathbb{E} \big(
 \big\| \mathrm{id}_m \otimes \Tr R^{(k)}(ze_{11}-\gamma_0)-\mathbb{E}\big( \mathrm{id}_m \otimes \Tr R^{(k)}(ze_{11}-\gamma_0)\big) \big\|^{4}\big)^{1-t/4}\nonumber\\
 & \hspace{10.5cm} +O(N^{-2(1+2t)})\Big\}^{(1+t)/(1+2t)}
\nonumber\\
&= \Big\{O(N^{-1 +t/4}) +O(N^{-2(1+2t)})\Big\}^{(1+t)/(1+2t)}\nonumber\\
&= O(N^{(-1 +t/4)(1+t)/(1+2t)}) \label{O}
\end{align}
where we use Remark \ref{crudevarianceboundk} in the last line.

It follows that 
\begin{align}
\mathbb{E}[\Vert s_k^{-1}-\mathbb{E}[s_k]^{-1}\Vert^4]
&\leq \mathbb{E}[\Vert s_k^{-1}-\mathbb{E}[s_k]^{-1}\Vert^{4(1+t)}]^{1/(1+t)}\nonumber\\
&=O(N^{(t/4-1)/(1+2t)})\label{dvpt2}
\end{align}
and 
\begin{align}\label{dvpt3}
 \mathbb{E}[\Vert \mathrm{id}_m &\otimes \Tr(R^{(k)}(ze_{11}-\gamma_0)\gamma_1\otimes C_k^{(k)}(s_k^{-1}-\mathbb{E}[s_k]^{-1})\gamma_1\otimes C_k^{(k)*}R^{(k)}(ze_{11}-\gamma_0))\Vert^4]\nonumber\\
& \leq m^{16}\Vert \gamma_1\Vert^8\mathbb{E}[\Vert R^{(k)}(ze_{11}-\gamma_0)\Vert^8\Vert C_k^{(k)}\Vert^8\Vert s_k^{-1}-\mathbb{E}[s_k]^{-1}\Vert^4]\nonumber\\
& \leq m^{16}\Vert \gamma_1\Vert^8\mathbb{E}[\Vert R^{(k)}(ze_{11}-\gamma_0)\Vert^{8(1+1/t)}\Vert C_k^{(k)}\Vert^{8(1+1/t)}]^{t/(1+t)}\mathbb{E}[\Vert s_k^{-1}-\mathbb{E}[s_k]^{-1}\Vert^{4(1+t)}]^{1/(1+t)}\nonumber\\
& =O(N^{(t/4-1)/(1+2t)}),
\end{align}
by using \eqref{O}, Remark \ref{remmajRNRkLp} and Proposition \ref{boundnorm}.\\
Finally, if $R^{(k)}=\sum\limits_{i,j\neq k}R^{(k)}_{ij}\otimes E_{ij}$ and  $R^{(k)pT}=\sum\limits_{i,j\neq k}R^{(k)}_{ij}\otimes E_{ji}$, then 
\begin{align*}
\mathrm{id}_m\otimes \Tr(R^{(k)}(ze_{11}-\gamma_0)\gamma_1\otimes C_k^{(k)}\mathbb{E}[s_k]^{-1}\gamma_1\otimes  C_k^{(k)*}R^{(k)} & (ze_{11}-\gamma_0))\\
&= \sum_{a,b,c\neq k}R^{(k)}_{ab}\gamma_1\mathbb{E}[s_k]^{-1}\gamma_1R^{(k)}_{ca}W_{bk}\overline{W}_{ck}\\
&= I_m\otimes \overline{C_k^{(k)}}^*\Sigma I_m\otimes \overline{C_k^{(k)}},
\end{align*}
where the fourth moment of the norm of 
\[\Sigma=\sum_{a,b,c\neq k}R^{(k)}_{ab}\gamma_1\mathbb{E}[s_k]^{-1}\gamma_1R^{(k)}_{ca}\otimes E_{bc}=R^{(k)}(ze_{11}-\gamma_0)^{pT}(\gamma_1\mathbb{E}[s_k]^{-1}\gamma_1\otimes I_{N-1})R^{(k)}(ze_{11}-\gamma_0)^{pT}\]
is $O(1)$ by using Proposition \ref{partialtranspose}, Remarks \ref{remmajRNRkLp} and \ref{remboundomegatilde}.
It follows from Lemma \ref{lem_moment_quadratic_forms} that 
\begin{align}\label{dvpt4}
\mathbb{E} [\Vert \mathrm{id}_m  \otimes\Tr & (R^{(k)}(ze_{11}-\gamma_0)\gamma_1\otimes C_k^{(k)}\mathbb{E}[s_k]^{-1}\gamma_1\otimes C_k^{(k)*}R^{(k)}(ze_{11}-\gamma_0))\nonumber\\
& \hspace{2cm} -\mathbb{E}_k[\mathrm{id}_m\otimes\Tr(R^{(k)}(ze_{11}-\gamma_0)\gamma_1\otimes C_k^{(k)}\mathbb{E}[s_k]^{-1}\gamma_1\otimes C_k^{(k)*}R^{(k)}(ze_{11}-\gamma_0))]\Vert^4]\nonumber\\
& \leq K_{m,4}\mathbb{E}[\Vert \Sigma\Vert^4]((Nm_N)^2+O(N^{-3})N)
=O(N^{-2}),
\end{align}
where we use Lemma \ref{puissance} in the last line.
\eqref{dvpt}, \eqref{dvpt2}, \eqref{dvpt3} and \eqref{dvpt4} yield
\[\mathbb{E}[\Vert(\mathbb{E}_{\leq k}-\mathbb{E}_{\leq k-1})[\mathrm{id}_m\otimes\Tr R^{(k)}(ze_{11}-\gamma_0)-\mathrm{id}_m\otimes\Tr R_N(ze_{11}-\gamma_0)]\Vert^4]=O(N^{-1+c}),\]
by choosing $t>0$ small enough. 
We conclude by Lemma \ref{martingalelp}.
\end{proof}

\begin{rmk}\label{rmk-varianceboundtrace}
Using Lemma \ref{compkLp}, one may deduce from Lemma \ref{lem-varianceboundtrace} that, for 
$c>0$, 
\[\mathbb{E}[\|\mathrm{id}_m\otimes \Tr R^{(k)} (ze_{11}-\gamma_0)-\mathbb{E}[\mathrm{id}_m\otimes\Tr R^{(k)} (ze_{11}-\gamma_0)]\|^4]=O(N^{2+c}).\]
\end{rmk}

\begin{lem}\label{concentrationboundentry}
For 
$p\in [1,4)$ and $c>0$ 
\[\mathbb{E}[\| R_{kk}(ze_{11}-\gamma_0)-(\Omega_N(ze_{11}-\gamma_0)\otimes I_N-\gamma_2\otimes D_N)_{kk}^{-1}]\|^{p}]=O(N^{(c-2)p/4}).\]
\end{lem}

\begin{proof} 
Recall that $\Omega_N$ was defined in \eqref{defomegatilde} by 
$$\Omega_N(\beta)=\beta-\eta_N\big(\mathbb{E}[({\rm id}_m\otimes N^{-1}\Tr)(R_N(\beta))]\big).$$
According to Lemma \ref{boundomegatilde} and its proof, for any compact subset $K$ of $\mathbb{C}\setminus \mathbb{R}$, for large enough $N$, $\Omega_N(ze_{11}-\gamma_0)\otimes I_N-\gamma_2\otimes D_N,\, z\in K,$ are all invertible with uniformly bounded inverses. For such $N$ and any  $k=1,\ldots ,N$,
\begin{align*}
R_{kk}(ze_{11}- & \gamma_0)-(\Omega_N(ze_{11}-\gamma_0)-D_{kk}\gamma_2)^{-1}\\
& = R_{kk}(ze_{11}-\gamma_0)\Big(W_{kk}\gamma_1+\gamma_1\otimes C_k^{(k)*}R^{(k)}(ze_{11}-\gamma_0)\gamma_1\otimes C_k^{(k)}\\
& \hspace{2.5cm} -\eta_N\big(\mathbb{E}\big[(\mathrm{id}_m\otimes N^{-1}\Tr)(R_N(ze_{11}-\gamma_0))\big]\big)\Big)\big(\Omega_N(ze_{11}-\gamma_0)-D_{kk}\gamma_2\big)^{-1}.
\end{align*}
Then, using Lemmas \ref{puissance}, \ref{estimfq}, \ref{compkLp} and \ref{lem-varianceboundtrace}, for $p<4$,
\begin{align*}
 \mathbb{E}\big[\Vert R_{kk} & (ze_{11}-\gamma_0)-(\Omega_N(ze_{11}-\gamma_0)-D_{kk}\gamma_2)^{-1}\Vert^p\big]\\
& \leq 
\mathbb{E}\big[\Vert R_{kk}(ze_{11}-\gamma_0)\Vert^{4p/(4-p)}\big]^{1-p/4}\big(O(N^{-2})+O(N^{c-2})\big)^{p/4}O(1)\\
& =O(N^{(c-2)p/4}).
\end{align*}
\end{proof}

\subsection{Quantitative asymptotic freeness}


\begin{lem}\label{sprN}
The family of operators defined on $M_m(\mathbb{C})$ by $$u_N(z_1,z_2):b\mapsto  \sigma_N^{2}\sum_{i=1}^N\hat{R}_i(z_1e_{11}-\gamma_0)\gamma_1b \gamma_1\hat{R}_i(z_2e_{11}-\gamma_0),\quad z_1, z_2\in \C\setminus \R,\, N\in \mathbb{N},$$ satisfies the following: for any compact subset $K$ of $\C\setminus \R$, $\limsup_{N\to +\infty}\sup_{z_1,z_2\in K}\rho(u_N(z_1,z_2))<1$ and $\limsup_{N\to +\infty}\sup_{z_1,z_2\in K}\Vert (\mathrm{id}_m-u_N(z_1,z_2))^{-1} \Vert<+\infty$.
\end{lem}

\begin{proof}
By Proposition \ref{welldefined3}, we know that the supremum over $z_1,z_2\in K$ of the spectral radii of operators $u_{z_1e_{11}-\gamma_0,z_2e_{11}-\gamma_0}:M_m(\mathbb{C})\to M_m(\mathbb{C})$ defined by 
$$b\mapsto \sigma^2\mathrm{id}_m\otimes \tau \left((\omega(z_1e_{11}-\gamma_0)-\gamma_2\otimes d)^{-1}(\gamma_1b\gamma_1)\otimes 1_\mathcal{A} (\omega(z_2 e_{11}-\gamma_0)-\gamma_2\otimes d)^{-1}\right),\quad z_1,z_2\in \C\setminus \R,$$
is strictly smaller than 1. Since $(D_N)_{N\in \mathbb{N}}$ converges in $*$-moments towards $d$, $\lim_{N\to +\infty}N\sigma_N^2=\sigma^2$, we can easily deduce (using Lemmas \ref{Qbound}, \ref{bound} and Lemma \ref{cvunif}) that the family of operators defined for all $N\in \N$ and for all $z_1,z_2$ in $\C \setminus \R$,  by 
\begin{align*}
u_N & (z_1,z_2) :\\
 & b\mapsto \sigma_N^2\mathrm{id}_m\otimes \tau_N \left((\omega_N(z_1e_{11}-\gamma_0)-\gamma_2\otimes D_N)^{-1}(\gamma_1b\gamma_1)\otimes 1_\mathcal{A} (\omega_N(z_2 e_{11}-\gamma_0)-\gamma_2\otimes D_N)^{-1}\right), 
\end{align*}
satisfies $\sup_{z_1,z_2\in K}\left\|u_N(z_1,z_2)-u_{z_1 e_{11} -\gamma_0,z_2 e_{11} -\gamma_0}\right\| \longrightarrow_{N \rightarrow +\infty}0$. The first assertion of the lemma readily follows by continuity of the spectral radius in finite dimension and the second one follows from the continuity of $X  \mapsto X^{-1}$ and of the norm on $\mathbb{G}L_{m^2}(\C)$.
\end{proof}

\begin{cor}\label{sprNsanstilde}
For any $z_1, z_2$ in $\C\setminus \R$, the sequence of operators defined on $M_m(\mathbb{C})\otimes M_m(\mathbb{C})$ by $$T_N(z_1,z_2):b_1\otimes b_2\mapsto  N^{-1}\sum_{i=1}^N\hat{R}_i(z_1e_{11}-\gamma_0)\gamma_1b_1\otimes b_2\gamma_1\hat{R}_i(z_2e_{11}-\gamma_0),\quad N\in \mathbb{N},$$ satisfies $\limsup_{N\to +\infty}\rho(T_N(z_1,z_2))<\sigma^{-2}$.
\end{cor}

\begin{lem}\label{cvRinfiniLp}
For 
$q\in [2,4(1+\varepsilon))$,
\[\mathbb{E}[\|\mathrm{id}_m\otimes N^{-1}\Tr (R_N(ze_{11}-\gamma_0))-\mathrm{id}_m\otimes \tau(r_N(ze_{11}-\gamma_0))\|^{q}]=O(N^{-\min(2-\varepsilon,q/2)}).\]
\end{lem}

\begin{proof}
Define $\hat G=\sum_{i=1}^N E_{ii} \otimes \hat R_i(ze_{11}-\gamma_0),$
and $G=\sum_{i=1}^N E_{ii} \otimes R_{ii}(ze_{11}-\gamma_0).$
Recall from \eqref{defhatRk} that $\hat R_i(ze_{11}-\gamma_0)= \left( \omega_N(ze_{11}-\gamma_0) -D_{ii}\gamma_2\right)^{-1}$.
We have by \eqref{fixed}
\begin{align*} 
\omega_N(ze_{11}-\gamma_0)-D_{ii}\gamma_2 
&=ze_{11}-\gamma_0-N\sigma_N^2 \gamma_1 N^{-1}\sum_{k=1}^N 
\left( \omega_N(ze_{11}-\gamma_0) -D_{kk}\gamma_2\right)^{-1}\gamma_1-D_{ii}\gamma_2\\
&= ze_{11}-\gamma_0-N\sigma_N^2 \gamma_1 N^{-1}\sum_{k=1}^N \hat R_k(ze_{11}-\gamma_0) \gamma_1 -D_{ii}\gamma_2.
\end{align*}
Thus,
\begin{equation}\label{equationtilde}
\hat G^{-1}=I_N \otimes \left[ze_{11}-\gamma_0 - N \sigma_N^2 \gamma_1 N^{-1}\Tr\otimes \mathrm{id}_m (\hat G)\gamma_1\right]-D_N\otimes \gamma_2.\end{equation}
On the other hand, by Schur formula,
\begin{align*} 
\left( R_{ii}(ze_{11}-\gamma_0) \right)^{-1}&=  ze_{11}-\gamma_0 -W_{ii}\gamma_1-D_{ii} \gamma_2-  \gamma_1\otimes C_i^{(i)*} R^{(i)}(ze_{11}-\gamma_0 ) \gamma_1 \otimes C_i^{(i)} 
\\&= ze_{11}-\gamma_0  -D_{ii} \gamma_2 -N \sigma_N^2 \gamma_1 N^{-1}\Tr\otimes \mathrm{id}_m\left( G\right)\gamma_1
+ \Delta_i,
\end{align*} 
where 
\begin{align*}
\Delta_i&= -W_{ii}\gamma_1 -  \gamma_1\otimes C_i^{(i) *}  R^{(i)}(ze_{11}-\gamma_0 ) \gamma_1 \otimes C_i^{(i)} + N \sigma_N^2 \gamma_1 N^{-1}\Tr\otimes \mathrm{id}_m \left( G\right)\gamma_1\\
&=
- W_{ii}\gamma_1-\gamma_1\otimes C_i^{(i) *}R^{(i)}(ze_{11}-\gamma_0)\gamma_1\otimes C_i^{(i)}+N\sigma_N^2\gamma_1(\mathrm{id}_m\otimes N^{-1}\Tr) (R_N(ze_{11}-\gamma_0))\gamma_1.
\end{align*}
Define $\Delta=\sum\limits_{i=1}^N E_{ii} \otimes \Delta_i.$
Thus,
\begin{equation}\label{equation}G^{-1}=I_N\otimes\left[ze_{11}-\gamma_0-N\sigma_N^2 \gamma_1N^{-1}\Tr\otimes \mathrm{id}_m \left(G\right)\gamma_1\right]-D_N\otimes \gamma_2+\Delta.\end{equation}
Set $$\Xi_N=N^{-1}\Tr\otimes \mathrm{id}_m\left(\hat G-G\right)=\mathrm{id}_m\otimes\tau_N(r_N(ze_{11}-\gamma_0))-\mathrm{id}_m\otimes N^{-1}\Tr(R_N (ze_{11}-\gamma_0)).$$
Note that, 
by Lemma \ref{alta lemma}, 
all moments of $\Vert\Xi_N\Vert$ are $O(1)$.
By substracting \eqref{equationtilde} from \eqref{equation}, 
\begin{align*}
\hat G -G &=G\left( I_N \otimes N \sigma_N^2 \gamma_1 \Xi_N\gamma_1\right)\hat G +G\Delta \hat G \nonumber\\
&=
\hat G\left(I_N\otimes N\sigma_N^2\gamma_1 \Xi_N\gamma_1 \right)\hat G + \left( G-\hat G\right)\left( I_N \otimes N \sigma_N^2 \gamma_1 \Xi_N\gamma_1\right)\hat G+G\Delta \hat G.\label{dif} 
\end{align*}
Hence
\begin{align*}
\Xi_N & =
N^{-1}\sum_{k=1}^N \hat R_{k}(ze_{11}-\gamma_0 )N\sigma_N^2\gamma_1 \Xi_N\gamma_1\hat R_{k}(ze_{11}-\gamma_0 )\\
& \quad + N^{-1}\Tr\otimes \mathrm{id}_m\left[\left(G-\hat G\right)\left(I_N\otimes N\sigma_N^2\gamma_1 \Xi_N\gamma_1\right)\hat G+G\Delta\hat G\right]\\
& =u_N(z,z)\Xi_N+N^{-1}\Tr\otimes \mathrm{id}_m\left[\left(G-\hat G\right)\big(I_N\otimes N\sigma_N^2\gamma_1 \Xi_N\gamma_1\big)\hat G+G\Delta\hat G\right]
\end{align*}
According to Lemma \ref{sprN}, for large enough N, the operators  $\mathrm{id}_m-u_N(z,z)$ on $M_m(\mathbb{C})$ are invertible with uniformly bounded inverses.
Therefore 
\[\Xi_N=(\mathrm{id}_m-u_N(z,z))^{-1}\left\{N^{-1}\Tr\otimes \mathrm{id}_m\left[\left(G-\hat G\right)\left(I_N\otimes N\sigma_N^2\gamma_1\Xi_N\gamma_1\right)\hat G+G\Delta \hat G\right]\right\}\]
and 
\begin{equation}\label{xi}
\left\|\Xi_N\right\|\leq \|(\mathrm{id}_m-u_N(z,z))^{-1}\|\left\|N^{-1}\Tr\otimes \mathrm{id}_m\left[\left(G-\hat G\right)\left(I_N\otimes N\sigma_N^2 \gamma_1\Xi_N\gamma_1\right)\hat G+G\Delta \hat G\right]\right\|.
\end{equation}
We have 
\begin{align*}
 N^{-1}\Tr\otimes \mathrm{id}_m\Big[\left(G-\hat G\right) & \left(I_N\otimes N\sigma_N^2\gamma_1\Xi_N\gamma_1\right)\hat G\Big]\\
  & = N^{-1}\sum_{i=1}^N (R_{ii}(ze_{11}-\gamma_0)-\hat R_i(ze_{11}-\gamma_0)) N \sigma_N^2 \gamma_1 \Xi_N \gamma_1 \hat R_i(ze_{11}-\gamma_0).
\end{align*}
Thus \eqref{xi} yields
\begin{align*}
\Vert \Xi_N\Vert 
&\leq   \|(\mathrm{id}_m-u_N(z,z))^{-1}\|{\Big[}\Vert \gamma_1 \Vert^2\big\|\hat G\big\| {\sigma_N^2}\Big\{\sum_{i=1}^N\big\|R_{ii}(ze_{11}-\gamma_0)-\hat R_i(ze_{11}-\gamma_0)\big\|\Big\}\Vert \Xi_N\Vert\\
& \hspace{2cm} +\big\| N^{-1}\Tr\otimes \mathrm{id}_m(G\Delta \hat G)\big\|{\Big]}.
\end{align*}
Therefore, if  $\sigma_N^2\big\{\sum_{i=1}^N  \big\|R_{ii}(ze_{11}-\gamma_0)-\hat R_i(ze_{11}-\gamma_0)\big\|\big\}\leq (2\Vert \gamma_1 \Vert^2\big\|\hat G\big\|\|(\mathrm{id}_m-u_N(z,z))^{-1}\|)^{-1}$, then by using Lemma \ref{alta lemma},
\begin{align*}
\Vert\Xi_N\Vert & \leq 2\|(\mathrm{id}_m-u_N(z,z))^{-1}\|\|N^{-1}\Tr\otimes \mathrm{id}_m(G\Delta\tilde G)\|\\ 
& \leq 2\|(\mathrm{id}_m-u_N(z,z))^{-1}\|\|N^{-1}\sum_{i=1}^N R_{ii}(ze_{11}-\gamma_0)\Delta_i \hat R_i(ze_{11}-\gamma_0)\|\\
& \leq 2\|(\mathrm{id}_m-u_N(z,z))^{-1}\|\|R_N(ze_{11}-\gamma_0)\|\|\hat{R}(ze_{11}-\gamma_0)\|\{N^{-1}\sum_{i=1}^N \left\|\Delta_i \right\|\}. 
\end{align*}
Set $E= \left\{\sigma_N^2\sum\limits_{i=1}^N  \big\|R_{ii}(ze_{11}-\gamma_0)-\hat R_i(ze_{11}-\gamma_0)\big\|> (2\Vert \gamma_1 \Vert^2\big\|\hat G\big\|\|(\mathrm{id}_m-u_N(z,z))^{-1}\|)^{-1}\right\}$.
For $q\in (0,4(1+\varepsilon))$, using Lemmas \ref{majRNRkLp} and \ref{sprN}, we have 
\begin{align*}
\mathbb{E}\left(\Vert \Xi_N \Vert^{q} \right)&=\mathbb{E}\left(\Vert \Xi_N \Vert^{q}\1_{E} \right)+\mathbb{E}\left(\Vert \Xi_N \Vert^{q} \1_{E^c}\right)\\
&\leq 
O(1)\mathbb{P}\left(E\right)^{1/p} +   
O(1)\mathbb{E}\Big(\Big\{N^{-1}\sum_{i=1}^N \big\|\Delta_i \big\|\Big\}^{4(1+\varepsilon)}\Big)^{q/4(1+\varepsilon)}
\end{align*}
where $p\geq 1$
will be choosen later on.\\
Note that, for $q\in [2,4(1+\varepsilon))$,
\begin{align*}
\mathbb{E}\Big(\Big\{N^{-1}\sum_{i=1}^N \big\|\Delta_i \big\|\Big\}^{4(1+\varepsilon)}\Big)^{q/4(1+\varepsilon)} & \leq \max_{i=1,\ldots,N}
\mathbb{E}\Big(\big\|\Delta_i \big\|^{4(1+\varepsilon)}\Big)^{q/4(1+\varepsilon)}\\
& =O\Big(N^{-2(1+\varepsilon)}\Big)^{q/4(1+\varepsilon)}=O\Big(N^{-q/2}\Big),
\end{align*}
where we use the convexity of $x\mapsto x^{4(1+\varepsilon)}$
and Lemmas \ref{estimfq}, \ref{compkLp} and \ref{puissance}.
Now, according to Lemma \ref{asymptoticfreenessL1} and Lemma \ref{inversionsomme},
\[\Vert \hat R(ze_{11}-\gamma_0)-(\Omega_N(ze_{11}-\gamma_0)\otimes I_N-\gamma_2\otimes D_N)^{-1}\Vert \underset{N\to +\infty}{\longrightarrow} 0.\]
For those $N$ for which $$N\sigma_N^2\Vert \hat R(ze_{11}-\gamma_0)-(\Omega_N(ze_{11}-\gamma_0)\otimes I_N-\gamma_2\otimes D_N)^{-1}\Vert \leq (4\Vert \gamma_1 \Vert^2\big\|\hat G\big\|\|(\mathrm{id}_m-u_N(z,z))^{-1}\|)^{-1},$$ 
the probability of the event $E$ satisfies, for $\varepsilon'\leq \varepsilon$:
\begin{align*}
 \mathbb{P}(E)
&\leq \mathbb{P}\Big(\sigma_N^2\sum_{i=1}^N  \big\|R_{ii}(ze_{11}-\gamma_0)-(\Omega_N(ze_{11}-\gamma_0)\otimes I_N-\gamma_2\otimes D_N)_{ii}^{-1}\big\|\\
& \hspace{7cm}> (4\Vert \gamma_1 \Vert^2\big\|\hat G\big\|\|(\mathrm{id}_m-u_N(z,z))^{-1}\|)^{-1}\Big)\\
&\leq \inf_{x\in[1,4)}(4\Vert \gamma_1 \Vert^2\big\|\hat G\big\|\|(\mathrm{id}_m-u_N(z,z))^{-1}\|N\sigma_N^2)^{x}
\end{align*}
\begin{align*}
& \hspace{4cm} \times \max_{i=1,\ldots ,N}\mathbb{E}\Big[\big\|R_{ii}(ze_{11}-\gamma_0)-(\Omega_N(ze_{11}-\gamma_0)\otimes I_N-\gamma_2\otimes D_N)_{ii}^{-1}\big\|^{x}\Big]\\
&=O(N^{-2+\varepsilon'})
\end{align*}
by Lemma \ref{concentrationboundentry}.

Finally choosing $p=\frac{2-\varepsilon'}{2-\varepsilon}$, we obtain \[\mathbb{E}\left(\Vert \Xi_N \Vert^{q} \right)=O(N^{-(2-\varepsilon)})+O(N^{-q/2})=O(N^{-\min(2-\varepsilon;q/2)}).\]
\end{proof}

\section{Proof of convergence in finite-dimensional distributions in Theorem \ref{cvfdim}} \label{CLT}
In this section, we will give a proof of the convergence in finite-dimensional distributions of the complex process $( \xi_N(z)=\Tr(zI_N-X_N)^{-1} -\mathbb{E}\left[\Tr(zI_N-X_N)^{-1}\right], \, z\in \mathbb{C}\setminus\mathbb{R})$ to the centred complex Gaussian process $\mathcal G$, based on Theorem \ref{thm_CLT_martingale}.


\subsection{Reduction of the problem} \label{subsectionreduction}
 
One has to prove that any linear combination of $\xi_N(z),\, z\in \mathbb{C}\setminus \mathbb{R},$ converges in distribution to a complex Gaussian variable $Z\sim \mathcal{N}_{\mathbb{C}}(0,V,W)$. In the following, we prove the convergence of $(\xi_N(z))_{N\in \N}$. The case of a general linear combination does not need any additional argument and is left to the reader.
Notice that 
\begin{equation}\label{decomposition}
\xi_N(z)=\sum_{k=1}^N (\mathbb{E}_{\leq k}-\mathbb{E}_{\leq k-1})\big[\Tr(zI_N-X_N)^{-1}\big].
\end{equation}
For each $N\in \N$, the random variable $\Tr(zI_N-X_N)^{-1}$ being bounded, $(\mathbb{E}_{\leq k}[\Tr(zI_N-X_N)^{-1}])_{k\geq 1}$ is a square integrable complex martingale, hence $(\mathbb{E}_{\leq k}[\Tr(zI_N-X_N)^{-1}]-\mathbb{E}_{\leq k-1}[\Tr(zI_N-X_N)^{-1}])_{k\geq 1}$ is a martingale difference. Our strategy is to apply the central limit theorem for sums of martingale differences. More precisely, we will decompose $(\mathbb{E}_{\leq k}-\mathbb{E}_{\leq k-1})[\Tr(zI_N-X_N)^{-1}]$ in two parts and apply Theorem \ref{thm_CLT_martingale} to the first part.

\begin{prop}
	For $z\in \mathbb{C}\setminus \mathbb{R}$ and $k\in \{1,\ldots ,N\}$, 
	\[
	(\mathbb{E}_{\leq k}-\mathbb{E}_{\leq k-1})\big[\Tr(zI_N-X_N)^{-1}\big]=\Delta_k^{(N)}+\varepsilon_k^{(N)},
	\]
	where \[\Delta_k^{(N)}=\esp_{\leq k}\Big[-\frac{\partial}{\partial z}\Tr\big((W_{kk}\gamma_1+\Phi_k(ze_{11}-\gamma_0))\hat{R}_k(ze_{11}-\gamma_0)\big)\Big],\]
	with
	\[ \Phi_k(ze_{11}-\gamma_0):=\gamma_1\otimes C_k^{(k)*}R^{(k)}(ze_{11}-\gamma_0)\gamma_1\otimes C_k^{(k)}-\gamma_1 (\mathrm{id}_m\otimes \sigma_N^2\Tr)\big(R^{(k)}(ze_{11}-\gamma_0)\big)\gamma_1,\]
	and $\sum\limits_{k=1}^N\varepsilon_k^{(N)} \underset{N \to +\infty}{\longrightarrow} 0$ in probability.
\end{prop}

\begin{proof}
By \eqref{coin} and then Proposition \ref{Schur}, with the notation $\beta=ze_{11}-\gamma_0$,
\begin{align*}
\Tr(zI_N-X_N)^{-1}
&=(\Tr\otimes \Tr)\big((e_{11}\otimes I_N)R_N(\beta)\big)\\
&=(\Tr\otimes \Tr)\big((e_{11}\otimes I_{N-1})R^{(k)}(\beta)\big)\\
&+\Tr\Big(e_{11}\big(\beta-W_{kk}\gamma_1-D_{kk}\gamma_2-(\gamma_1\otimes C_k^{(k)*})R^{(k)}(\beta)(\gamma_1\otimes C_k^{(k)})\big)^{-1}\Big)\\
&+(\Tr\otimes \Tr)\Big((e_{11}\otimes I_{N-1})R^{(k)}(\beta)(\gamma_1\otimes C_k^{(k)})^{\ }\\
&\times\big(\beta-W_{kk}\gamma_1-D_{kk}\gamma_2-\gamma_1\otimes C_k^{(k)*}R^{(k)}(\beta)\gamma_1\otimes C_k^{(k)}\big)^{-1}(\gamma_1\otimes C_k^{(k)*})R^{(k)}(\beta)\Big).
\end{align*}
By traciality, the third term of the right-hand side rewrites
\[\Tr(\gamma_1\otimes C_k^{(k)*}R^{(k,1)}(\beta)\gamma_1\otimes C_k^{(k)}(\beta-W_{kk}\gamma_1-D_{kk}\gamma_2-\gamma_1\otimes C_k^{(k)*}R^{(k)}(\beta)\gamma_1\otimes C_k^{(k)})^{-1}),\]
where $R^{(k,1)}(\beta):=R^{(k)}(\beta)(e_{11}\otimes I_{N-1})R^{(k)}(\beta)=-\frac{\partial}{\partial z}R^{(k)}(\beta)$
and combines with the second term to get
\begin{align*}
\Tr & (zI_N-X_N)^{-1}\\
&=(\Tr\otimes \Tr)((e_{11}\otimes I_{N-1})R^{(k)}(\beta))\\
&+\Tr\big((e_{11}+\gamma_1\otimes C_k^{(k)*}R^{(k,1)}(\beta)\gamma_1\otimes C_k^{(k)})(\beta-W_{kk}\gamma_1-D_{kk}\gamma_2-\gamma_1\otimes C_k^{(k)*}R^{(k)}(\beta)\gamma_1\otimes C_k^{(k)})^{-1}\big).
\end{align*}
In the second term of the right-hand side, decompose 
\[e_{11}+\gamma_1\otimes C_k^{(k)*}R^{(k,1)}(\beta)\gamma_1\otimes C_k^{(k)}=(e_{11}+\gamma_1 (\mathrm{id}_m\otimes \sigma_N^2\Tr)(R^{(k,1)}(\beta))\gamma_1)-\frac{\partial}{\partial z}\Phi_k(\beta)\]
	and
\begin{align*}
  (\beta-W_{kk}\gamma_1 -D_{kk}\gamma_2 - & \gamma_1\otimes C_k^{(k)*} R^{(k)}(\beta)\gamma_1\otimes C_k^{(k)})^{-1}\\
 & = \hat{R}_k(\beta)+\hat{R}_k(\beta)\big(W_{kk}\gamma_1+\Psi_k(\beta)\big)\hat{R}_k(\beta)\\
 & + \hat{R}_k(\beta)\big(W_{kk}\gamma_1+\Psi_k(\beta)\big)\hat{R}_k(\beta)\big(W_{kk}\gamma_1+\Psi_k(\beta)\big)\\
 & \hspace{3cm} \times \Big(\beta-W_{kk}\gamma_1-D_{kk}\gamma_2-\gamma_1\otimes C_k^{(k)*}R^{(k)}(\beta)\gamma_1\otimes C_k^{(k)}\Big)^{-1},
\end{align*}
where $\Psi_k(\beta)=\Phi_k(\beta)+\gamma_1((\mathrm{id}_m\otimes\sigma_N^2\Tr)(R^{(k)}(\beta))-N\sigma_N^2(\mathrm{id}_m\otimes \tau)(r_N(\beta)))\gamma_1$
so that $\Tr(zI_N-X_N)^{-1}$ is the sum of seven terms. 
Observe that the first two terms satisfy
\begin{align*}
(\mathbb{E}_{\leq k}-\mathbb{E}_{\leq k-1})\Big[(\Tr\otimes \Tr) & \big((e_{11}\otimes I_{N-1}) R^{(k)}(\beta)\big)\\
& +\Tr\big((e_{11}+\gamma_1 (\mathrm{id}_m\otimes \sigma_N^2\Tr)(R^{(k,1)}(ze_{11}-\gamma_0))\gamma_1)\hat{R}_k(\beta)\big)\Big]=0
\end{align*}
and that the following two terms combine to get
\begin{align*}
 \Tr\big((e_{11}+\gamma_1 (\mathrm{id}_m\otimes &  \sigma_N^2\Tr)(R^{(k,1)}(\beta))\gamma_1)\hat{R}_k(\beta)(W_{kk}\gamma_1+\Psi_k(\beta))\hat{R}_k(\beta)-\frac{\partial}{\partial z}\Phi_k(\beta)\hat{R}_k(\beta)\big)\\
 & =-\frac{\partial}{\partial z}\Tr((W_{kk}\gamma_1+\Phi_k(\beta))\hat{R}_k(\beta))+\varepsilon_{k,4}^{(N)}+\varepsilon_{k,5}^{(N)},
\end{align*}
where
\[\varepsilon_{k,4}^{(N)}:=-\Tr(\gamma_1\frac{\partial}{\partial z}((\mathrm{id}_m\otimes \sigma_N^2\Tr)(R^{(k)}(\beta))-N\sigma_N^2(\mathrm{id}_m\otimes \tau)(r_N(\beta)))\gamma_1\hat{R}_k(\beta)(W_{kk}\gamma_1+\Phi_k(\beta))\hat{R}_k(\beta)),\]
and
\begin{align*}
  \varepsilon_{k,5}^{(N)} & :=\Tr\Big(\big(e_{11}+  \gamma_1 (\mathrm{id}_m\otimes \sigma_N^2\Tr)(R^{(k,1)}(\beta))\gamma_1\big)\\
  &  \hspace{3cm} \times\hat{R}_k(\beta)\big(\gamma_1((\mathrm{id}_m\otimes\sigma_N^2\Tr)(R^{(k)}(\beta))-N\sigma_N^2(\mathrm{id}_m\otimes \tau)(r_N(\beta)))\gamma_1\big)\hat{R}_k(\beta)\Big).
\end{align*}
Note that  \[\esp_{\leq k-1}\Big[-\frac{\partial}{\partial z}\Tr\big((W_{kk}\gamma_1+\Phi_k(\beta))\hat{R}_k(\beta)\big)\Big]=0.\] 
It remains to prove that the last three terms 
	\begin{align*}
\varepsilon_{k,1}^{(N)} &  := \Tr\Big((e_{11}+\gamma_1 (\mathrm{id}_m\otimes\sigma_N^2\Tr)(R^{(k,1)}(\beta))\gamma_1)\hat{R}_k(\beta)(W_{kk}\gamma_1+\Psi_k(\beta))\\
& \hspace{2cm} \times 
	\hat{R}_k(\beta)(W_{kk}\gamma_1+\Psi_k(\beta))(\beta-W_{kk}\gamma_1-D_{kk}\gamma_2-\gamma_1\otimes C_k^{(k)*}R^{(k)}(\beta)\gamma_1\otimes C_k^{(k)})^{-1}\Big),\\
	\varepsilon_{k,2}^{(N)} & :=\Tr(
	-\frac{\partial}{\partial z}\Phi_k(\beta)\hat{R}_k(\beta)(W_{kk}\gamma_1+\Psi_k(\beta))\hat{R}_k(\beta)),\\
\varepsilon_{k,3}^{(N)} & :=\Tr\Big(
	-\frac{\partial}{\partial z}\Phi_k(\beta)\hat{R}_k(\beta)(W_{kk}\gamma_1+\Psi_k(\beta))
	\hat{R}_k(\beta)(W_{kk}\gamma_1+\Psi_k(\beta))\\
	& \hspace{5cm} \times (\beta-W_{kk}\gamma_1-D_{kk}\gamma_2-\gamma_1\otimes C_k^{(k)*}R^{(k)}(\beta)\gamma_1\otimes C_k^{(k)})^{-1}\Big),
\end{align*}
are such that 
	\[\varepsilon_k^{(N)}=\left( \mathbb{E}_{\leq k}-\mathbb{E}_{\leq k-1}\right)\Big[\varepsilon_{k,1}^{(N)}+\varepsilon_{k,2}^{(N)}+\varepsilon_{k,3}^{(N)}+\varepsilon_{k,4}^{(N)}+\varepsilon_{k,5}^{(N)}\Big]\]
	satisfies $\sum_{k=1}^N\varepsilon_k^{(N)} \underset{N \to +\infty}{\longrightarrow} 0$ in probability. This is the object of Lemma \ref{epsilonN} below (assuming that the entries of $W_N$ are bounded by $\delta_N$).
\end{proof}

\begin{lem}\label{bornepdphik}
For all $p\geq 2$, 
\[\mathbb{E}\Big( \Big\| \frac{\partial}{\partial z} \Phi_k(\beta)\Big\|^{p}\Big) =O\big( N^{-\min(p/2, 2+3\varepsilon)}\big).\;\]
\end{lem}

\begin{proof}
For $\beta= ze_{11}-\gamma_0$, $z\in \C \setminus \R$, 
by Lemma \ref{lem_moment_quadratic_forms}, for $p \geq 2$,
\[
\mathbb{E}\Big( \Big\| \frac{\partial}{\partial z} \Phi_k(\beta)\Big\|^{p}\Big)\leq C \Big[\Big( 
\mathbb{E}[|W_{12}|^{4}] N \mathbb{E}\Big(\big\| R^{(k)}(\beta)\big\|^4\Big) \Big)^{p/2}
+N \mathbb{E}\Big(\big\| R^{(k)}(\beta)\big\|^{2p} \Big) \mathbb{E}[|W_{12}|^{2p}]\Big].\]
As $\mathbb{E}[|W_{12}|^{4}]=O(N^{-2})$, 
\[\Big( \mathbb{E}[|W_{12}|^{4}] N \mathbb{E}\Big(\big\| R^{(k)}(\beta)\big\|^4\Big) \Big)^{p/2}=O(N^{-p/2}).\]
Now, if $p\geq 3(1+\varepsilon)$,
\[\mathbb{E}[|W_{12}|^{2p}]=O( \delta_N^{2p-6(1+\varepsilon)}N^{-3(1+\varepsilon)})=O\big(N^{-3(1+\varepsilon )}\big).\]
Therefore,
\[N \mathbb{E}\Big(\big\| R^{(k)}(\beta)\big\|^{2p} \Big) \mathbb{E}[|W_{12}|^{2p}]=O\big(N^{-(2+3\varepsilon)}\big).\]
As a consequence, for $p\geq 3(1+\varepsilon)$,
\[\mathbb{E}\Big( \Big\| \frac{\partial}{\partial z} \Phi_k(\beta)\Big\|^{p}\Big)=O\big(N^{-\min(p/2,2+3\varepsilon)}\big).\]
If $2\leq p < 3(1+\varepsilon)$, $\mathbb{E}[|W_{12}|^{2p}]=O(N^{-p})$ and 
\[\mathbb{E}\Big( \Big\| \frac{\partial}{\partial z} \Phi_k(\beta)\Big\|^{p}\Big)=O\big(N^{-p/2}\big).\]
Note that in this case, $\min(p/2,2+3\varepsilon)=p/2$, which concludes the proof.
\end{proof}

\begin{lem}\label{normeqpsik}
$\forall q \in [2,4(1+\varepsilon))$,
\[\mathbb{E} \Big( \big\| \Psi_k(\beta)\big\|^{q} \Big)=O\big(N^{-\min(2-\varepsilon,q/2)}\big).\]
\end{lem}

\begin{proof} For $q\geq 1$, 
\begin{align*}
\|\Psi_k(\beta)\|_{L^{q}}&\leq \|\Phi_k(\beta)\|_{L^{q}}+\|\gamma_1\|^2\|(\mathrm{id}_m\otimes\sigma_N^2\Tr)(R^{(k)}(\beta))-\mathrm{id}_m\otimes\sigma_N^2\Tr)(R_N(\beta))\|_{L^{q}} \\ 
& \quad +  \|\gamma_1 \|^2 \|(\mathrm{id}_m\otimes \sigma_N^2\Tr)(R_N(\beta)) -N  \sigma_N^2(\mathrm{id}_m\otimes\tau)(r_N(\beta))\|_{L^{q}}. 
\end{align*}
Thus, for $q \in [2,4(1+\varepsilon))$, Lemmas \ref{estimfq}, \ref{compkLp} and \ref{cvRinfiniLp} readily yield that
\[  \| \Psi_k(\beta)\|_{L^{q}}= O\big(N^{-\min(1/2,(2-\varepsilon)/q)}\big),\]
which concludes the proof.
\end{proof}

\begin{lem}\label{epsilonN}
	\[\Big\|\sum_{k\geq 1}\varepsilon_{k}^{(N)}\Big\|_{L^2} \underset{N \to +\infty}{\longrightarrow} 0.\]
\end{lem}

\begin{proof}
	Using H\"older's inequality with $q\in [1,1+\varepsilon)$, $p,r\geq 1$ such that $p^{-1}+q^{-1}+r^{-1}=1$, 
	\begin{align*}
    \mathbb{E}[|\varepsilon_{k,1}^{(N)}|^2]
	&\leq  m^2\|\hat{R}_k(\beta)\|^4\mathbb{E}\Big[\|e_{11}+\gamma_1 (\mathrm{id}_m\otimes \sigma_N^2\Tr)(R^{(k,1)}(\beta))\gamma_1\|^2\|W_{kk}\gamma_1+\Psi_k(\beta)\|^4\|R_N(\beta)\|^2\Big]\\
    &\leq  m^2\|\hat{R}_k(\beta)\|^4\mathbb{E}\big[\|e_{11}+\gamma_1 (\mathrm{id}_m\otimes \sigma_N^2\Tr)(R^{(k,1)}(\beta))\gamma_1\|^{2p}\big]^{1/p}\\
	& \hspace{7cm} \times \mathbb{E}\big[\|W_{kk}\gamma_1+\Psi_k(\beta)\|^{4q}\big]^{1/q}\mathbb{E}\big[\|R_N(\beta)\|^{2r}\big]^{1/r}\\
	&\leq  m^2\|\hat{R}_k(\beta)\|^4\mathbb{E}[\|R_N(\beta)\|^{2r}]^{1/r}(1+C\|\gamma_1\|^2\mathbb{E}[\|R_N(\beta)\|^{4p}]^{1/2p})^{2}\\
	& \hspace{7cm} \times \big(\|\gamma_1\| \|W_{kk}\|_{L^{4q}}+\|\Psi_k(\beta)\|_{L^{4q}}\big)^4.
	\end{align*}
	Therefore, by Lemmas \ref{Qbound},  \ref{majRNRkLp},  \ref{puissance}  and \ref{normeqpsik}, $\esp[|\varepsilon_{k,1}^{(N)}|^2]=o(N^{-1})$, uniformly in $k$.

Using H\"older's inequality with $q\in [1,1+\varepsilon)$, $\frac{1}{p}+\frac{1}{q}=1$, 
	\[
		\mathbb{E}[|\varepsilon_{k,2}^{(N)}|^2]
		\leq m^2\|\hat{R}_k(\beta)\|^{4}\Big[\mathbb{E}[\|\frac{\partial}{\partial z}\Phi_k(\beta)\|^{2p}]\Big]^{1/p}(\|\gamma_1\| \|W_{kk}\|_{L^{2q}}+\|\Psi_k(\beta)\|_{L^{2q}})^2\]
	
	Therefore, by Lemmas \ref{Qbound},  \ref{bornepdphik},  \ref{puissance}  and \ref{normeqpsik} $\mathbb{E}[|\varepsilon_{k,2}^{(N)}|^2]=o(N^{-1})$, uniformly in $k$.
	
	Using H\"older's inequality with $q\in [1,1+\varepsilon)$, $\frac{1}{p}+\frac{1}{q}+\frac{1}{r}=1$,  by Lemmas \ref{Qbound},  \ref{majRNRkLp},  \ref{puissance}  and \ref{normeqpsik},
	\begin{align*}
		\mathbb{E}[|\varepsilon_{k,3}^{(N)}|^2]
		&\leq m^2\|\hat{R}_k(\beta)\|^4\Big[\mathbb{E}[\|\frac{\partial}{\partial z}\Phi_k(\beta)\|^{2p}]\Big]^{1/p}(\|\gamma_1\| \|W_{kk}\|_{L^{4q}}+\|\Psi_k(\beta)\|_{L^{4q}})^4 \mathbb{E}[\|R_N(\beta)\|^{2r}]^{1/r}\\
		&= \Big(O\big(N^{-\min(p,1+\varepsilon)}\big)\Big)^{1/p}\Big(O\big(N^{-\frac{1}{2q}+\frac{\varepsilon}{4q}}\big)\Big)^4
\end{align*}
so that, noticing that $\frac{2-\varepsilon}{q}>1$ for any $\varepsilon < 1/2$, $\mathbb{E}[|\varepsilon_{k,3}^{(N)}|^2]=o(N^{-1})$, uniformly in $k$.
	
	\noindent Finally, by H\"older's inequality,
	\begin{align*}
    & \mathbb{E}[|\varepsilon_{k,4}^{(N)}|^2]\\
		&\leq m^2\|\gamma_1\|^4\|\hat{R}_k(\beta)\|^4\mathbb{E}[\|\frac{\partial}{\partial z}((\mathrm{id}_m\otimes \sigma_N^2\Tr)(R^{(k)}(\beta))-N\sigma_N^2(\mathrm{id}_m\otimes \tau)(r_N(\beta)))\|^2\|W_{kk}\gamma_1+\Phi_k(\beta)\|^2]\\
&\leq 2^4 m^2\|\gamma_1\|^4\|\hat{R}_k(\beta)\|^4\left\{\mathbb{E}(\|\frac{\partial}{\partial z}((\mathrm{id}_m\otimes \sigma_N^2\Tr)(R^{(k)}(\beta) -(\mathrm{id}_m\otimes \sigma_N^2\Tr)(R_N(\beta))\Vert^4)\right.\\&\left. \hspace{1cm} + \mathbb{E}(\|\frac{\partial}{\partial z}((\mathrm{id}_m\otimes \sigma_N^2\Tr)(R_N(\beta) -N\sigma_N^2(\mathrm{id}_m\otimes \tau)(r_N(\beta)))\|^4\right\}^{1/2}\\
& \hspace{4cm} \times (\|\gamma_1\|^2\mathbb{E}[\|W_{kk}\|^4]^{1/2}+\mathbb{E}[\|\Phi_k(\beta)\|^4]^{1/2})\\
&=o(N^{-1}),
	\end{align*}
where we use Lemmas \ref{Qbound}, \ref{estimfq}, \ref{compkLp}, \ref{puissance}	and \ref{asymptoticfreenessLp} in the last line.
	
\noindent Noticing that $(\mathbb{E}_{\leq k}-\mathbb{E}_{\leq k-1})[\varepsilon_{k,5}^{(N)}]=0$, by Lemma \ref{martingalevariance} and Jensen's inequality, 
	\begin{align*} 
		\mathbb{E}\big[\big|\sum_{k=1}^N(\mathbb{E}_{\leq k}-\mathbb{E}_{\leq k-1}) \big[\varepsilon_{k,1}^{(N)} & +\varepsilon_{k,2}^{(N)}+\varepsilon_{k,3}^{(N)}+\varepsilon_{k,4}^{(N)}+\varepsilon_{k,5}^{(N)}\big]\big|^2\big]\\
		& = \sum_{k=1}^N\mathbb{E}\big[\big|(\mathbb{E}_{\leq k}-\mathbb{E}_{\leq k-1})\big[\varepsilon_{k,1}^{(N)}+\varepsilon_{k,2}^{(N)}+\varepsilon_{k,3}^{(N)}+\varepsilon_{k,4}^{(N)}\big]\big|^2\big]\\
		&\leq \sum_{k=1}^N8\sum_{j=1}^4\big(\mathbb{E}\big[\big|(\mathbb{E}_{\leq k-1}\big[\varepsilon_{k,j}^{(N)}\big]\big|^2\big]+\mathbb{E}\big[\big|(\mathbb{E}_{\leq k}\big[\varepsilon_{k,j}^{(N)}\big]\big|^2\big]\big)\\
		&\leq 16\sum_{k=1}^N\mathbb{E}\big[\big|\varepsilon_{k,1}^{(N)}\big|^2+\big|\varepsilon_{k,2}^{(N)}\big|^2+\big|\varepsilon_{k,3}^{(N)}\big|^2+\big|\varepsilon_{k,4}^{(N)}\big|^2\big]=o(1).
	\end{align*}
so that 
	$\Big\|\sum\limits_{k=1}^N\varepsilon_{k}^{(N)}\Big\|_{L^2}=o(1).$
\end{proof}


As announced at the beginning of the section, the strategy is now to apply Theorem \ref{thm_CLT_martingale} to 
\begin{align*}
\Delta_k^{(N)}&=\esp_{\leq k}\Big[-\frac{\partial}{\partial z}\Tr((W_{kk}\gamma_1+\Phi_k(ze_{11}-\gamma_0))\hat{R}_k(ze_{11}-\gamma_0))\Big]\\&=(\esp_{\leq k} -\esp_{\leq k-1})\Big[-\frac{\partial}{\partial z}\Tr((W_{kk}\gamma_1+\Phi_k(ze_{11}-\gamma_0))\hat{R}_k(ze_{11}-\gamma_0))\Big].
\end{align*}

\subsection{Verification of Lyapounov condition}
To check condition \eqref{L}, one first uses Markov inequality to get, for 
$p\in (2,4(1+\varepsilon))$
:
$$ L(\varepsilon,N) \leq \varepsilon^{2-p}\sum_{k=1}^N\|\Delta_k^{(N)}\|_{L^p}^p.$$
It is therefore sufficient to prove that 
\begin{equation}\label{eq_cond_norm_p}
\|\Delta_k^{(N)}\|_{L^p}=O(N^{-\frac{1}{2}}).
\end{equation}
By Jensen's and triangular inequalities,
for $1\leq k\leq N$,
\begin{align*}
	\|\Delta_k^{(N)}\|_{L^p}
	&=\Big\|\esp_{\leq k}[-\frac{\partial}{\partial z}\Tr((W_{kk}\gamma_1+\Phi_k(ze_{11}-\gamma_0))\hat{R}_k(ze_{11}-\gamma_0))]\Big\|_{L^p}\\
	&\leq \bigg\|\Tr(\frac{\partial}{\partial z}\Phi_k(ze_{11}-\gamma_0))\hat{R}_k(ze_{11}-\gamma_0))\bigg\|_{L^p}\\
	& \quad + \ \bigg\|\Tr((W_{kk}\gamma_1+\Phi_k(ze_{11}-\gamma_0))\frac{\partial}{\partial z}\hat{R}_k(ze_{11}-\gamma_0))\bigg\|_{L^p}\\
	&\leq m\big\|\hat{R}_k(ze_{11}-\gamma_0)\big\|\big\|
\frac{\partial}{\partial z}\Phi_k(ze_{11}-\gamma_0))
\big\|_{L^p}\\
	&\quad +m\big\|\frac{\partial}{\partial z}\hat{R}_k(ze_{11}-\gamma_0)\big\|\big(\| W_{kk}\|_{L^p}\|\gamma_1\|+  \big\| \Phi_k(ze_{11}-\gamma_0)\big\|_{L^p}\big).
\end{align*}
One deduces \eqref{eq_cond_norm_p} by using  \eqref{majorationunifhatRk}, \eqref{majhatdR}, Lemmas \ref{estimfq},  \ref{puissance} and  \ref{bornepdphik}.

\subsection{Convergence of the hook process} \label{subsectionhook}

\noindent By bilinearity, the verification of conditions \eqref{v} and \eqref{w} is equivalent to the convergence in probability of the {\em hook process}:
\[\Gamma_N(z_1,z_2):=\sum_{k=1}^N\mathbb{E}_{\leq k-1}\Big[\mathbb{E}_{\leq k}\big[\frac{\partial}{\partial z}\Tr((W_{kk}\gamma_1+\Phi_k(\beta_1))\hat R_k(\beta_1))\big]\mathbb{E}_{\leq k}\big[\frac{\partial}{\partial z}\Tr((W_{kk}\gamma_1+\Phi_k(\beta_2))\hat R_k(\beta_2))\big]\Big],\] \[\beta_1=z_1e_{11}-\gamma_0,  \beta_2=z_2e_{11}-\gamma_0, z_1,z_2\in\mathbb{C}\setminus\mathbb{R}.\]

\begin{prop} \label{hook}
 For all $z_1,z_2\in\mathbb{C}\setminus\mathbb{R}$
 \[\Gamma_N(z_1,z_2) \underset{N\to +\infty}{\longrightarrow} \Gamma(z_1,z_2)\]
in probability, where 
 \begin{align*}
 \Gamma(z_1,z_2) & := \frac{\partial^2}{\partial z_1\partial z_2}\gamma(z_1,z_2),
 \end{align*}
\begin{align*}
\gamma(z_1,z_2)
&:=-\Tr\otimes\Tr \left\{ \log \left[\mathrm{id}_m\otimes \mathrm{id}_m -\sigma^2T_{\{z_1 e_{11} -\gamma_0,z_2 e_{11} -\gamma_0\}}\right]  (I_m\otimes I_m)\right\}\\
& \quad -\Tr\otimes\Tr \left\{ \log \left[\mathrm{id}_m\otimes \mathrm{id}_m -\theta T_{\{z_1 e_{11} -\gamma_0,z_2 e_{11} -\gamma_0\}}\right]  (I_m\otimes I_m)\right\}\\
& \quad +(\tilde{\sigma}^2-\sigma^2-\theta)\Tr\otimes \Tr\{T_{\{z_1 e_{11} -\gamma_0,z_2 e_{11} -\gamma_0\}}(I_m\otimes I_m)\}\\
& \quad +\kappa/2\Tr \otimes \Tr \{T_{\{z_1 e_{11} -\gamma_0,z_2e_{11} -\gamma_0\}}^2(I_m\otimes I_m)\}.
\end{align*}
\end{prop}

In what follows, we focus on the convergence in probability of 
\begin{equation}\label{eq:hook_primitive}
\gamma_N(z_1,z_2):=\sum_{k=1}^N\mathbb{E}_{\leq k-1}\Big[\mathbb{E}_{\leq k}\big[\Tr((W_{kk}\gamma_1+\Phi_k(\beta_1))\hat R_k(\beta_1))\big]\mathbb{E}_{\leq k}\big[\Tr((W_{kk}\gamma_1+\Phi_k(\beta_2))\hat R_k(\beta_2))\big]\Big].
\end{equation}
This will be enough to establish the convergence of the hook process, as explained in Subsection \ref{sec:CV_hook_ccl}. 
Using the independence of $\Phi_k$ and $W_{kk}$, one can easily see that  
\begin{align}\label{gammadecomposition}
 & \sum_{k=1}^N \mathbb{E}_{\leq k-1}\Big[\mathbb{E}_{\leq k}\big[\Tr((W_{kk}\gamma_1+\Phi_k(\beta_1))\hat R_k(\beta_1))\big]\mathbb{E}_{\leq k}\big[\Tr((W_{kk}\gamma_1+\Phi_k(\beta_2))\hat R_k(\beta_2))\big]\Big]\nonumber\\
& =\tilde{\sigma}_N^2 \sum_{k=1}^N\Tr(\gamma_1\hat{R}_k(\beta_1))\Tr(\gamma_1\hat{R}_k(\beta_2))+\sum_{k=1}^N\mathbb{E}_{\leq k-1}\Big[\Tr(\mathbb{E}_{\leq k}\big[\Phi_k(\beta_1)\big]\hat{R}_k(\beta_1))\Tr(\mathbb{E}_{\leq k}\big[\Phi_k(\beta_2)\big]\hat{R}_k(\beta_2))\Big].
\end{align}
Thus, by Lemma \ref{quadratic forms},
\begin{align}\label{eq:hook_primitive2}
 \gamma_N(z_1,z_2) & =\tilde{\sigma}_N^2 \sum_{k=1}^N\Tr(\gamma_1\hat{R}_k(\beta_1))\Tr(\gamma_1\hat{R}_k(\beta_2))\\
& + \sigma_N^4\sum_{k=1}^N\sum_{i,j<k}\Tr\left(\gamma_1\mathbb{E}_{\leq k}[R_{ij}^{(k)}(\beta_1)]\gamma_1\hat{R}_k(\beta_1)\right)\Tr\left(\gamma_1\mathbb{E}_{\leq k}[R_{ji}^{(k)}(\beta_2)]\gamma_1\hat{R}_k(\beta_2)\right)\nonumber\\
& +  |\theta_N|^2\sum_{k=1}^N\sum_{i,j<k}\Tr\left(\gamma_1\mathbb{E}_{\leq k}[R_{ij}^{(k)}(\beta_1)]\gamma_1\hat{R}_k(\beta_1)\right)\Tr\left(\gamma_1\mathbb{E}_{\leq k}[R_{ij}^{(k)}(\beta_2)]\gamma_1\hat{R}_k(\beta_2)\right)\nonumber\\
& + \kappa_N \sum_{k=1}^N\sum_{i<k}\Tr\left(\gamma_1\mathbb{E}_{\leq k}[R_{ii}^{(k)}(\beta_1)]\gamma_1\hat{R}_k(\beta_1)\right)\Tr\left(\gamma_1\mathbb{E}_{\leq k}[R_{ii}^{(k)}(\beta_2)]\gamma_1\hat{R}_k(\beta_2)\right)+ \sum_{k=1}^N
\varepsilon_k,\nonumber
\end{align}
where  \[\sum\limits_{k=1}^N\varepsilon_k \overset{\P}{\underset{N \to +\infty}{\longrightarrow}} 0.\]
Therefore, what remains to study is the sum of four terms. They will be studied separately in the following paragraphs. The first and the fourth terms are studied very easily, whereas the second and third ones need quite long computations, making repeated use of linear algebra properties which were collected in Section \ref{sec:preliminary_results}. These second and third terms are very similar.

\subsubsection{Contribution of the first term of \eqref{eq:hook_primitive2}}\label{sec_rst_term}

\begin{lem}\label{vingt5}
Define $f(\omega,t):=\Tr(\gamma_1(\omega-t\gamma_2)^{-1})$ for $\omega\in M_m(\mathbb C),t\in\mathbb R$
such that $\omega-t\gamma_2$ is invertible in $M_m(\mathbb C)$. Then
\[
\int_{\mathbb{R}}f(\omega_N(\beta_1),t)f(\omega_N(\beta_2),t)\nu_N(dt)
\underset{N\to +\infty}{\longrightarrow}\int_{\mathbb{R}}f(\omega(\beta_1),t)f(\omega(\beta_2),t)\nu(dt).
\]
\end{lem}
\begin{proof}
\begin{align*}
\Big|\int_{\mathbb{R}} f( &\omega_N(\beta_1),t)f(\omega_N(\beta_2),t)\nu_N(dt)-\int_{\mathbb{R}}f(\omega(\beta_1),t)f(\omega(\beta_2),t)\nu(dt)\Big|\\
&= \Big|\int_{\mathbb{R}}f(\omega_N(\beta_1),t)f(\omega_N(\beta_2),t)-f(\omega(\beta_1),t)f(\omega(\beta_2),t)\nu_N(dt)\\
& \hspace{7cm} +\int_{\mathbb{R}}
f(\omega(\beta_1),t)f(\omega(\beta_2),t)[\nu_N(dt)-\nu(dt)]\Big|\\
&\le \int_{\mathbb{R}}\Big|f(\omega_N(\beta_1),t)f(\omega_N(\beta_2),t)-f(\omega(\beta_1),t)f(\omega(\beta_2),t)\Big|\nu_N(dt)\\
& \hspace{7cm}+\Big|\int_{\mathbb{R}}f(\omega(\beta_1),t)f(\omega(\beta_2),t)[\nu_N(dt)-\nu(dt)]\Big|\\
&\le \sup_{t\in\mathrm{supp}(\nu_N)}\Big|f(\omega_N(\beta_1),t)f(\omega_N(\beta_2),t)-f(\omega(\beta_1),t)f(\omega(\beta_2),t)\Big|\\
& \hspace{7cm}+\Big|\int_{\mathbb{R}}f(\omega(\beta_1),t)f(\omega(\beta_2),t)[\nu_N(dt)-\nu(dt)]\Big|.
\end{align*}
The first summand tends to zero as $N\to\infty$ by (\eqref{bornemanq}, Lemma \ref{bound} and) Lemma \ref{cvunif}, and the second summand tends to zero according to the hypothesis that $\nu_N$ weakly converges to $\nu$, the function $t\mapsto f(\omega(\beta_1),t)f(\omega(\beta_2),t)$ being bounded continuous on a compact set containing the supports of all $\nu_N, N\in \mathbb{N}$ (by Lemma \ref{bound}).
\end{proof} 
Hence, since $N^{-1}\sum_{k=1}^N\Tr(\gamma_1\hat{R}_k(\beta_1))\Tr(\gamma_1\hat{R}_k(\beta_2))=\int_{\mathbb{R}}f(\omega_N(\beta_1),t)f(\omega_N(\beta_2),t)
\nu_N(dt)$ (see \eqref{hatR}), it follows that
$$
\tilde{\sigma}_N^2\sum_{k=1}^N\Tr(\gamma_1\hat{R}_k(\beta_1))\Tr(\gamma_1\hat{R}_k(\beta_2))\underset{N\to +\infty}{\longrightarrow} \tilde{\sigma}^2\int_{\mathbb{R}}f(\omega(\beta_1),t)
f(\omega(\beta_2),t)\nu(dt).
$$

\subsubsection{Contribution of the second term of \eqref{eq:hook_primitive2}} \label{sec_2nd_term}
Denote by $T_{N,k,t}(z_1,z_2)$ the operator defined on $M_m(\mathbb{C}) \otimes M_m(\mathbb{C})$ by 
\begin{align*}
 T_{N,k,t}(z_1,z_2)(b_1\otimes b_2)& = N^{-1}\sum_{i<k} \hat{R}_i(z_1e_{11}-\gamma_0)\gamma_1b_1\otimes b_2\gamma_1\hat{R}_i(z_2e_{11}-\gamma_0)\\& \hspace{3cm}+N^{-1}t \hat{R}_k(z_1e_{11}-\gamma_0)\gamma_1b_1\otimes b_2\gamma_1\hat{R}_k(z_2e_{11}-\gamma_0).
\end{align*}
The key idea is to note that this second term can be rewritten as follows 
\begin{align*}
\sigma_N^4\sum_{k=1}^N\sum_{i,j<k}\Tr( & \gamma_1\mathbb{E}_{\leq k}[ R_{ij}^{(k)}(\beta_1)]\gamma_1\hat{R}_k(\beta_1))\Tr(\gamma_1\mathbb{E}_{\leq k}[ R_{ji}^{(k)}(\beta_2)]\gamma_1\hat{R}_k(\beta_2))\\
& =N\sigma_N^2\Tr \otimes \Tr T_N(z_1,z_2)(\nabla_k), 
\end{align*}
where, for any $k\in \{1,\ldots,N\}$, 
$$\nabla_k=\sigma_N^2\sum_{i,j<k}\mathbb{E}_{\leq k}[ R_{ij}^{(k)}(\beta_1)]\gamma_1\otimes \gamma_1\mathbb{E}_{\leq k}[ R_{ji}^{(k)}(\beta_2)] \in M_m(\mathbb{C})\otimes M_m(\mathbb{C})$$
satisfies the approximate equation
$$\nabla_k =N\sigma_N^2T_{N,k,0}(z_1,z_2)(I_m\otimes I_m+\nabla_k)+o_{L^1}(1).$$ 
In order to deduce an estimate of $\nabla_k$ we need to study the spectral radius of 
$T_{N,k,0}(z_1,z_2)$. 

\begin{prop}\label{Kz}
For  any $z \in \C\setminus \R$, there exists 
$K(z)\in (0,1)$ such that for all large $N$, for any $k\in\{1,\ldots,N\}$, for any $t\in [0;1]$, the spectral radius of 
$N\sigma_N^2T_{N,k,t}(z, \bar z)$ is smaller than $K(z)$.
\end{prop}

\begin{proof}
According to Corollary \ref{sprNsanstilde}, we know that one may find $K(z)\in (0,1)$ such that, for large $N\in \N$, the spectral radius of $N\sigma_N^2T_N(z, \bar z)$
is smaller than $K(z)$. Now, note that for any choice of $k\in\{1,\ldots,N\}$ and $t\in [0;1]$, the operator $T_{N,k,t}(z, \bar z)$ satisfies $T_{N,k,t}(z, \bar z)\leq T_N(z, \bar z)$.
Thus, by Lemma  \ref{cp}, the spectral radius of $N\sigma_N^2T_{N,k,t}(z, \bar z)$
is smaller than $K(z)$ for large $N$. Corollary \ref{Kz} follows.
\end{proof}

\begin{lem}\label{rz}
Around any $z_1\in \C \setminus \R$, there is an open set $O_{z_1}$ such that, for $z\in O_{z_1}$, for $N\geq N_{z_1}$, for any $k\leq N$, for any $t\in [0,1]$,  the spectrum of the operator $N\sigma_N^2T_{N,k,t}(z,\bar z_1):M_m(\mathbb{C}) \otimes M_m(\mathbb{C})\to M_m(\mathbb{C}) \otimes M_m(\mathbb{C})$ 
is included in the open unit disk and $\{(\mathrm{id}_m\otimes \mathrm{id}_m - N\sigma_N^2T_{N,k,t}(z,\bar z_1))^{-1},z\in O_{z_1},N\geq N_{z_1},k\leq N,t\in [0,1]\}$ is bounded.
\end{lem}

Without loss of generality, one may assume that $z\in O_{z_1}$ satisfy $\Im z\geq c\Im z_1$ for some $c>0$.


\begin{proof}
By \eqref{majorationunifhatRk}, $\{T_{N,k,t}(z_1,\bar z_1), N \in \N, k\leq N, t\in [0,1]\} $ is included in a  centered ball with some  radius  $r_{z_1}$ in $M_{m^4}(\C)$. Moreover, by \eqref{majorationunifhatRk} and Lemma \ref{lem_resolvent_identity}, the family of functions $\{z\mapsto T_{N,k,t}(z,\bar z_1), N\in \N, k\leq N,t\in [0,1]\}$ is equicontinuous 
on any compact set of $\mathbb{C}\setminus \mathbb{R}$.
Therefore, the  assertions easily follow from Proposition \ref{Kz} by using the  uniform continuity of the spectral radius and  of the norm  on compact sets  of $M_{m^4}(\C)$ and   Lemma  \ref{majnormerayon}.
\end{proof}


\begin{prop}\label{analog}
Fix $z_1 \in \C \setminus \R$ 
and $z\in O_{z_1}$ (defined in Lemma \ref{rz}). Set  $\beta=ze_{11}-\gamma_0$ and $\beta_1^*=\bar z_1e_{11}-\gamma_0$ in $M_m(\C)$.
The following convergence holds in probability:
\begin{align}
 \sigma_N^4 & \sum_{k=1}^N\sum_{i,j<k}\Tr(\gamma_1\mathbb{E}_{\leq k}[ R_{ij}^{(k)}(\beta)]\gamma_1\hat{R}_k(\beta))\Tr(\gamma_1\mathbb{E}_{\leq k}[ R_{ji}^{(k)}
(\beta_1^*)]\gamma_1\hat{R}_k(\beta_1^*))\nonumber\\
& \rightarrow_{N\rightarrow +\infty}-{\rm Tr}\otimes\Tr\left\{\log\left[\mathrm{id}_m\otimes\mathrm{id}_m-\sigma^2T_{\{\beta,\beta_1^*\}}\right](I_m\otimes I_m)\right\}\nonumber\\
& \hspace{7cm} -\sigma^2{\rm Tr}\otimes\Tr\left\{T_{\{\beta,\beta_1^*\}}(I_m\otimes I_m)\right\}.\nonumber
\end{align}
\end{prop}

\begin{proof}
Define for any $k\in \{1,\ldots,N\}$, 
\begin{equation}\label{defnabl}\nabla_k=\sigma_N^2\sum_{i,j<k}\mathbb{E}_{\leq k}[ R_{ij}^{(k)}(\beta)]\gamma_1\otimes \gamma_1\mathbb{E}_{\leq k}[ R_{ji}^{(k)}(\beta_1^*)] \in M_m(\mathbb{C})\otimes M_m(\mathbb{C}).\end{equation}
Note that 
\begin{align}\label{trick}
 \sigma_N^4\sum_{k=1}^N\sum_{i,j<k} \Tr(\gamma_1\mathbb{E}_{\leq k}[ R_{ij}^{(k)}(\beta)]\gamma_1\hat{R}_k(\beta)) \Tr(\gamma_1\mathbb{E}_{\leq k} & [ R_{ji}^{(k)}(\beta_1^*)]\gamma_1\hat{R}_k(\beta_1^*))\nonumber\\
 & =N\sigma_N^2\Tr \otimes \Tr T_{N,k,0}(z,\bar z_1)(\nabla_k).
\end{align}

\begin{lem}\label{nablak}
With the notation of Proposition \ref{analog}, 
\begin{align*}
\nabla_k
&=\sigma_N^2\sum_{i<k}\hat{R}_i(\beta) \gamma_1 \otimes \gamma_1 \mathbb{E}_{\leq k-1}\big(R_{ii}^{(k)}(\beta_1^*)\big)\\
& \hspace{1cm} +\sigma_N^2\sum_{i,j,l<k}W_{il}\hat {R}_i(\beta)\gamma_1\mathbb{E}_{\leq k-1}\big( R_{lj}^{(kil)}(\beta)\big)\gamma_1\otimes \gamma_1\mathbb{E}_{\leq k-1}\big( R_{ji}^{(k)}(\beta_1^*)\big)+o_{L^2}^{(u)}(1).
\end{align*}
\end{lem}

\begin{proof}
By definition of $R^{(k)}(\beta)$, for all $i,j \neq k$,
\[(\beta_1-D_{ii}\gamma_2)R_{ij} ^{(k)}(\beta)=\delta_{ij}I_m+\sum_{l\neq k}W_{il}\gamma_1R_{lj}^{(k)}(\beta).\]
We want to remove the dependence between $W_{il}$ and $R^{(k)}(\beta)$, using \eqref{remove}:
\begin{align*}
(\beta-D_{ii}\gamma_2)R_{ij}^{(k)}(\beta) = \delta_{ij}I_m & + \sum_{l\neq k}\Big\{W_{il}\gamma_1R_{lj}^{(kil)}(\beta)\\
&+(1-\tfrac{1}{2}\delta_{il})\Big(W_{il}^2\gamma_1R_{li}^{(kil)}(\beta)\gamma_1R_{lj}^{(k)}(\beta)+|W_{il}|^2\gamma_1R_{ll}^{(kil)}(\beta)\gamma_1R_{ij}^{(k)}(\beta)\Big)\Big\}.
\end{align*}
Hence, noticing that 
\[N^{-1}\sum_{l=1}^N \hat R_l(\beta)= \mathrm{id}_m\otimes \tau_{N}((\beta\otimes 1_{\mathcal{A}}-\gamma_1\otimes s_{N}-\gamma_2\otimes D_N)^{-1}),\]
we can deduce that
\begin{align}
 ({\omega}_N(\beta)- D_{ii}\gamma_2) R_{ij}^{(k)}(\beta) =\delta_{ij}I_m  & + \sum_{l\neq k}W_{il}\gamma_1 R_{lj}^{(kil)}(\beta)\nonumber \\
& +\sum_{l\neq k}(1-\tfrac{1}{2}\delta_{il})W_{il}^2\gamma_1 R_{li}^{(kil)}(\beta)\gamma_1 R_{lj}^{(k)}(\beta)\nonumber \\
&+\sum_{l\neq k}\big[(1-\tfrac{1}{2}\delta_{il})|W_{il}|^2\gamma_1 R_{ll}^{(kil)}(\beta)\gamma_1-\sigma_N^2\gamma_1\hat R_l(\beta))\gamma_1\big] R_{ij}^{(k)}(\beta)\nonumber
\\&- \sigma_N^2\gamma_1\hat R_k(\beta)\gamma_1 R_{ij}^{(k)}(\beta), \nonumber\end{align}
and
\begin{align}
R_{ij}^{(k)}(\beta) =\delta_{ij}\hat{R}_i(\beta)
& + \sum_{l\neq k}W_{il}\hat {R}_i(\beta)\gamma_1 R_{lj}^{(kil)}(\beta)\nonumber\\
& + \sum_{l\neq k}(1-\tfrac{1}{2}\delta_{il})W_{il}^2\hat{R}_i(\beta)\gamma_1 R_{li}^{(kil)}(\beta)\gamma_1 R_{lj}^{(k)}(\beta)\nonumber\\
& +\sum_{l\neq k}\hat{R}_i(\beta)\big[(1-\tfrac{1}{2}\delta_{il})|W_{il}|^2\gamma_1 R_{ll}^{(kil)}(\beta)\gamma_1-\sigma_N^2\gamma_1\hat R_l(\beta)\gamma_1\big] R_{ij}^{(k)}(\beta)
\\
&- \sigma_N^2 \hat{R}_i(\beta)\gamma_1\hat R_k(\beta)\gamma_1 R_{ij}^{(k)}(\beta). \nonumber
\end{align}
Therefore, for $i,j<k$, 
\begin{align*}
 \mathbb{E}_{\leq k-1} &\big( R_{ij}^{(k)}(\beta) \big)\gamma_1 \otimes \gamma_1 \mathbb{E}_{\leq k-1}\big( R_{ji}^{(k)}(\beta_1^*)\big)\\
& =
\delta_{ij}\hat{R}_i(\beta) \gamma_1 \otimes \gamma_1 \mathbb{E}_{\leq k-1}\big( R_{ji}^{(k)}(\beta_1^*)\big)
+ \sum_{l<k}W_{il}\hat {R}_i(\beta)\gamma_1 \mathbb{E}_{\leq k-1}\big( R_{lj}^{(kil)}(\beta)\big)\gamma_1\otimes \gamma_1 \mathbb{E}_{\leq k-1}\big( R_{ji}^{(k)}(\beta_1^*)\big)\\
& \hspace{0.2cm} + \sum_{l\neq k}\big(1-\tfrac{1}{2}\delta_{il}\big)
\hat{R}_i(\beta)\gamma_1 \mathbb{E}_{\leq k-1}\Big(W_{il}^2
R_{li}^{(kil)}(\beta_1)\gamma_1 R_{lj}^{(k)}(\beta)\Big) \gamma_1 \otimes \gamma_1 \mathbb{E}_{\leq k-1}\big( R_{ji}^{(k)}(\beta_1^*)\big) \\
& \hspace{0.2cm} + \sum_{l\neq k}\hat{R}_i(\beta)\gamma_1 \mathbb{E}_{\leq k-1}\Big(\big[(1-\tfrac{1}{2}\delta_{il})|W_{il}|^2-\sigma_N^2 \big] \hat R_l(\beta)\gamma_1  R_{ij}^{(k)}(\beta)\Big)
\gamma_1 \otimes \gamma_1 \mathbb{E}_{\leq k-1}\big( R_{ji}^{(k)}(\beta_1^*)\big) 
\\
& \hspace{0.2cm} + \sum_{l\neq k}\hat{R}_i(\beta)\gamma_1 \mathbb{E}_{\leq k-1}\Big(\big(1-\tfrac{1}{2}\delta_{il}\big)|W_{il}|^2\big[ R_{ll}^{(kil)}(\beta)- \hat R_l(\beta)\big] \gamma_1  R_{ij}^{(k)}(\beta)\Big)
\gamma_1 \otimes \gamma_1 \mathbb{E}_{\leq k-1}\big( R_{ji}^{(k)}(\beta_1^*)\big) \\
& \hspace{0.2cm} - \sigma_N^2 \hat{R}_i(\beta)\gamma_1\hat R_k(\beta)\gamma_1\mathbb{E}_{\leq k-1}\big( R_{ij}^{(k)}(\beta) \big)\gamma_1 \otimes \gamma_1 \mathbb{E}_{\leq k-1}\big( R_{ji}^{(k)}(\beta_1^*)\big)\\
& =
\delta_{ij}\hat{R}_i(\beta) \gamma_1 \otimes \gamma_1 \mathbb{E}_{\leq k-1}\big( R_{ji}^{(k)}(\beta_1^*)\big)
+ \sum_{l<k}W_{il}\hat {R}_i(\beta)\gamma_1 \mathbb{E}_{\leq k-1}\big( R_{lj}^{(kil)}(\beta)\big)\gamma_1\otimes \gamma_1 \mathbb{E}_{\leq k-1}\big( R_{ji}^{(k)}(\beta_1^*)\big)\\
& \hspace{0.2cm} +I_{k,i,j} +II_{k,i,j}+III_{k,i,j}+IV_{k,i,j}.
\end{align*}
We are going to prove that the contribution of the  last four terms of the right-hand side is negligible. We have  
$$\hspace*{-0.7 cm}\Vert  \sigma_N^2\sum_{i,j<k}I_{k,i,j}\Vert_{L^2}
\leq \sigma_N^2\sum_{i<k}\Vert \sum_{j<k}I_{k,i,j}\Vert_{L^2}$$
with
\begin{align*}
\Vert \sum_{j<k} &I_{k,i,j}\Vert^2_{L^2}\\
& = \mathbb{E}\Big( \big\|
\sum_{j<k}
\sum_{l\neq k}(1-\tfrac{1}{2}\delta_{il})
\hat{R}_i(\beta)\gamma_1 \mathbb{E}_{\leq k-1}\big(W_{il}^2
R_{li}^{(kil)}(\beta)\gamma_1 R_{lj}^{(k)}(\beta)\big)\gamma_1 \otimes \gamma_1 \mathbb{E}_{\leq k-1}\big( R_{ji}^{(k)}(\beta_1^*)\big) \big\|^2\Big).
\end{align*}
Using the fact  that for any matrices $A$, $B$ and $X$, we have 
\begin{equation}\label{correspondance} 
vec \left[AXB\right]=\left(B^T\otimes A \right) vec X \end{equation}
where $vec X = (X_{11}, \ldots, X_{m1}, \ldots, X_{1m}, \ldots, X_{mm})^T,$ and that there exist permutation matrices $P$ and $Q$ such that  for any matrices $A$, $B$, $$A\otimes B=P(B\otimes A)Q$$ (see \cite[Chapter 4]{HJ}), we have
\begin{align*}
 \Big\|
\sum_{j<k}
\sum_{l\neq k} & \big(1-\tfrac{1}{2}\delta_{il}\big) 
\hat{R}_i(\beta)\gamma_1\mathbb{E}_{\leq k-1}\big(W_{il}^2
R_{li}^{(kil)}(\beta)\gamma_1 R_{lj}^{(k)}(\beta)\big) \gamma_1 \otimes \gamma_1 \mathbb{E}_{\leq k-1}\big( R_{ji}^{(k)}(\beta_1^*)\big) \Big\|\\
& =\sup_{\begin{array}{ll} \scriptstyle{X\in M_m(\C)}\\ \scriptstyle{\Tr(XX^*) =1} \end{array}}\Big\| \sum_{l\neq k} \sum_{j<k}\gamma_1\mathbb{E}_{\leq k-1}\big( R_{ji}^{(k)}(\beta_1^*)\big)X\gamma_1^T \\
& \hspace{3.5cm} \times   \mathbb{E}_{\leq k-1}\Big( (R^{(k)}(\beta)^T)_{jl}( R^{(kil)}(\beta)^T)_{il} (1-\tfrac{1}{2}\delta_{il}) W_{il}^2\Big) \gamma_1^T
\hat{R}_i(\beta)^T \Big\|_{HS}.
\end{align*}

Now, for any $X\in M_m(\C)$ such that $\Tr(XX^*) =1 $, denoting by  ${\bf P}$ the  orthogonal projection onto the subspace generated by the first $k-1$ vectors of the canonical basis of $\C^{N-1}$ and using Lemma \ref{Qbound},
Lemma \ref{prelim} and Lemma \ref{alta lemma}, we have the following inequalities where $C(z_1)$ denotes a positive constant depending on $z_1$ and $m$ which may vary from line to line, 
\begin{align*}
 \Big\| & \sum_{l\neq k}  \sum_{j<k}\gamma_1 \mathbb{E}_{\leq k-1}\big( R_{ji}^{(k)}(\beta_1^*)\big) X \gamma_1^T  \mathbb{E}_{\leq k-1}\Big( (R^{(k)}(\beta)^T)_{jl}( R^{(kil)}(\beta)^T)_{il} (1-\tfrac{1}{2}\delta_{il}) W_{il}^2\Big) \gamma_1^T
\hat{R}_i(\beta)^T \Big\|_{HS}\\
&=\Big\| \sum_{l\neq k} \gamma_1\mathbb{E}_{\leq k-1}\Big(\sum_{j<k} \mathbb{E}_{\leq k-1}\big( R_{ji}^{(k)}(\beta_1^*)\big) X \gamma_1^T   (R^{(k)}(\beta)^T)_{jl}( R^{(kil)}(\beta)^T)_{il} (1-\tfrac{1}{2}\delta_{il}) W_{il}^2\Big) \gamma_1^T
\hat{R}_i(\beta)^T \Big\|_{HS}\\
&\leq  \Vert \gamma_1 \Vert^2 \sum_{l\neq k}\mathbb{E}_{\leq k-1}\Big( \Big\| \sum_{j<k} \mathbb{E}_{\leq k-1}\big( R_{ji}^{(k)}(\beta_1^*)\big) X \gamma_1^T   (R^{(k)}(\beta)^T)_{jl}( R^{(kil)}(\beta)^T)_{il}\Big\|(1-\tfrac{1}{2}\delta_{il}) |W_{il}|^2 \Big)\\
 & \hspace{13cm} \times 
\big\|\hat{R}_i(\beta)^T \big\|_{HS}\\
&= \Vert \gamma_1 \Vert^2 \sum_{l\neq k} \mathbb{E}_{\leq k-1}\Big( \Big\| \Big[ \mathbb{E}_{\leq k-1}\big(\Theta( R^{(k)}(\beta_1^*))\big) (X \gamma_1^T\otimes  {\bf P}) R^{(k)}(\beta)^T\Big]_{il} (R^{(kil)}(\beta_1)^T)_{il}\Big\| \big(1-\tfrac{1}{2}\delta_{il}\big) |W_{il}|^2 \Big)\\
& \hspace{13cm} \times \big\|\hat{R}_i(\beta)^T \big\|_{HS}
\end{align*}
\begin{align*}
&\leq  C(z_1)\Vert \gamma_1 \Vert^2  \mathbb{E}_{\leq k-1}\Big\{\Big(\sum_{l\neq k} \Big\| \Big[ \mathbb{E}_{\leq k-1}\big( \Theta(R^{(k)}(\beta_1^*))\big) (X \gamma_1^T\otimes   {\bf P}) R^{(k)}(\beta)^T\Big]_{il}\Big\|^2 \Big)^{1/2}  \\ 
& \hspace{7.5cm}\times
\Big(\sum_{l\neq k} \big\|
(R^{(kil)}(\beta)^T)_{il}\big\|^2 (1-\tfrac{1}{2}\delta_{il})^2 |W_{il}|^4\Big)^{1/2}\Big\}\\
&\leq C(z_1)\Vert \gamma_1 \Vert^2  \mathbb{E}_{\leq k-1}\Big\{\Big\| \mathbb{E}_{\leq k-1}\big( \Theta(R^{(k)}(\beta_1^*))\big)
 (X \gamma_1^T\otimes   {\bf P}) R^{(k)}(\beta)^T\Big\|\\
 & \hspace{7.5cm} \times 
\Big(\sum_{l\neq k} \big\|
(R^{(kil)}(\beta)^T)_{il}\big\|^2 (1-\tfrac{1}{2}\delta_{il})^2 |W_{il}|^4\Big)^{1/2}\Big\}\\
&\leq C(z_1)  \mathbb{E}_{\leq k-1}\Big\{\mathbb{E}_{\leq k-1}\Big(Q\big(\frac{1}{\vert \Im z_1 \vert }, \Vert W_N^{(k)}\Vert\big) \Big) 
Q\big(\frac{1}{\vert \Im z \vert }, \Vert W_N^{(k)}\Vert\big)\\
& \hspace{7.5cm} \times \Big(\sum_{l\neq k} Q(\frac{1}{\vert \Im z \vert }, \Vert W_N^{(kil)}\Vert)^2 (1-\tfrac{1}{2}\delta_{il})^2 |W_{il}|^4\Big)^{1/2}\Big\},
\end{align*}
where $\Theta$ denotes the partial transpose map
on $M_m(\C)\otimes M_N(\C)$, defined by 
$\Theta(\sum_k A_k\otimes B_k) :=\sum_k A_k\otimes B_k^T$ and we use Proposition \ref{partialtranspose}.
Hence  \begin{align*}
\Vert \sum_{j<k}I_{k,i,j}\Vert^2_{L^2}
&\leq C(z_1)  \sum_{l\neq k}\mathbb{E}\Big( \mathbb{E}_{\leq k-1}\Big(\big[Q\big(\frac{1}{\vert \Im z_2 \vert }, \Vert W_N^{(k)}\Vert\big)\big]^2 \Big) 
\big[Q\big(\frac{1}{\vert \Im z \vert }, \Vert W_N^{(k)}\Vert\big)\big]^2 Q\big(\frac{1}{|\Im z|}, \Vert W_N^{(kil)}\Vert\big)^2\Big.\\
&\hspace{8cm}\times (1-\tfrac{1}{2}\delta_{il})^2 |W_{il}|^4\Big)\\
&\leq C(z_1)  \sum_{l\neq k}\mathbb{E}\Big( \mathbb{E}_{\leq k-1}\Big(\big[Q\big(\frac{1}{\vert \Im z_2 \vert },  \Vert W_N^{(kil)}\Vert+2\delta_N\big)\big]^2 \Big) 
\big[Q\big(\frac{1}{\vert \Im z \vert },  \Vert W_N^{(kil)}\Vert+2\delta_N\big)\big]^2\Big.\\
& \hspace{7cm}\times  Q\big(\frac{1}{|\Im z|}, \Vert W_N^{(kil)}\Vert\big)^2\big(1-\tfrac{1}{2}\delta_{il}\big)^2 |W_{il}|^4\Big)
\end{align*}
\begin{align*}
&= C(z_1)  \sum_{l\neq k}\mathbb{E}\Big( \mathbb{E}_{\leq k-1}\Big(\big[Q\big(\frac{1}{\vert \Im z_1 \vert },  \Vert W_N^{(kil)}\Vert+2\delta_N\big)\big]^2 \Big) 
\big[Q\big(\frac{1}{\vert \Im z \vert },  \Vert W_N^{(kil)}\Vert+2\delta_N\big)\big]^2\\
&\hspace{7cm}\times  Q\big(\frac{1}{|\Im z|}, \Vert W_N^{(kil)}\Vert\big)^2\Big)\big(1-\tfrac{1}{2}\delta_{il}\big)^2 \mathbb{E} \big( |W_{il}|^4\big)\\
&\leq  C(z_1)  \sum_{l\neq k}\mathbb{E}\Big( \mathbb{E}_{\leq k-1}\Big(\big[Q\big(\frac{1}{\vert \Im z_1 \vert },  \Vert W_N\Vert+4\delta_N\big)\big]^2 \Big) 
\big[Q\big(\frac{1}{\vert \Im z \vert },  \Vert W_N\Vert+4\delta_N\big)\big]^2\\
& \hspace{7cm} \times  Q\big(\frac{1}{|\Im z|}, \Vert W_N\Vert +2 \delta_N\big)^2\Big)\big(1-\tfrac{1}{2}\delta_{il}\big)^2 \mathbb{E} \big( |W_{il}|^4\big)\\
&\leq C(z_1)  \sum_{l\neq k} \big(1-\tfrac{1}{2}\delta_{il}\big)^2\mathbb{E} \big( |W_{il}|^4\big),
\end{align*}
where we use Proposition \ref{boundnorm} in the last line.
Thus, 
\[\Vert  \sigma_N^2\sum_{i,j<k}I_{k,i,j}\Vert_{L^2}\leq C(z_1) \sigma_N^2 \sum_{i<k} \Big\{ \sum_{l\neq k} \mathbb{E} \Big(\vert W_{il} \vert^4\Big)\Big\}^{1/2}\leq C(z_1)N\sigma_N^2(\delta_N^2 \tilde{\sigma}_N^2 + (N-2)m_N)^{1/2} =O(N^{-1/2}).\]
Now, similarly, using again \eqref{correspondance},
\[\Vert  \sigma_N^2\sum_{i,j<k}II_{k,i,j}\Vert_{L^2}
\leq  \sigma_N^2\sum_{i<k}\| {\mathcal E}(i,k)\|_{L^2}\]
where 
\begin{align*}
 {\mathcal E}(i,k)
 & = \sup_{\begin{array}{l} \scriptstyle{X\in M_m(\C)}\\ \scriptstyle{\Tr(XX^*) =1} \end{array}} \Big\| \gamma_1 \mathbb{E}_{\leq k-1}\Big\{\Big[ \mathbb{E}_{\leq k-1}\big( \Theta(R^{(k)}(\beta_1^*))\big) (X\gamma_1^T \otimes {\bf P})  R^{(k)}(\beta)^T\Big]_{ii}\gamma_1^T( \hat R_l(\beta))^T\\
 & \hspace{7.5cm} \times \Big(\sum_{l\neq k} [(1-\tfrac{1}{2}\delta_{il}) \vert W_{il}\vert^2-\sigma_N^2]\Big)\Big\} \gamma_1^T
\hat{R}_i(\beta)^T \Big\|_{HS}
\end{align*}
\begin{align*}
& \leq
\| \gamma_1\|^3  \mathbb{E}_{\leq k-1}\Big\{\Big\|\Big[ \mathbb{E}_{\leq k-1}\Big( \Theta(R^{(k)}(\beta_1^*))\Big) (X\gamma_1^T \otimes {\bf P})  R^{(k)}(\beta)^T\Big]_{ii}\Big\|\Big|\sum_{l\neq k} [(1-\tfrac{1}{2}\delta_{il}) \vert W_{il}\vert^2-\sigma_N^2]\Big|\Big\}\\
& \hspace{10.5cm} \times \big\|( \hat R_l(\beta))^T\big\|\big\|
\hat{R}_i(\beta)^T \big\|_{HS}\\
& \leq C(z_1) \mathbb{E}_{\leq k-1}\Big\{\mathbb{E}_{\leq k-1}\Big(Q(\frac{1}{\vert \Im z_1 \vert }, \Vert W_N^{(k)}\Vert)\Big)Q(\frac{1}{\vert \Im z\vert }, \Vert W_N^{(k)}\Vert)\Big|\sum_{l\neq k} [(1-\tfrac{1}{2}\delta_{il}) \vert W_{il}\vert^2-\sigma_N^2]\Big|\Big\}.
\end{align*}
Hence, using H\"older's inequality and Proposition \ref{boundnorm}, we obtain that 
 \[
\big\Vert  \sigma_N^2\sum_{i,j<k}II_{k,i,j}\big\Vert_{L^2}
\leq C(z_1)\sigma_N^2\sum_{i<k}\mathbb{E}\Big(\Big|\sum_{l\neq k} [(1-\tfrac{1}{2}\delta_{il}) \vert W_{il}\vert^2-\sigma_N^2]\Big|^4\Big)^{1/4}=o(1),\]
where, in the last equality,  we use that, uniformly in $i$ and $l$, 
\[\mathbb{E}\Big( [(1-\frac{1}{2}\delta_{il_1}) \vert W_{il_1}\vert^2-\sigma_N^2]\Big)=\delta_{il_1}O(1/N),\]
and for $p=2,3,4,$
\[\mathbb{E}\Big(\Big\{ [(1-\tfrac{1}{2}\delta_{il_1}) \vert W_{il_1}\vert^2-\sigma_N^2\Big\}^p\Big)=o(1/N).\]
Similarly, using again \eqref{correspondance},
\begin{align*}\lefteqn{\Vert  \sum_{j<k}III_{k,i,j}\Vert } 
\\&\leq\displaystyle{\sup_{\begin{array}{ll} \scriptstyle{X\in M_m(\C)}\\ \scriptstyle{\Tr(XX^*) =1} \end{array}}} \| \gamma_1\|^3 
\big\|
\hat{R}_i(\beta)^T \big\|_{HS} \\
& \hspace{2cm} \times \mathbb{E}_{\leq k-1}\Big\{
\Big\|\Big[ \mathbb{E}_{\leq k-1}\big( \Theta(R^{(k)}(\beta_1^*))\big) (X\gamma_1^T \otimes  {\bf P})  R^{(k)}(\beta)^T\Big]_{ii}\Big\|\Big\|\sum_{l\neq k} (1-\tfrac{1}{2}\delta_{il}) \vert W_{il}\vert^2 \Delta_{ilk}^T \Big\|\Big\}\\
&\leq 
C(z_1)  \mathbb{E}_{\leq k-1}\Big\{\mathbb{E}_{\leq k-1}\Big(Q(\frac{1}{\vert \Im z_1\vert }, \Vert W_N^{(k)} \Vert) \Big) Q(\frac{1}{\vert \Im z\vert }, \Vert W_N^{(k)}\Vert)\sum_{l\neq k}  (1-\tfrac{1}{2}\delta_{il}) \vert W_{il}\vert^2 \Big\|  \Delta_{ilk}^T \Big\|\Big\},
\end{align*}
where $\Delta_{ilk}=R_{ll}^{(kil)}(\beta)- \hat R_l(\beta))$.
Thus, 
\begin{align*}
 \big\Vert \sum_{j<k} & III_{k,i,j}\big\Vert_{L_2}^2\\
 & \leq 
C(z_1) \mathbb{E}\Big[\mathbb{E}_{\leq k-1}\Big(Q^2(\frac{1}{\vert \Im z_1\vert }, \Vert W_N^{(k)} \Vert) \Big) Q^2(\frac{1}{\vert \Im z\vert }, \Vert W_N^{(k)}\Vert)\\
& \hspace{5cm}\times \sum_{l,l' \neq k} 
(1-\tfrac{1}{2}\delta_{il}) \vert W_{il}\vert^2  (1-\frac{1}{2}\delta_{il'}) \vert W_{il'}\vert^2
\big\| \Delta_{ilk}^T \big\| \big\| \Delta_{il'k}^T \big\|\Big]\\
& \leq C(z_1)\sum_{l,l' \neq k} \Big\{\mathbb{E}\Big[ \Big(\mathbb{E}_{\leq k-1}\Big(Q^2(\frac{1}{\vert \Im z_1\vert }, \Vert W_N^{(k)} \Vert) \Big) Q^2(\frac{1}{\vert \Im z\vert }, \Vert W_N^{(k)}\Vert)\Big)^{p}  
\big\| \Delta_{ilk}^T \big\|^p \big\| \Delta_{il'k}^T \big\|^p \Big]\Big\}^{1/p}\\
& \hspace{5cm}\times \Big\{ \mathbb{E}\Big(  (1-\tfrac{1}{2}\delta_{il})^q \vert W_{il}\vert^{2q} (1-\frac{1}{2}\delta_{il'})^q \vert W_{il'}\vert^{2q}\Big)\Big\}^{1/q},
\end{align*}
where $q=1+\varepsilon$ and $p^{-1}+q^{-1}=1$.
\eqref{hyp:offdiagonal} and \eqref{hyp:diagonal} readily yield that 
\[\Big\{ \mathbb{E}\Big(  (1-\tfrac{1}{2}\delta_{il})^q \vert W_{il}\vert^{2q}) (1-\frac{1}{2}\delta_{il'})^q \vert W_{il'}\vert^{2q})\Big)\Big\}^{1/q}
=O(N^{-2})\]
uniformly in $i,l,l'$ so that $\Vert \sum_{j<k}III_{k,i,j}\Vert_{L_2}
=o(1)$ uniformly in $i,k$ by using Remark \ref{kilhat} and Proposition \ref{boundnorm}. 
Therefore 
\[\Big\|\sigma_N^2 \sum_{i,j<k}III_{k,i,j}\Big\|_{L^2}=o(1).\]
Finally, similarly,
\begin{align*}
\big\Vert  \sum_{j<k}IV_{k,i,j}\big\Vert &\leq\sigma_N^2\displaystyle{\sup_{ \begin{array}{ll} \scriptstyle{X\in M_m(\C)}\\ \scriptstyle{\Tr(XX^*) =1} \end{array}}} \big\| \gamma_1\big\|^3 \big\|
\hat{R}_i(\beta) \big\|
\big\|
\hat{R}_k(\beta)^T \big\|_{HS} \\
& \hspace{2cm} \times 
\Big\|\Big[\mathbb{E}_{\leq k-1}\big( \Theta(R^{(k)}(\beta_1^*))\big) (X\gamma_1^T \otimes  {\bf P})   \mathbb{E}_{\leq k-1}\big( R^{(k)}(\beta)^T\big)\Big]_{ii}\Big\|\\ 
&\leq 
C(z_1)\sigma_N^2 \mathbb{E}_{\leq k-1}\big(R^{(k)}(\beta_1^*)\big) \mathbb{E}_{\leq k-1}\big(R^{(k)}(\beta)\big)
\\
&\leq 
C(z_1)\sigma_N^2 \mathbb{E}_{\leq k-1}\Big(Q(\frac{1}{\vert \Im z_1\vert }, \Vert W_N^{(k)} \Vert) \Big) \mathbb{E}_{\leq k-1}\Big( Q(\frac{1}{\vert \Im z\vert }, \Vert W_N^{(k)}\Vert)\Big),
 \end{align*}
so that by using Proposition \ref{boundnorm}$, \Vert  \sum_{j<k}IV_{k,i,j}\Vert_{L^2}=O(\sigma_N^2).$
Thus  $$\hspace*{-0.7 cm}\big\Vert  \sigma_N^2\sum_{i,j<k}IV_{k,i,j}\big\Vert_{L^2}
\leq \sigma_N^2\sum_{i<k}\Vert \sum_{j<k}IV_{k,i,j}\Vert_{L^2}=O(N^{-1}).$$
Lemma \ref{nablak} readily follows.
\end{proof}

\begin{lem}\label{equationapprochee} 
With the notation of Proposition \ref{analog},
\begin{equation*}
\nabla_k = N\sigma_N^2T_{N,k,0}(z,\bar z_1)(I_m\otimes I_m+\nabla_k)+o_{L^1}(1).\end{equation*}
\end{lem}

\begin{proof}
Let us consider the second term of  the right hand side of Lemma \ref{nablak}. Using \eqref{remove}, we have 
\begin{align}
 \sigma_N^2   \sum_{i,j,l<k} & W_{il}\hat {R}_i(\beta)\gamma_1 \mathbb{E}_{\leq k-1}\left( R_{lj}^{(kil)}(\beta)\right)\gamma_1\otimes \gamma_1 \mathbb{E}_{\leq k-1}\left( R_{ji}^{(k)}(\beta_1^*)\right)\nonumber\\
 & = \label{1lemme55} \sigma_N^2 \sum_{i,j,l<k}W_{il}\hat {R}_i(\beta)\gamma_1 \mathbb{E}_{\leq k-1}\left( R_{lj}^{(kil)}(\beta)\right)\gamma_1\otimes \gamma_1 \mathbb{E}_{\leq k-1}\left( R_{ji}^{(kil)}(\beta_1^*)\right)\\
 & \label{2lemme55}+\sigma_N^2 \sum_{i,j,l<k}(1-\tfrac{1}{2}\delta_{il})W_{il}^2 \hat {R}_i(\beta)\gamma_1\mathbb{E}_{\leq k}[ R_{lj}^{(kil)}(\beta)]\gamma_1\otimes\gamma_1\mathbb{E}_{\leq k}[ R_{ji}^{(kil)}(\beta_1^*)\gamma_1 R_{li}^{(k)}(\beta_1^*)]\\
 & \label{3lemme55}+\sigma_N^2 \sum_{i,j,l<k}(1-\tfrac{1}{2}\delta_{il})|W_{il}|^2 \hat {R}_i(\beta)\gamma_1\mathbb{E}_{\leq k}[ R_{lj}^{(kil)}(\beta)]\gamma_1\otimes \gamma_1\mathbb{E}_{\leq k}[ R_{jl}^{(kil)}(\beta_1^*)\gamma_1 R_{ii}^{(k)}(\beta_1^*)].
\end{align}




Set  \[\alpha_{il}=\sum_{j<k}\gamma_1 \mathbb{E}_{\leq k-1}\left( R_{lj}^{(kil)}(\beta)\right)\gamma_1\otimes \gamma_1 \mathbb{E}_{\leq k-1}\left( R_{ji}^{(kil)}(\beta_1^*)\right).\]
Note that $ \alpha_{il}$ and $W_{il}$ are independent. Let us consider the $L^2$ norm of the term \eqref{1lemme55}:
\begin{align} 
\|\sigma_N^2 \sum_{i<k}\sum_{l<k}W_{il}\hat {R}_i(\beta)\alpha_{il} \|_{L^2} &\leq
\sigma_N^2  \sum_{i<k} \| \hat {R}_i(\beta)\| \| \sum_{l<k} W_{il}\alpha_{il} \|_{L^2}
\nonumber\\&\leq  \sigma_N^2 Q(|\Im z|^{-1},\| D_N\|,N^{1/2}\sigma_N)
 \sum_{i<k} \Big\{ \sum_{l<k,l'<k} \mathbb{E}\left( W_{il} \overline{W_{il'}} \Tr \alpha_{il} \alpha_{il'}^*\right)\Big\}^{1/2} \label{firstfirst}
\end{align}
by using \eqref{majorationunifhatRk}.
First,
\begin{align*} 
 \sum_{l<k} \mathbb{E}\left(| W_{il} |^2 \Tr \alpha_{il} \alpha_{il}^*\right)&=
\sum_{l<k} \mathbb{E}\left(| W_{il} |^2 \right) \mathbb{E}\left(\Tr \alpha_{il} \alpha_{il}^*\right)\\&\leq  m^2\sigma_N^2 \sum_{l<k, l\neq i}\mathbb{E}\left(\|\alpha_{il} \|^2\right) +m^2\tilde{\sigma}_N^2 \mathbb{E}\left(\|\alpha_{ii} \|^2\right).
\end{align*}
Replacing $R^{(kil)}$ by $R^{(k)}$ yields the following.
\begin{align*}
	\alpha_{il}&=\sum_{j<k}\gamma_1 \mathbb{E}_{\leq k-1}\big( R_{lj}^{(k)}(\beta)\big)\gamma_1\otimes \gamma_1 \mathbb{E}_{\leq k-1}\big( R_{ji}^{(k)}(\beta_1^*)\big)\\
	& \quad + \sum_{j<k}\gamma_1 \mathbb{E}_{\leq k-1}\big( R_{lj}^{(kil)}(\beta)-R_{lj}^{(k)}(\beta)\big)\gamma_1\otimes \gamma_1 \mathbb{E}_{\leq k-1}\big( R_{ji}^{(kil)}(\beta_1^*)\big)\\
	& \quad + \sum_{j<k}\gamma_1 \mathbb{E}_{\leq k-1}\big( R_{lj}^{(k)}(\beta)\big)\gamma_1\otimes \gamma_1 \mathbb{E}_{\leq k-1}\big( R_{ji}^{(kil)}(\beta_1^*)-R_{ji}^{(k)}(\beta_1^*)\big)\\
	&:= I_{il} +II_{il} +III_{il}. 
\end{align*}
Set $A(k)= (\gamma_1\otimes I_{N-1}) R^{(k)}(\beta) \gamma_1\otimes I_{N-1},\; B(k,i,l)=\gamma_1 \otimes I_{N-1}\mathbb{E}_{\leq k-1}\left( R^{(kil)}(\beta_1^*) -R^{(k)} (\beta_1^*)\right),$
and $C_{s,t}^{(il)} =\sum_{j<k} A(k)_{sj} \otimes B(k,i,l)_{jt}.$
By Lemma \ref{technicalbound}, we obtain $\| C_{s,t}^{(il)} \|\leq \| A(k)\| \| B(k,i,l)\|$
so that, choosing $s=l$ and $t=i$, we have 
 \begin{align*}\|III_{il}\| 
&\leq  \|\gamma_1\|^3   \mathbb{E}_{\leq k-1}\left( Q(\|W_N\|, |\Im z|^{-1})\right)
\mathbb{E}_{\leq k-1}\big(\|R^{(kil)}(\beta_1^*) -R^{(k)} (\beta_1^*)\|\big)
\\&\leq  \|\gamma_1\|^4\delta_N  \mathbb{E}_{\leq k-1}\big( Q(\|W_N\|, |\Im z|^{-1})\big)  \mathbb{E}_{\leq k-1}\big( Q(\|W_N\|, |\Im z_1|^{-1})\big),\end{align*}
where we use the resolvent identity and Lemma \ref{alta lemma}.
Similarly, 
$$\|II_{il}\|
\leq \|\gamma_1\|^4\delta_N  \mathbb{E}_{\leq k-1}\left( Q(\|W_N\|, |\Im z|^{-1})\right)  \mathbb{E}_{\leq k-1}\left( Q(\|W_N\|, |\Im z_1|^{-1})\right).$$
Thus, using Proposition \ref{boundnorm}, we obtain $$\mathbb{E}\left( \|II_{il}\|^2+ \|III_{il}\|^2\right)=O(\delta_N^2).$$
Now, set $B(k)=\gamma_1\otimes I_{N-1} \mathbb{E}_{\leq k-1}\left( R^{(k)}(\beta_2)\right)$
and $C_{s,t} =\sum_{j<k} A(k)_{sj} \otimes B(k)_{jt}.$
By Lemma \ref{technicalbound},
$$\Big( \sum_{l\neq k} \| C_{l,i} \|^2\Big)^{1/2} \leq \|A(k)\|  \|B(k)\|$$
that is 
$$\sum_{l\neq k}^N \| I_{il}\|^2 \leq \| \gamma_1\|^6    \mathbb{E}_{\leq k-1}\left( \|R^{(k)}(\beta)\|^2\right)  \mathbb{E}_{\leq k-1}\left( \|R^{(k)}(\beta_1^*)\|^2\right).$$
It readily follows by Remark \ref{remmajRNRkLp} that $\sum_{l<k} \mathbb{E} \| \alpha_{il} \|^2=
O(N\delta_N^2) +O(1)=O(N\delta_N^2)$ and then that 
$$ \sum_{l<k} \mathbb{E} \left( |W_{il}|^2 \Tr \alpha_{il} \alpha_{il}^*\right)=O(\delta_N^2),$$
uniformly in $k$.

It remains to control the sum of the cross terms $\esp[W_{il}\overline{W_{il'}}\Tr(\alpha_{il}{\alpha_{il'}}^*)]$, $l\neq l'$, $i<k, l<k, l'<k$.
Let us  replace $R^{(kil')}$ and $R^{(kil)}$ by $R^{(kil'il)}$:
\begin{align*}
	\alpha_{il}&=\sum_{j<k}\gamma_1 \mathbb{E}_{\leq k-1}\left( R_{lj}^{(kilil')}(\beta)\right)\gamma_1\otimes \gamma_1 \mathbb{E}_{\leq k-1}\left( R_{ji}^{(kilil')}(\beta_1^*)\right)\\
	& \quad + \sum_{j<k}\gamma_1 \mathbb{E}_{\leq k-1}\left( R_{lj}^{(kil)}(\beta)-R_{lj}^{(kilil')}(\beta)\right)\gamma_1\otimes \gamma_1 \mathbb{E}_{\leq k-1}\left( R_{ji}^{(kil)}(\beta_1^*)\right)\\
	& \quad + \sum_{j<k}\gamma_1 \mathbb{E}_{\leq k-1}\left( R_{lj}^{(kilil')}(\beta)\right)\gamma_1\otimes \gamma_1 \mathbb{E}_{\leq k-1}\left( R_{ji}^{(kil)}(\beta_1^*)-R_{ji}^{(kilil')}(\beta_1^*)\right)\\
	& := \alpha_{ill'}+\beta_{ill'}+\gamma_{ill'}. 
\end{align*}
Note that, due to independence properties, the sum vanishes when $\alpha_{il}$ is replaced by $\alpha_{ill'}$ or when $\alpha_{il'}$ is replaced by $\alpha_{ill'}$. It remains to control the four error terms: 
\[\esp[W_{il}\overline{W_{il'}}\Tr(\beta_{ill'}{\beta_{il'l}}^*)], \esp[W_{il}\overline{W_{il'}}\Tr(\beta_{ill'}{\gamma_{il'l}}^*)], \esp[W_{il}\overline{W_{il'}}\Tr(\gamma_{ill'}{\beta_{il'l}}^*)] \ \text{and} \ \esp[W_{il}\overline{W_{il'}}\Tr(\gamma_{ill'}{\gamma_{il'l}}^*)].\]
\eqref{remove} yields (note that we can remove one $\esp_{\leq k-1}$):
\begin{align}
\esp[W_{il}&\overline{W_{il'}}\Tr(\beta_{ill'}{\beta_{il'l}}^*) ]= (1-\tfrac{1}{2}\delta_{il})(1-\frac{1}{2}\delta_{il'})\nonumber\\
&\times \bigg\{\esp\Big[|W_{il}|^2|W_{il'}|^2\Tr\Big\{ \mathbb{E}_{\leq k-1}\Big(\gamma_1 R_{li}^{(kilil')}(\beta)\gamma_1\sum_{j<k}R_{l'j}^{(kil)}(\beta))\gamma_1\otimes \gamma_1 \mathbb{E}_{\leq k-1}\Big( R_{ji}^{(kil)}(\beta_1^*)\Big)\Big)\nonumber\\
& \hspace{3.5cm} \times  \Big(\gamma_1  R_{l'i}^{(kil'il)}(\beta)\gamma_1\sum_{j<k}R_{lj}^{(kil')}(\beta)\gamma_1\otimes \gamma_1 \mathbb{E}_{\leq k-1}\Big( R_{ji}^{(kil')}(\beta_1^*)\Big)\Big)^*\Big\}\Big]\label{betabeta0}
\end{align}
\begin{align}
& \hspace{2cm} + \esp\Big[W_{il}^2|W_{il'}|^2 \Tr\Big\{\gamma_1
\esp_{\leq k-1}\Big(
R_{li}^{(kilil')}(\beta)\gamma_1\sum_{j<k}R_{l'j}^{(kil)}(\beta)\gamma_1\otimes \gamma_1\mathbb{E}_{\leq k-1}[R_{ji}^{(kil)}(\beta_1^*)]\Big)\nonumber\\
& \hspace{3.5cm} \times \Big(\gamma_1R_{l'l}^{(kil'il)}(\beta)\gamma_1\sum_{j'<k}R_{ij'}^{(kil')}(\beta)\gamma_1\otimes \gamma_1\esp_{\leq k-1}[R_{j'i}^{(kil')}(\beta_1^*)]\Big)^*\Big\}\Big]\label{betabeta1}\\
& \hspace{2cm}  + \esp\Big[|W_{il}|^2\overline{W_{il'}}^2\Tr \Big\{\gamma_1R_{ll'}^{(kilil')}(\beta)\gamma_1\sum_{j<k}R_{ij}^{(kil)}(\beta)\gamma_1\otimes \gamma_1\esp_{\leq k-1}[R_{ji}^{(kil)}(\beta_1^*)]\nonumber\\
& \hspace{3.5cm} \times \Big(\gamma_1\mathbb{E}_{\leq k-1}\Big(R_{l'i}^{(kil'il)}(\beta)\gamma_1\sum_{j'<k}R_{lj'}^{(kil')}(\beta)\gamma_1\otimes \gamma_1\esp_{\leq k-1}[R_{j'i}^{(kil')}(\beta_1^*)]\Big)\Big]\Big)^* \Big\}\label{betabeta2}\\
& \hspace{2cm} + \esp\Big[W_{il}^2\overline{W_{il'}}^2\Tr \Big\{\gamma_1\mathbb{E}_{\leq k-1}\Big(R_{ll'}^{(kilil')}(\beta)\gamma_1\sum_{j<k}R_{ij}^{(kil)}(\beta)\gamma_1\otimes \gamma_1\esp_{\leq k-1}[R_{ji}^{(kil)}(\beta_1^*)]\Big)\nonumber\\
& \hspace{3.5cm} \times  \Big(\gamma_1R_{l'l}^{(kil'il)}(\beta)\gamma_1\sum_{j'<k}R_{ij'}^{(kil')}(\beta)\gamma_1\otimes \gamma_1\esp_{\leq k-1}[R_{j'i}^{(kil')}(\beta_1^*)]\Big)^*\Big\}\Big]\bigg\}\label{betabeta3}.
\end{align}
Using Lemmas \ref{technicalbound} and \ref{alta lemma}, there exist polynomials $Q_1$ and $Q_2$ such that 
$$\Big\| \sum_{j<k}R_{l'j}^{(kil)}(\beta)\gamma_1\otimes \gamma_1 \mathbb{E}_{\leq k-1}\Big( R_{ji}^{(kil)}(\beta_1^*)\Big) \Big\| \leq \|\gamma_1\|^2Q_1(\|W_N\|, |\Im z|^{-1})\mathbb{E}_{\leq k-1}\Big(Q_2(\|W_N\|,  |\Im z_1|^{-1})\Big) .$$
Thus, using H\"older's inequality and Proposition \ref{boundnorm}, the sum over $l\neq l'$ of \eqref{betabeta0} can be bounded as follows: 
\begin{align*}
\Big|\sum_{l\neq l'}(1-\tfrac{1}{2}\delta_{il})& (1-\frac{1}{2}\delta_{il'})\bigg\{\esp\Big[|W_{il}|^2|W_{il'}|^2\Tr \Big\{\gamma_1 \mathbb{E}_{\leq k-1}\Big( R_{li}^{(kilil')}(\beta)\gamma_1\sum_{j<k}R_{l'j}^{(kil)}(\beta)\gamma_1\otimes \gamma_1\\
 & \times \mathbb{E}_{\leq k-1}\Big( R_{ji}^{(kil)}(\beta_1^*)\Big)\Big)
(\gamma_1  R_{l'i}^{(kil'il)}(\beta)\gamma_1\sum_{j<k}R_{lj}^{(kil')}(\beta)\gamma_1\otimes \gamma_1 \mathbb{E}_{\leq k-1}\Big( R_{ji}^{(kil')}(\beta_1^*)\Big))^*\Big\}\Big]\Big|
\end{align*}
\begin{align*}
& \leq O(1)\sum_{l\neq l'}\esp[|W_{il}|^{4}|W_{il'}|^{4} \mathbb{E}_{\leq k-1}\Big(\|R_{li}^{(kilil')}(\beta)\|^2\Big) \mathbb{E}_{\leq k-1}\Big(\|R_{l'i}^{(kil'il)}(\beta)\|^2\Big)]^{1/2}\\
& \leq O(1)N^{-2}\sum_{l\neq l'}\esp[ \mathbb{E}_{\leq k-1}\Big(\|R_{li}^{(kilil')}(\beta)\|^2\Big) \mathbb{E}_{\leq k-1}\Big(\|R_{l'i}^{(kil'il)}(\beta)\|^2\Big)]^{1/2} \ \text{by independence}\\
& \leq O(1)\Big(N^{-2}\sum_{l\neq l'}\esp[\mathbb{E}_{\leq k-1}\Big(\|R_{li}^{(kilil')}(\beta)\|^2\Big) \mathbb{E}_{\leq k-1}\Big(\|R_{l'i}^{(kil'il)}(\beta)\|^2\Big)]\Big)^{1/2}\\& \ \text{by concavity of $x\mapsto x^{1/2}$}\\
& \leq O(1)\Big(N^{-2}\sum_{l,l'}\esp[\mathbb{E}_{\leq k-1}(\|R^{(k)}(\beta)_{li}\|^2)\mathbb{E}_{\leq k-1}(\|R^{(k)}(\beta)_{l'i}\|^2]+O(\delta_N)\Big)^{1/2}\\& \ \text{by \eqref{remove}, Lemmas \ref{alta lemma} and Proposition \ref{boundnorm}}\\
& \leq O(1)(N^{-2}\esp[(\sum_{l}\mathbb{E}_{\leq k-1}(\|R^{(k)}(\beta)_{li}\|^2))^{2}]+O(\delta_N))^{1/2}\\
& \leq O(1)(N^{-2}\esp[(\mathbb{E}_{\leq k-1}(\|R^{(k)}(\beta)\|^{2}))^2]+O(\delta_N))^{1/2}\ \text{by Lemma \ref{prelim}}\\
& \leq O(1)(O(N^{-2})+O(\delta_N))^{1/2}=O(\delta_N^{1/2}) \ \text{using Remark \ref{remmajRNRkLp}},
\end{align*}
uniformly in $k$.
The sums  over $l$ and $l'$ of the  three other terms \eqref{betabeta1}, \eqref{betabeta2} and \eqref{betabeta3} can be treated similarly: they are of order $O(\delta_N)$, uniformly in $k$.

Using again the resolvent identity and very similar computations, the sums over $l$ and $l'$ of the other three error terms $\esp[W_{il}\overline{W_{il'}}\Tr\beta_{ill'}{\gamma_{il'l}}^*]$, $\esp[W_{il}\overline{W_{il'}}\Tr\gamma_{ill'}{\beta_{il'l}}^*]$ and $\esp[W_{il}\overline{W_{il'}}\Tr\gamma_{ill'}{\gamma_{il'l}}^*]$ 
are proved to be of order $O(\sqrt{\delta_N})$, uniformly in $k$.
As a consequence, from \eqref{firstfirst} we can deduce that the first term { \eqref{1lemme55}}
\[\sigma_N^2 \sum_{i,j,l<k}W_{il}\hat {R}_i(\beta)\gamma_1 \mathbb{E}_{\leq k-1}\Big( R^{(kil)}(\beta)_{lj}\Big)\gamma_1\otimes \gamma_1 \mathbb{E}_{\leq k-1}\Big( R^{(kil)}(\beta_1^*)_{ji}\Big)
=o_{L^2}(1)\]
uniformly in $k$.

The $L^2$ norm of the second term { \eqref{2lemme55}} is bounded by
\begin{align*}
\sigma_N^2 & \|\hat {R}(\beta)\|\|\gamma_1\|\\
& \times \sum_{i<k}\mathbb{E}\Big[\Big|\sum_{l<k}(1-\tfrac{1}{2}\delta_{il})|W_{il}|^2\mathbb{E}_{\leq k}[\|\sum_{j<k}\mathbb{E}_{\leq k}[ R^{(kil)}(\beta)_{lj}]\gamma_1\otimes\gamma_1 R_{ji}^{(kil)}(\beta_1^*)\|\|\gamma_1\|\| R_{li}^{(k)}(\beta_1^*)\|]\Big|^2\Big]^{1/2},
\end{align*}
and then, using Jensen's inequality (with respect to $\mathbb{E}_{\leq k-1}$) and Cauchy-Schwarz inequality (with respect to the $l$-sum), by
\[\sigma_N^2\|\hat {R}(\beta)\|\|\gamma_1\|^2 \sum_{i<k}\mathbb{E}[\sum_{l<k}|W_{il}|^4\|\sum_{j<k}\mathbb{E}_{\leq k}[ R_{lj}^{(kil)}(\beta)]\gamma_1\otimes\gamma_1 R_{ji}^{(kil)}(\beta_1^*)\|^2\sum_{l<k}\| R_{li}^{(k)}(\beta_1^*)\|^2]^{1/2}.\]
{ From Lemma \ref{technicalbound} and Lemma \ref{prelim}, and then Lemma \ref
{alta lemma} one can deduce the following bounds of the $L^2$ norm of { \eqref{2lemme55}}:
\begin{align*}
 \sigma_N^2 & \|\hat {R}(\beta)\|\|\gamma_1\|^4 \sum_{i<k}\mathbb{E}[\sum_{l<k}|W_{il}|^4\|\mathbb{E}_{\leq k}[R^{(kil)}(\beta)]\|^2 \|R^{(kil)}(\beta_1^*)\|^2\| R^{(k)}(\beta_1^*)\|^2]^{1/2}\\
&\leq  \sigma_N^2\|\hat {R}(\beta)\|\|\gamma_1\|^4 \sum_{i<k}\mathbb{E}[\sum_{l<k}|W_{il}|^4\|\mathbb{E}_{\leq k}[R^{(kil)}(\beta)]\|^2 \|R^{(kil)}(\beta_1^*)\|^2 Q(\|W_N^{(k)}\|, |\Im z_1|^{-1})^2]^{1/2}\\
&\leq  \sigma_N^2\|\hat {R}(\beta)\|\|\gamma_1\|^4\\
& \hspace{1.5cm} \times \sum_{i<k}\Big\{\sum_{l<k}\mathbb{E}(|W_{il}|^4)\mathbb{E}\Big(\|\mathbb{E}_{\leq k}[R^{(kil)}(\beta)]\|^2 \|R^{(kil)}(\beta_1^*)\|^2 Q_1(\|W_N^{(kil)}\|+ 2 \delta_N, |\Im z_1|^{-1})^2\Big)\Big\}^{1/2}\\
&\leq  \sigma_N^2\|\hat {R}(\beta_1)\|\|\gamma_1\|^4 \sum_{i<k}\Big\{\sum_{l<k}\mathbb{E}(|W_{il}|^4)\mathbb{E}\Big(|\mathbb{E}_{\leq k}[Q_2(\|W_N\|, \delta_N, |\Im z|^{-1})]|^2 Q_3(\|W_N\|, \delta_N, |\Im z_1|^{-1})\Big)\Big\}^{1/2}\\
&=O(N^{-1/2}),
\end{align*}
where we use Lemma \ref{Qbound},
Lemmas \ref{puissance} and  \ref{boundnorm}  in the last line.}
Let us consider the last term \eqref{3lemme55}. Replacing successively $(1-\tfrac{1}{2}\delta_{il})|W_{il}|^2$, $R_{ii}^{(k)}(\beta_1^*)$, $R_{lj}^{(kil)}(\beta)$, $R_{jl}^{(kil)}(\beta_1^*)$ by $\sigma_N^2, \hat{R}_{i}(\beta_1^*), R_{lj}^{(k)}(\beta), R_{jl}^{(k)}(\beta_1^*)$, this last term can be written as follows:
\begin{align*}
 \sigma_N^2  \sum_{i,j,l<k}&(1-\tfrac{1}{2}\delta_{il})|W_{il}|^2 \hat {R}_i(\beta)\gamma_1\mathbb{E}_{\leq k}[ R_{lj}^{(kil)}(\beta)]\gamma_1\otimes \gamma_1\mathbb{E}_{\leq k}[ R_{jl}^{(kil)}(\beta_1^*)\gamma_1 R_{ii}^{(k)}(\beta_1^*)]\\
 & =\sigma_N^4 \sum_{i,j,l<k}\hat{R}_i(\beta)\gamma_1\mathbb{E}_{\leq k}[R_{lj}^{(k)}(\beta)]\gamma_1\otimes \gamma_1\mathbb{E}_{\leq k}[R_{jl}^{(k)}(\beta_1^*)]\gamma_1 \hat{R}_i(\beta_1^*)\\
 & \quad +\sigma_N^2 \sum_{i,j,l<k}((1-\tfrac{1}{2}\delta_{il})|W_{il}|^2-\sigma_N^2)\hat{R}_i(\beta)\gamma_1\mathbb{E}_{\leq k}[R_{lj}^{(kil)}(\beta)]\gamma_1\otimes \gamma_1\mathbb{E}_{\leq k}[R_{jl}^{(kil)}(\beta_1^*)\gamma_1 R_{ii}^{(k)}(\beta_1^*)]\\
 & \quad +\sigma_N^4 \sum_{i,j,l<k}\hat{R}_i(\beta)\gamma_1\mathbb{E}_{\leq k}[R_{lj}^{(kil)}(\beta)]\gamma_1\otimes \gamma_1\mathbb{E}_{\leq k}[R_{jl}^{(kil)}(\beta_1^*)\gamma_1 (R_{ii}^{(k)}(\beta_1^*)-\hat{R}_i(\beta_1^*))]\\
 & \quad +\sigma_N^4 \sum_{i,j,l<k}\hat{R}_i(\beta)\gamma_1\mathbb{E}_{\leq k}[R_{lj}^{(kil)}(\beta)-R_{lj}^{(k)}(\beta)]\gamma_1\otimes \gamma_1\mathbb{E}_{\leq k}[R_{jl}^{(kil)}(\beta_1^*)]\gamma_1 \hat{R}_i(\beta_1^*)\\
 & \quad +\sigma_N^4 \sum_{i,j,l<k}\hat{R}_i(\beta)\gamma_1\mathbb{E}_{\leq k}[R_{lj}^{(k)}(\beta)]\gamma_1\otimes \gamma_1\mathbb{E}_{\leq k}[R_{jl}^{(kil)}(\beta_1^*)-R_{jl}^{(k)}(\beta_1^*)]\gamma_1 \hat{R}_i(\beta_1^*)\\
 & :=N\sigma_N^2T_{N,k,0}(z,\bar z_1)(\nabla_k)+h_k^{(N)}+p_k^{(N)}+q_k^{(N)}+r_k^{(N)}.
\end{align*}
By triangular inequality, Jensen's inequality and Cauchy-Schwarz inequality, the $L^1$ norm of term $h_k^{(N)}$ can be bounded as follows:
\begin{align*}
\|h_k^{(N)}\|_1 & \leq \sigma_N^2\|\hat{R}(\beta_1)\|\|\gamma_1\|\\
& \hspace{1cm} \times \sum_{i<k}\|\|\sum_{l<k}((1-\tfrac{1}{2}\delta_{il})|W_{il}|^2-\sigma_N^2)\sum_{j<k}\mathbb{E}_{\leq k}[R^{(kil)}(\beta)_{lj}]\gamma_1\otimes \gamma_1R_{jl}^{(kil)}(\beta_1^*)\gamma_1\|\|R_{ii}^{(k)}(\beta_1^*)\|\|_1
		\end{align*}
		\begin{align}
	\hspace{1cm} & \leq \sigma_N^2\|\hat{R}(\beta)\|\|\gamma_1\|\nonumber\\
	& \hspace{0.5cm} \times \sum_{i<k}\|\sum_{l<k}((1-\tfrac{1}{2}\delta_{il})|W_{il}|^2-\sigma_N^2)\sum_{j<k}\mathbb{E}_{\leq k}[R^{(kil)}(\beta)_{lj}]\gamma_1\otimes \gamma_1R_{jl}^{(kil)}(\beta_1^*)\gamma_1\|_2\|R_{ii}^{(k)}(\beta_1^*)\|_2\nonumber\\
	& \leq \sigma_N^2\|\hat{R}(\beta)\|\|\gamma_1\|\nonumber\\
	& \hspace{0.5cm} \times \|R^{(k)}(\beta_1^*)\|_2\sum_{i<k}\mathbb{E}[\|\sum_{l<k}((1-\tfrac{1}{2}\delta_{il})|W_{il}|^2-\sigma_N^2)\sum_{j<k}\mathbb{E}_{\leq k}[R^{(kil)}(\beta)_{lj}]\gamma_1\otimes \gamma_1R_{jl}^{(kil)}(\beta_1^*)\gamma_1\|^2]^{1/2}\nonumber\\
	&=O\big(\delta_N^{1/2}\big),\label{h}
\end{align}
if one may prove that 
\begin{equation}\label{forh}\mathbb{E}[\|\sum_{l<k}((1-\tfrac{1}{2}\delta_{il})|W_{il}|^2-\sigma_N^2)\sum_{j<k}\mathbb{E}_{\leq k}[R^{(kil)}(\beta)_{lj}]\gamma_1\otimes \gamma_1R_{jl}^{(kil)}(\beta_1^*)\gamma_1\|_{HS}^2]=O(\delta_N),\end{equation}
uniformly in $i,k$, using also Remark \ref{remmajRNRkLp}
and Lemma \ref{Qbound}.

Develop the Hilbert-Schmidt norm of the $l$-sum: the sum of ``squares" is bounded by
\begin{align*}
  m(m_N-\sigma_N^4)\sum_{l<k, l\neq i} &\mathbb{E}  \Big[\|\sum_{j<k}\mathbb{E}_{\leq k}[R^{(kil)}(\beta)_{lj}]\gamma_1\otimes \gamma_1R_{jl}^{(kil)}(\beta_1^*)\gamma_1\|^2\Big]\\
 & +m(\frac{1}{4}\delta_N^2\tilde{\sigma}_N^2-\sigma_N^2\tilde{\sigma}_N^2+\sigma_N^4)\mathbb{E}\Big[\|\sum_{j<k}\mathbb{E}_{\leq k}[R^{(kii)}(\beta)_{ij}]\gamma_1\otimes \gamma_1R_{ji}^{(kii)}(\beta_1^*)\gamma_1\|^2\Big]
\end{align*}
which is $O(N^{-1})$, uniformly in $i,k$, using Lemma \ref{technicalbound} and Remark \ref{remmajRNRkLp}. 

It remains to control the sum of the cross terms 
\[\esp[((1-\tfrac{1}{2}\delta_{il})|W_{il}|^2-\sigma_N^2)((1-\frac{1}{2}\delta_{il'})|W_{il'}|^2-\sigma_N^2)\langle C_{ll}^{il},C_{l'l'}^{il'}\rangle],\]
where $C_{ll}^{il}=\sum_{j<k}\mathbb{E}_{\leq k}[R_{lj}^{(kil)}(\beta)]\gamma_1\otimes \gamma_1R_{jl}^{(kil)}(\beta_1^*)\gamma_1$
and $\langle A,B\rangle=\Tr AB^*$.
Define $$C^{ilil'}_{ll}= \sum_{j<k}\mathbb{E}_{\leq k}[R_{lj}^{(kilil')}(\beta)]\gamma_1\otimes \gamma_1R_{jl}^{(kilil')}(\beta_1^*)\gamma_1.$$
Note that by Lemma \ref{technicalbound} 
\begin{align*}
 \|C_{ll}^{il}\| & \leq \|\gamma_1\|^2\mathbb{E}_{\leq k}[\|R^{(kil)}(\beta) \|]\|  R^{(kil)}(\beta_1^*)\|,\\
 \|C_{ll}^{ilil'}\| & \leq \|\gamma_1\|^2\mathbb{E}_{\leq k}[\|R^{(kilil')}(\beta) \|]\|  R^{(kilil')}(\beta_1^*)\|.
\end{align*}
Observe that, by independence, since at least one of $l,l'\neq i$,
\begin{align*}
 |\esp & [((1-\tfrac{1}{2}\delta_{il})|W_{il}|^2-\sigma_N^2)((1-\frac{1}{2}\delta_{il'})|W_{il'}|^2-\sigma_N^2)\langle C_{ll}^{il},C_{l'l'}^{il'}\rangle]|\\
 & =|\esp[((1-\tfrac{1}{2}\delta_{il})|W_{il}|^2-\sigma_N^2)((1-\frac{1}{2}\delta_{il'})|W_{il'}|^2-\sigma_N^2)(\langle C_{ll}^{il},C_{l'l'}^{il'}\rangle-\langle C_{ll}^{ilil'},C_{l'l'}^{ilil'}\rangle)]|\\
 & \leq\esp[|(1-\tfrac{1}{2}\delta_{il})|W_{il}|^2-\sigma_N^2|^2]^{1/2}\esp[|(1-\frac{1}{2}\delta_{il'})|W_{il'}|^2-\sigma_N^2|^2]^{1/2}\esp[|\langle C_{ll}^{il},C_{l'l'}^{il'}\rangle-\langle C_{ll}^{ilil'},C_{l'l'}^{ilil'}\rangle|^2]^{1/2}.
\end{align*}
Now, from \eqref{remove}, 
\begin{equation}\label{estimprelim}
\big\| R^{(kil)}(\beta) - R^{(kilil')}(\beta)\big\| \leq 2 \delta_N \|\gamma_1\| \big\| R^{(kil)}(\beta)\big\| \big\|  R^{(kilil')}(\beta)\big\|. \end{equation}
Thus, using Cauchy-Schwarz inequality and Lemma \ref{prelim}, one can easily obtain that 
\begin{align*}
\big\| C_{ll}^{il} - C_{ll}^{ilil'}\big\| &\leq O(\delta_N) \Big(\mathbb{E}_{\leq k}[\|R^{(kil)}(\beta) \|^2 \big\|  R^{(kilil')}(\beta)\big\|^2]\Big)^{1/2} \|R^{(kil)}(\beta_1^*) \|\\
&+O(\delta_N) \Big(\mathbb{E}_{\leq k}\big[\|R^{(kilil')}(\beta) \|^2 \big]\Big)^{1/2} \|R^{(kil)}(\beta_1^*) \|\big\|  R^{(kilil')}(\beta_1^*)\big\|.
\end{align*}
Then Remark \ref{remmajRNRkLp} readily implies that $$ \esp[|\langle C_{ll}^{il},C_{l'l'}^{il'}\rangle-\langle C_{ll}^{ilil'},C_{l'l'}^{ilil'}\rangle|^2]^{1/2}
=O(\delta_N)$$
uniformly in $i,k,l,l'$ and then that $\esp[((1-\tfrac{1}{2}\delta_{il})|W_{il}|^2-\sigma_N^2)((1-\frac{1}{2}\delta_{il'})|W_{il'}|^2-\sigma_N^2)\langle C_{ll}^{il},C_{l'l'}^{il'}\rangle]$
$=O(\delta_NN^{-2})$ uniformly in $i,k,l,l'$. There are less than $N^2$ such cross terms. Therefore, \eqref{forh} and then \eqref{h}  are true.

The $L^1$ norm of $p_k^{(N)}$ can be bounded as follows (using \eqref{majorationunifhatRk}, Lemma \ref{technicalbound},
Remarks \ref{elemdiag} and \ref{remmajRNRkLp}:
\begin{align*}
	\|p_k^{(N)}\|_1 
	& \leq \sigma_N^4\sum_{i,l<k}\|\hat{R}_i(\beta)\|\|\gamma_1\|\mathbb{E}[\|\sum_{j<k}\mathbb{E}_{\leq k}[R_{lj}^{(kil)}(\beta)]\gamma_1\otimes \gamma_1R_{jl}^{(kil)}(\beta_1^*)\gamma_1\|\|R_{ii}^{(k)}(\beta_1^*)-\hat{R}_i(\beta_1^*)\|]\\
	& =O(N\sigma_N^4\sum_{i<k}\mathbb{E}[\|R_{ii}^{(k)}(\beta_1^*)-\hat{R}_i(\beta_1^*)\|^p]^{1/p})=o(1)
\end{align*}
uniformly in $k$.

Using Lemma \ref{technicalbound} and \eqref{remove}, 
\begin{align*}
\Big\|\sum_{j<k}\mathbb{E}_{\leq k}[R_{lj}^{(kil)}(\beta)\!-\!R^{(k)}(\beta)_{lj}]\gamma_1\otimes \gamma_1R_{jl}^{(kil)}(\beta_1^*)\Big\|
&\leq \|\gamma_1\|^2 \|\esp_{\leq k-1}[R^{(kil)}(\beta)-R^{(k)}(\beta)]\| \| R^{(kil)}(\beta_1^*)\|\\
&\leq 2\delta_N\|\gamma_1\|^3 \esp_{\leq k-1}[\|R^{(kil)}(\beta)\|\|R^{(k)}(\beta)\|]\| R^{(kil)}(\beta_1^*)\|.
\end{align*}
Hence, using Remark \ref{remmajRNRkLp}, the $L^1$ norm of $q_k^{(N)}$ is $O(\delta_N)$,
uniformly in $k$.
Similarly, the $L^1$ norm of $r_k^{(N)}$ is $O(\delta_N)$, uniformly in $k$.
Thus, we have established that \eqref{3lemme55} is equal to 
$N\sigma_N^2T_{N,k,0}(z,\bar z_1)(\nabla_k)+o_{L^1}(1).$
Since moreover we also established  that  \eqref{1lemme55} and \eqref{2lemme55} are 
$o_{L^1}(1)$ uniformly in $k$, Lemma \ref{nablak} yields that 
$$\nabla_k=\sigma_N^2\sum_{i<k}\hat{R}_i(\beta) \gamma_1 \otimes \gamma_1 \mathbb{E}_{\leq k-1}\left( R_{ii}^{(k)}(\beta_1^*)\right)
+N\sigma_N^2T_{N,k,0}(z,\bar z_1)(\nabla_k)+o_{L^1}(1).
$$
Lemma \ref{equationapprochee} readily follows by using Remark \ref{elemdiag}.

\end{proof}

Thus, for any  $z_1 \in \C \setminus \R$ and $z \in O_{z_1}$ (see Lemma \ref{rz}), setting  $\beta=ze_{11}-\gamma_0$ and $\beta_1^*=\bar z_1e_{11}-\gamma_0$ in $M_m(\C)$, we obtain from Lemma \ref{equationapprochee} that 
\begin{align*} 
\nabla_k &= \left(\mathrm{id}_m\otimes \mathrm{id}_m-N\sigma_N^2T_{N,k,0}(z,\bar z_1)\right)^{-1} \left(N\sigma_N^2T_{N,k,0}(z,\bar z_1)(I_m\otimes I_m)\right)+o_{L^{1}}^{(u)}(1)\\&= - I_m\otimes I_m + \left(\mathrm{id}_m\otimes \mathrm{id}_m-N\sigma_N^2T_{N,k,0}(z,\bar z_1)\right)^{-1} \left(I_m\otimes I_m\right)+o_{L^{1}}^{(u)}(1).
\end{align*}
Therefore,
\begin{align*}
 N\sigma_N^2T_N(z,\bar z_1)(\nabla_k) & =  -N\sigma_N^2T_N(z,\bar z_1)(I_m\otimes I_m)\\
 & \quad + N\sigma_N^2T_N(z,\bar z_1)\left(\mathrm{id}_m\otimes \mathrm{id}_m-N\sigma_N^2T_{N,k,0}(z,\bar z_1)\right)^{-1} \left(I_m\otimes I_m\right)+o_{L^1}(1).
\end{align*}

Thus from \eqref{trick}, the term under study in Proposition \ref{analog} can be rewritten as follows 
\begin{align*}
\sigma_N^4 \sum_{k=1}^N\sum_{i,j<k} & \Tr(\gamma_1\mathbb{E}_{\leq k}[R_{ij}^{(k)}(\beta)]\gamma_1\hat{R}_k(\beta))\Tr(\gamma_1\mathbb{E}_{\leq k}[R_{ji}^{(k)}(\beta_1^*)]\gamma_1\hat{R}_k(\beta_1^*)) \\
 & =-\sigma_N^2 \sum_{k=1}^N \Tr\left(\hat{R}_k(\beta) \gamma_1\right)\Tr \left(\gamma_1 \hat{R}_k(\beta_1^*)\right)
+N^{-1}\sum_{k=1}^N f_{k,k,N}(0)+o_{\mathbb{P}}(1),
\end{align*}
where for $t\in [0;1]$, with the notations of Section \ref{sec_2nd_term}, $$f_{k,k,N}(t)=N\sigma_N^2\Tr \otimes \Tr \left[\hat{R}_k(\beta)\gamma_1\otimes I_m\left(\mathrm{id}_m\otimes \mathrm{id}_m-N\sigma_N^2T_{N,k,t}(z,\bar z_1)\right)^{-1} \left(I_m\otimes I_m\right)I_m\otimes \gamma_1 \hat{R}_k(\beta_1^*)\right].$$
The first term can be analysed as in Section \ref{sec_rst_term}. The second term can be analysed as follows: for any $t\in [0;1]$,
\begin{align*}
f_{k,k,N}(t)& -f_{k,k,N}(0)\\
& = tN\sigma_N^4\Tr\otimes\Tr \Big[ \hat{R}_k(\beta)\gamma_1\otimes I_m\big(\mathrm{id}_m\otimes \mathrm{id}_m-N\sigma_N^2T_{N,k,t}(z,\bar z_1)\big)^{-1}\hat {R}_k(\beta) \gamma_1\otimes I_m\\
 & \hspace{1cm} \times \big(\mathrm{id}_m\otimes \mathrm{id}_m-N\sigma_N^2T_{N,k,0}(z,\bar z_1)\big)^{-1}\!\big(I_m\otimes I_m\big)I_m\otimes \gamma_1\hat {R}_k(\beta_1^*),I_m\otimes \gamma_1 \hat{R}_k(\beta_1^*)\Big]. 
\end{align*}
It readily follows
from Lemma \ref{Qbound} and Lemma \ref{rz} that there exists some constant $C(z_1)>0$ such that for any $t\in [0,1]$,
$$\left| f_{k,k,N}(t)-f_{k,k,N}(0) \right| \leq C(z_1)N\sigma_N^4.$$
Integrating with respect to $t\in [0,1]$ and summing on $k$ we obtain that 
$$\Big|N^{-1}\sum_{k=1}^N  f_{k,k,N}(0)-N^{-1}\sum_{k=1}^N \int_0^1 f_{k,k,N}(t)dt \Big| \leq C(z_1)N \sigma_N^4=o(1).$$
Now, set for $k=1,\ldots,N$,
$$A_k:X\mapsto \sigma_N^2 \hat {R}_k(\beta) \gamma_1\otimes I_mXI_m\otimes \gamma_1 \hat {R}_k(\beta_1^*),$$
$$B_k=N\sigma_N^2T_{N,k,0}(z,\bar z_1),$$
and for any $t\in [0,1]$,
$$g_k(t)= \log (\mathrm{id}_m\otimes \mathrm{id}_m-B_k-tA_k)=-\sum_{p=1}^{+\infty}\frac{1}{p}(B_k+tA_k)^p.$$
Note that $f_{k,k,N}(t)=N(\Tr\otimes\Tr)\left[A_k(\mathrm{id}_m\otimes \mathrm{id}_m-B_k-tA_k)^{-1}(I_m\otimes I_m)\right]$.

\begin{lem}\label{idlog}
$$(\Tr\otimes\Tr)(g_k'(t)(I_m\otimes I_m)) =(\Tr\otimes\Tr)\left[-A_k(\mathrm{id}_m\otimes \mathrm{id}_m-B_k-tA_k)^{-1}(I_m\otimes I_m)\right].$$
\end{lem}

\begin{proof} 
First observe that $T\mapsto \Tr \otimes \Tr(T(I_m\otimes I_m))$ is a trace on the algebra generated by $A_1,\ldots ,A_N$ (to which belong $B_1,\ldots ,B_N$).	
Note that $$g_k'(t)= -\sum_{p=1}^\infty \frac{1}{p}\sum_{i=0}^{p-1}(B_k+tA_k)^{i}A_k(B_k+tA_k)^{p-1-i}$$
Hence 
\begin{align*}
\Tr \otimes \Tr[g_k'(t) (I_m\otimes I_m)]&=-\sum_{p=1}^\infty \frac{1}{p}\sum_{i=0}^{p-1}\Tr \otimes \Tr[(B_k+tA_k)^{i}A_k(B_k+tA_k)^{p-1-i}(I_m\otimes I_m)]\\
&=-\sum_{p=1}^\infty \frac{1}{p}\sum_{i=0}^{p-1}\Tr \otimes \Tr[A_k(B_k+tA_k)^{p-1}(I_m\otimes I_m)]\\
&=-\Tr \otimes \Tr\Big[\sum_{p=0}^\infty  A_k(B_k+tA_k)^{p} (I_m\otimes I_m)\Big]\\
&=\Tr \otimes \Tr \left[-A_k(\mathrm{id}_m\otimes \mathrm{id}_m-B_k-tA_k)^{-1}(I_m\otimes I_m)\right].\end{align*}
\end{proof}

Lemma \ref{idlog} readily implies that 
\begin{align*}
\int_0^1(\Tr & \otimes\Tr)\left[-A_k(\mathrm{id}_m\otimes \mathrm{id}_m-B_k-tA_k)^{-1}(I_m\otimes I_m)\right]\,{\rm d}t\\
&=(\Tr\otimes\Tr) \left[  \log (\mathrm{id}_m\otimes \mathrm{id}_m-B_k-A_k)(I_m\otimes I_m)\right]-(\Tr\otimes \Tr)\left[  \log (\mathrm{id}_m\otimes \mathrm{id}_m-B_k)(I_m\otimes I_m)\right].
\end{align*}
Therefore, noticing that $B_{k+1}=A_k + B_k$, we obtain that
\begin{align*} 
N^{-1}\sum_{k=1}^N \int_0^1\!f_{k,k,N}(t)dt& =(\Tr\otimes\Tr)\left[\log\left(\mathrm{id}_m\otimes \mathrm{id}_m-N\sigma_N^2T_N(z,\bar z_1)\right) (I_m\otimes I_m)\right].
\end{align*}

\begin{lem}\label{29}
%
For any $\beta_1=z_1e_{11}-\gamma_0, z_1\in\mathbb{C}\setminus\mathbb{R}$ and any $\beta_2=z_2e_{11}-\gamma_0,z_2\in\mathbb{C}\setminus\mathbb{R}$, we have
\begin{align*}
\lim_{N\to\infty}
(\Tr\otimes\Tr) \big[\log\big(\mathrm{id}_m \otimes\mathrm{id}_m- & N\sigma_N^2T_N(z_1,z_2)\big) (I_m\otimes I_m)\big]\\
&=(\Tr\otimes\Tr)\left[\log\left(\mathrm{id}_m\otimes \mathrm{id}_m-\sigma^2T_{\{\beta_1,\beta_2\}}\right)(I_m\otimes I_m)\right].
\end{align*}
\end{lem}

\begin{proof}
For any $\alpha\otimes\upsilon\in M_m(\mathbb{C})\otimes M_m(\mathbb{C})$, we have
\begin{align}
\big[N & \sigma_N^2  {T_{N}(z_1,z_2)}\big]\big(\alpha\otimes\upsilon\big)\ \ \nonumber\\
& =  N\sigma_N^2(\mathrm{id}_{m^2}\otimes\tau_N)\!\left[
(\omega_N(\beta_1)\otimes1-\gamma_2\otimes D_N)^{-1}(\gamma_1\alpha\otimes1)\right]\!\otimes\!\left[(\upsilon\gamma_1\otimes1)(\omega_N(\beta_2)\otimes1-\gamma_2\otimes D_N)^{-1}\right]\!,\label{en}
\end{align}
where by $(\mathrm{id}_{m^2}\otimes\tau_N)$ we mean that $\tau_N$ is applied entrywise to the matrix belonging to $M_{m^2}(C^*\langle D_N\rangle)\simeq
M_m(C^*\langle D_N\rangle)\otimes M_m(C^*\langle D_N\rangle)\simeq M_m(\mathbb C)\otimes C^*\langle D_N\rangle\otimes M_m(\mathbb C)\otimes C^*\langle D_N\rangle$. 
Since $(D_N)_{N\in \mathbb{N}}$ converges in $*$-moments towards $d$, $\lim_{N\to +\infty}N\sigma_N^2=\sigma^2$ and using Lemma \ref{cvunif}, we can easily deduce from \eqref{en} (using also Lemmas \ref{Qbound} and \ref{bound}) that  for any $z_1,z_2$ in $ \C \setminus \R$, the sequence of operators  $(N\sigma_N^2{T_{N}(z_1,z_2)})_N$ converges in operator norm to $\sigma^2T_{\{\beta_1,\beta_2\}}$. 
We know by Corollary \ref{sprNsanstilde} that there exists   $0<\epsilon_0<1$ such $\limsup_{N\rightarrow +\infty} \rho(N\sigma_N^2{T_{N}(z_1,z_2)}<1-\epsilon_0.$
 Thanks to the Cauchy formula, for all $x \in \mathbb{C}$ such that $\vert x \vert < 1-\epsilon_0/2$, for any $k \geq 0$, $x^k =\frac{1}{2i\pi} \int_{\vert w\vert = 1-\epsilon_0/2} \frac{w^k}{w-x} dw.$
Therefore,  using the holomorphic functional calculus, we have  for all large $N$,  
$$\forall k \geq 0~~, (N\sigma_N^2{T_{N}(z_1,z_2)})^k =\frac{1}{2i\pi} \int_{\vert w\vert = 1-\epsilon_0/2} {w^k}{(w{\rm id}_m\otimes{\rm id}_m-N\sigma_N^2{T_{N}(z_1,z_2)})^{-1}} dw,$$
and therefore $$\forall k \geq 0~~, \Vert  (N\sigma_N^2{T_{N}(z_1,z_2)})^k \Vert \leq  \sup_{\vert w\vert = 1-\epsilon_0/2 } \Vert {(w{\rm id}_m\otimes{\rm id}_m-N\sigma_N^2{T_{N}(z_1,z_2)})^{-1}}\Vert {( 1-\epsilon_0/2 ) }^{k+1}.$$
Now, using Lemmas \ref{Qbound} and \ref{majnormerayon}, there exists $C(m,\epsilon_0)>0$ such that  we have  for all large $N$,  
$$ \sup_{\vert w\vert = 1-\epsilon_0/2 } \Vert (w{\rm id}_m\otimes{\rm id}_m-N\sigma_N^2 T_{N}(z_1,z_2))^{-1}\Vert \leq C(m,\epsilon_0),$$
and thus $$\forall k \geq 0~~, \Vert  (N\sigma_N^2 T_{N}(z_1,z_2))^k \Vert \leq   C(m,\epsilon_0)\Vert ( 1-\epsilon_0/2 )^{k+1}.$$
Therefore, using dominated convergence Theorem,  it readily follows that     
\begin{align*}
 \log\big[{\rm id}_m\otimes{\rm id}_m-N\sigma_N^2 & T_{N}(z_1,z_2)\big]  =-\sum_{k=1}^\infty \frac{1}{k}{(N\sigma_N^2 T_{N}(z_1,z_2))^k}\\
 & \rightarrow_{N \rightarrow +\infty}-\sum_{k=1}^\infty \frac{1}{k}{(\sigma^2 T_{\{\beta_1,\beta_2\}})^k}
=\log\left[{\rm id}_m\otimes{\rm id}_m-\sigma^2 T_{\{\beta_1,\beta_2\}}\right].
\end{align*}

\end{proof}

Thus Proposition \ref{analog} is proved.
\end{proof}

\begin{prop}\label{analog2}   For $\beta_1=z_1e_{11}-\gamma_0$ and 
$\beta_2=z_2e_{11}-\gamma_0$, for  any  $z_1\in \mathbb{C}\setminus \R, z_2 \in \C \setminus \R,$
the following convergence holds in probability:
\begin{align}
 \sigma_N^4 \sum_{k=1}^N\sum_{i,j<k} & \Tr(\gamma_1\mathbb{E}_{\leq k}[ R_{ij}^{(k)}(\beta_1)]\gamma_1\hat{R}_k(\beta_1))\Tr(\gamma_1\mathbb{E}_{\leq k}[ R_{ji}^{(k)}
(\beta_2)]\gamma_1\hat{R}_k(\beta_2))\nonumber\\
& \rightarrow_{N\rightarrow +\infty}-{\rm Tr}\otimes\Tr\left\{\log\left[\mathrm{id}_m\otimes\mathrm{id}_m-\sigma^2T_{\{\beta_1,\beta_2\}}\right](I_m\otimes I_m)\right\}\\
& \hspace{7cm} -\sigma^2{\rm Tr}\otimes\Tr\left\{T_{\{\beta_1,\beta_2\}}(I_m\otimes I_m)\right\}.\nonumber
\end{align}
\end{prop}

\begin{proof}
Recall that a sequence $(X_N)_{N\geq 1}$ of random variables converges in probability to a random variable $X$ if and only if, from any subsequence extracted from $(X_N)_{N\geq 1}$, one can further extract a subsubsequence almost surely converging to $X$. We will use this criterion in the following argument.

Let $C$ be as in Proposition \ref{boundnorm}. If $\Vert W_N\Vert \leq C$,   there exists $M>0$ such that for any $k \in \{1,\ldots,N\}$, $\|P(W_N^{(k)}, D_N^{(k)})\|\leq M.$
Let $K>0$ be such that $\|P(s_N,D_N)\| \leq K$. Thus, for any $z \in \mathbb{R}$ such that $\vert z \vert >K$, $z\mathbf{1}_{\mathcal{A}_N}-P(s_N, D_N)$ is invertible and the, according to Lemma \ref{inversible}, $ze_{11}\otimes 1_{\mathcal A_N}-L_P(s_N,D_N)$ is invertible.

Let us fix $z_1 \in \C \setminus \R$. We know, using \eqref{bornemanq},  Lemma \ref{technicalbound}, Remark \ref{remmajRNRkLp} and Proposition \ref{boundnorm}, that by  Proposition \ref{analog}, for any $z\in O_{z_1}$, 
\begin{align*}f^{(z_1)}_N(z)&:=
\1_{\{\|W_N\| \leq C\}}\sigma_N^4\sum_{k=1}^N\sum_{i,j<k}\Tr(\gamma_1\mathbb{E}_{\leq k}[R_{ij}^{(k)}(ze_{11}-\gamma_0)]\gamma_1\hat{R}_k(ze_{11}-\gamma_0))\\
& \hspace{1cm}\times \Tr(\gamma_1\mathbb{E}_{\leq k}[R_{ji}^{(k)}(\bar z_1e_{11}-\gamma_0)]\gamma_1\hat{R}_k(\bar z_1e_{11}-\gamma_0))
\end{align*}
converges in probability towards 
\begin{align*}
 f^{(z_1)}(z) & =-\Tr\otimes\Tr\left\{\log\left[\mathrm{id}_m\otimes\mathrm{id}_m-\sigma^2T_{\{ze_{11}-\gamma_0,\bar z_1e_{11}-\gamma_0\}}\right](I_m\otimes I_m)\right\}\\
 & \hspace{0.4cm} -\sigma^2\Tr\otimes\Tr\left[T_{\{ze_{11}-\gamma_0,\bar z_1e_{11}-\gamma_0\}}(I_m\otimes I_m)\right].
\end{align*}
For $N\geq 1$, $f^{(z_1)}_N $ is an holomorphic function on $\C \setminus [-\max(M,K);\max(M,K)]$.

Fix an arbitrary subsequence extracted from $(f^{(z_1)}_N)_{N\geq 1}$. By diagonal extraction from the convergence in probability above, given a countable subset of uniqueness of $O_{z_1}$, one can extract a subsubsequence, let us say  $(f^{(z_1)}_{\Psi(N)})_{N\geq 1}$, almost surely converging to $f^{(z_1)}$
pointwise on this subset.

Using Lemma \ref{technicalbound} and Lemma \ref{alta lemma}, $(f^{(z_1)}_N)_{N\geq 1}$ is a bounded sequence  in $\mathcal{H}(\C \setminus [-\max(M,K);
\max(M,K)])$.
We conclude by Vitali's Theorem that almost surely $(f^{(z_1)}_{\Psi(N)})$ converges towards an holomorphic function on $ \C \setminus [-\max(M,K);
\max(M,K)].$
\\ Note that $f^{(z_1)}$ is holomorphic on $\C \setminus \R$.
Hence  almost surely, for any $z \in  \C \setminus \R$, $f^{(z_1)}_{\Psi(N)}(z)$ converges towards  $f^{(z_1)}(z)$.
Therefore  for any $z \in  \C \setminus \R$, $f^{(z_1)}_N(z)$ converges in probability towards  $f^{(z_1)}(z)$.
Proposition \ref{analog2} readily follows since, by Proposition \ref{boundnorm}, Remark \ref{remmajRNRkLp} and Lemma \eqref{bornemanq}, 
\begin{align*}
  \1_{\{\|W_N\| > C\}} \sigma_N^4 \sum_{k=1}^N\sum_{i,j<k} \Tr(\gamma_1 & \mathbb{E}_{\leq k}[R_{ij}^{(k)}(ze_{11}-\gamma_0)] \gamma_1\hat{R}_k(ze_{11}-\gamma_0))\\
  & \times \Tr(\gamma_1\mathbb{E}_{\leq k}[R_{ji}^{(k)}(\bar z_1e_{11}-\gamma_0)]\gamma_1\hat{R}_k(\bar z_1e_{11}-\gamma_0)) \ = o_{L^1}(1).
\end{align*}
\end{proof}

\subsubsection{Contribution of the third term of \eqref{eq:hook_primitive2}}\label{sec_3rd_term}
\begin{prop}\label{analogtauN}
Fix $z_1,z_2 \in \C \setminus \R$ and set  $\beta_1=z_1e_{11}-\gamma_0$ and $\beta_2= z_2e_{11}-\gamma_0$  in $M_m(\C)$.
The following convergence holds in probability:
\begin{align}
|\theta_N|^2\sum_{k=1}^N\sum_{i,j<k} \Tr(\gamma_1 & \mathbb{E}_{\leq k}[ R_{ij}^{(k)}(\beta_1)]\gamma_1\hat{R}_k(\beta_1))\Tr(\gamma_1\mathbb{E}_{\leq k}[ R_{ij}^{(k)}
(\beta_2)]\gamma_1\hat{R}_k(\beta_2))\nonumber\\
& \nonumber 
\rightarrow_{N\rightarrow +\infty}-{\rm Tr}\otimes\Tr\left\{\log\left[\mathrm{id}_m\otimes\mathrm{id}_m-\theta T_{\{\beta_1,\beta_2\}}\right](I_m\otimes I_m)\right\}\\
& \hspace{6cm} -\theta{\rm Tr}\otimes\Tr\{T_{\{\beta_1,\beta_2\}}(I_m\otimes I_m)\}\nonumber.
\end{align}
\end{prop}

The proof of Proposition \ref{analogtauN} is very similar to the proof of Propositions \ref{analog} and \ref{analog2}. Therefore, we  only notice the main differences.
Instead of \eqref{defnabl}, we define for any $k\in \{1,\ldots,N\}$, 
$$\tilde \nabla_k=|\theta_N|\sum_{i,j<k}\mathbb{E}_{\leq k}[ R_{ij}^{(k)}(\beta_1)]\gamma_1\otimes \gamma_1\mathbb{E}_{\leq k}[ R_{ij}^{(k)}(\beta_2)] \in M_m(\mathbb{C})\otimes M_m(\mathbb{C}),$$
and note that $$|\theta_N|^2\sum_{k=1}^N\sum_{i,j<k}\Tr(\gamma_1\mathbb{E}_{\leq k}[ R_{ij}^{(k)}(\beta_1)]\gamma_1\hat{R}_k(\beta_1))\Tr(\gamma_1\mathbb{E}_{\leq k}[ R_{ij}^{(k)}(\beta_2)]\gamma_1\hat{R}_k(\beta_2)) $$
$$=N|\theta_N|\Tr \otimes \Tr T_N(z_1,z_2)(\tilde \nabla_k).$$ 
Sticking to the proof of  Lemma \ref{nablak}, we can prove the following
\begin{lem}\label{tildenablak}
\begin{align*}
\tilde \nabla_k
&=|\theta_N|\sum_{i<k}\hat{R}_i(\beta_1) \gamma_1 \otimes \gamma_1 \mathbb{E}_{\leq k-1}\left( R_{ii}^{(k)}(\beta_2)\right)\\
& +|\theta_N|\sum_{i,j,l<k}W_{il}\hat {R}_i(\beta_1)\gamma_1 \mathbb{E}_{\leq k-1}\left( R_{lj}^{(kil)}(\beta_1)\right)\gamma_1\otimes \gamma_1 \mathbb{E}_{\leq k-1}\left( R_{ij}^{(k)}(\beta_2)\right)+o_{L^2}^{(u)}(1). 
\end{align*}
\end{lem}
We can also establish the following lemma which is an analog of Lemma \ref{equationapprochee}.
\begin{lem}\label{equationapprocheetilde}
\begin{equation*}
\tilde \nabla_k =N|\theta_N|T_{N,k,0}(z_1,z_2)(I_m\otimes I_m+\tilde \nabla_k)+o_{L^1}(1).
\end{equation*}
\end{lem}
The proof 
of the last lemma starts as the proof of Lemma \ref{equationapprochee} by writing
\begin{align*}
 |\theta_N|&\sum_{i,j,l<k}W_{il}\hat {R}_i(\beta_1)\gamma_1 \mathbb{E}_{\leq k-1}\big( R_{lj}^{(kil)}(\beta_1)\big)\gamma_1\otimes \gamma_1 \mathbb{E}_{\leq k-1}\big( R_{ij}^{(k)}(\beta_2)\big)\\
& = |\theta_N|\sum_{i,j,l<k}W_{il}\hat {R}_i(\beta_1)\gamma_1 \mathbb{E}_{\leq k-1}\big( R_{lj}^{(kil)}(\beta_1)\big)\gamma_1\otimes \gamma_1 \mathbb{E}_{\leq k-1}\big( R_{ij}^{(kil)}(\beta_2)\big)\\
 & +|\theta_N|\sum_{i,j,l<k}(1-\tfrac{1}{2}\delta_{il})W_{il}^2 \hat {R}_i(\beta_1)\gamma_1\mathbb{E}_{\leq k}[ R_{lj}^{(kil)}(\beta_1)]\gamma_1\otimes\gamma_1\mathbb{E}_{\leq k}[ R_{ii}^{(kil)}(\beta_2)\gamma_1 R_{lj}^{(k)}(\beta_2)]\\
 & +|\theta_N|\sum_{i,j,l<k}(1-\tfrac{1}{2}\delta_{il})|W_{il}|^2 \hat {R}_i(\beta_1)\gamma_1\mathbb{E}_{\leq k}[ R_{lj}^{(kil)}(\beta_1)]\gamma_1\otimes \gamma_1\mathbb{E}_{\leq k}[ R_{il}^{(kil)}(\beta_2)\gamma_1 R_{ij}^{(k)}(\beta_2)].
\end{align*}
But now, it is the second term (and note the third one) of the right-hand side that leads to a significant term whereas the other ones are negligible:
\begin{align*}
|\theta_N|\sum_{i,j,l<k}(1-\tfrac{1}{2}\delta_{il})W_{il}^2 \hat {R}_i(\beta_1)\gamma_1\mathbb{E}_{\leq k}[ R_{lj}^{(kil)}(\beta_1)]\gamma_1 & \otimes\gamma_1\mathbb{E}_{\leq k}[ R_{ii}^{(kil)}(\beta_2)\gamma_1 R_{lj}^{(k)}(\beta_2)]\\
& =|\theta_N|\sum_{i<k} T_{N,k,0}(z_1,z_2)(\tilde \nabla_k)
+ o_{L^1}(1).
\end{align*}
The rest of the proof of Proposition \ref{analogtauN} sticks to the proof of Propositions \ref{analog} and \ref{analog2}, using that 
$0\leq |\theta_N|\leq \sigma_N^2$ ensuring the invertibility of the involved operators.

\subsubsection{Contribution of the fourth term of \eqref{eq:hook_primitive2}} \label{sec_4th_term}
To handle the fourth term, define $$f^{(1)}(\omega,x,y)= \Tr\left(\gamma_1(\omega-x\gamma_2)^{-1}\gamma_1(\omega-y\gamma_2)^{-1}\right)$$ for $\omega\in M_m(\mathbb{C}), x,y\in \mathbb{R}$ such that $\omega-x\gamma_2$ and $\omega-y\gamma_2$ are invertible.
Note  that 
\begin{align*}
 \sum_{k=1}^N\kappa_N\sum_{i<k} & \Tr(\gamma_1\mathbb{E}_{\leq k}[R_{ii}^{(k)}(\beta_1)]\gamma_1\hat{R}_k(\beta_1))\Tr(\gamma_1\mathbb{E}_{\leq k}[R_{ii}^{(k)}(\beta_2)]\gamma_1\hat{R}_k(\beta_2))\\
& = \kappa_N\sum_{1\leq i<k\leq N}f^{(1)}({\omega}_N(\beta_1),D_{ii},D_{kk})f^{(1)}({\omega}_N(\beta_2),D_{ii},D_{kk})+I_1 +I_2,
\end{align*}
where $$I_1= \sum_{k=1}^N\kappa_N\sum_{i<k}\Tr(\gamma_1\mathbb{E}_{\leq k}[R_{ii}^{(k)}(\beta_1)-\hat R_i(\beta_1)]\gamma_1\hat{R}_k(\beta_1))\Tr(\gamma_1\mathbb{E}_{\leq k}[R_{ii}^{(k)}(\beta_2)]\gamma_1\hat{R}_k(\beta_2))$$
and $$I_2=\sum_{k=1}^N\kappa_N\sum_{i<k}\Tr(\gamma_1\hat R_{i}(\beta_1)\gamma_1\hat{R}_k(\beta_1))\Tr(\gamma_1\mathbb{E}_{\leq k}[R_{ii}^{(k)}(\beta_2)-\hat R_i(\beta_2)]\gamma_1\hat{R}_k(\beta_2)).$$
Using \eqref{majorationunifhatRk}, Remark \ref{elemdiag}, Remark \ref{remmajRNRkLp} and Proposition \ref{boundnorm}, we readily obtain that for any $p\geq 1$, $\Vert I_1+I_2\Vert_{L^p}\rightarrow_{N\rightarrow +\infty}0.$
Now, noticing that $f^{(1)}(\omega,x,y)=f^{(1)}(\omega,y,x)$ and using Lemma \ref{cvunif}, one can easily see that
\begin{align*}
 \kappa_N \sum_{1\leq i<k\leq N} & f^{(1)}({\omega}_N(\beta_1),D_{ii},D_{kk})f^{(1)}({\omega}_N(\beta_2),D_{ii},D_{kk})\\
 & =\frac{N^2\kappa_N}{2}\iint_{\mathbb{R}^2}f^{(1)}({\omega}_N(\beta_1),x,y)f^{(1)}({\omega}_N(\beta_2),x,y)(\nu_N\otimes \nu_N)(x,y)+O(N^{-1})\\
 & \underset{N\to +\infty}{\longrightarrow} \frac{\kappa}{2}\iint_{\mathbb{R}^2}f^{(1)}(\omega(\beta_1),x,y)f^{(1)}(\omega(\beta_2),x,y)(\nu\otimes\nu)(x,y)
\end{align*}
Therefore,
\begin{align*}
 \sum_{k=1}^N\kappa_N & \sum_{i<k}\Tr(\gamma_1\mathbb{E}_{\leq k}[R_{ii}^{(k)}(\beta_1)]\gamma_1\hat{R}_k(\beta_1))\Tr(\gamma_1\mathbb{E}_{\leq k}[R_{ii}^{(k)}(\beta_2)]\gamma_1\hat{R}_k(\beta_2))\\
 & = \frac{\kappa}{2}\iint_{\mathbb{R}^2}f^{(1)}(\omega(\beta_1),x,y)f^{(1)}(\omega(\beta_2),x,y)(\nu\otimes \nu)(x,y)+o_{\mathbb{P}}(1).
\end{align*}

\subsubsection{Conclusion}\label{sec:CV_hook_ccl}
We obtained that for any $z_1, z_2 \in {(\C \setminus \R)}^2$, $\gamma_N(z_1,z_2)$, defined in \eqref{eq:hook_primitive},  converges in probability towards $\gamma(z_1,z_2)$.
As already observed in \eqref{gammadecomposition},
\begin{align*}
 \gamma_N(z_1,z_2) & =\sum_{k=1}^N\Big\{\tilde \sigma_N^2 \Tr(\gamma_1\hat R_k(\beta_1))\Tr(\gamma_1\hat R_k(\beta_2))\\
 & \hspace{4cm} +\mathbb{E}_{k}\Big[\mathbb{E}_{\leq k}\big[\Tr(\Phi_k(\beta_1))\hat R_k(\beta_1))\big]\mathbb{E}_{\leq k}\big[\Tr(\Phi_k(\beta_2))\hat R_k(\beta_2))\big]\Big]\Big\}.
\end{align*}
Let $C$ be as in Proposition \ref{boundnorm}. If $\Vert W_N\Vert \leq C$, then there exists $M>0$ such that for any $k \in \{1,\ldots,N\}$, $\|P(W_N^{(k)}, D_N^{(k)})\|\leq M.$
Let $K>0$ be such that $\|P(x_N,D_N)\| \leq K$. Set $\tilde \gamma_N(z_1,z_2)=\gamma_N(z_1,z_2) \1_{\{\|W_N\| \leq C\}}$. 
Fix $z_1\in \C \setminus \R$ and set $g_N^{(z_1)}(z)= \tilde \gamma_N(z_1,z)$, $g^{(z_1)}(z)= \gamma(z_1,z)$.
Fix an arbitrary subsequence extracted from $(g^{(z_1)}_N)_{N\geq 1}$. By diagonal extraction from the convergence in probability above, given a countable subset of $\C \setminus \R$, one can extract a subsubsequence, let us say  $(g^{(z_1)}_{\Psi(N)})_{N\geq 1}$, almost surely converging to $g^{(z_1)}$
pointwise on this subset. Cauchy-Schwarz inequality (with respect to $\mathbb{E}_{\leq k}$ and then to the sum over $k$), \eqref{bornemanq}, Lemmas \ref{alta lemma} and \ref{lem_moment_quadratic_forms} readily yield  that $(g^{(z_1)}_N)$ is a bounded sequence  in $\mathcal{H}(\C \setminus [-\max(M,K);\max(M,K)])$.
We conclude by Vitali's Theorem that almost surely $(g^{(z_1)}_{\Psi(N)})$ converges, uniformly on any compact set of $\C \setminus [-\max(M,K);
\max(M,K)]$, towards an holomorphic function on $ \C \setminus [-\max(M,K);
\max(M,K)].$
\\ Note that $g^{(z_1)}$ is holomorphic on $\C \setminus \R$.
Hence  almost surely,  $g^{(z_1)}_{\Psi(N)}$ converges, uniformly on any compact set of $\C \setminus [-\max(M,K);
\max(M,K)]$, towards  $g^{(z_1)}$. This implies that  almost surely,  $\frac{d}{dz}g^{(z_1)}_{\Psi(N)}$ converges, uniformly on any compact set of $\C \setminus [-\max(M,K);
\max(M,K)]$, towards  $\frac{d}{dz}g^{(z_1)}$.
Thus, we obtain that for any $z_1, z_2 \in {(\C \setminus \R)}^2$, $\frac{\partial}{\partial z_2} \tilde \gamma_N(z_1,z_2)$  converges in probability towards $\frac{\partial}{\partial z_2}\gamma(z_1,z_2)$. \\ Now,
fix $z_2\in \C \setminus \R$ and set $h_N^{(z_2)}(z)= \frac{\partial}{\partial z_2} \tilde \gamma_N(z,z_2)$, $h^{(z_2)}(z)=\frac{\partial}{\partial z_2} \gamma(z,z_2)$. The same procedure applied to $h_N^{(z_2)}$  as the one used for $g_N^{(z_1)}$  above yields that for any $z_1, z_2 \in {(\C \setminus \R)}^2$, $\frac{\partial^2}{\partial z_1\partial z_2} \tilde \gamma_N(z_1,z_2)$  converges in probability towards $\frac{\partial^2}{\partial z_1\partial z_2}\gamma(z_1,z_2)$.  \\
Finally, $\gamma_N(z_1,z_2)$ is bounded in $L^2$ 
so that by Lemma \ref{boundnorm}, 
$\gamma_N(z_1,z_2) \1_{\{\|W_N\| >C\}}=o_{\mathbb{P}}(1)$. Proposition \ref{hook} follows.

\section{Tightness of 
	$\{\xi_N(z), z\in \C \setminus \R\}_{N\in \mathbb{N}}$ in  ${\mathcal H}( \mathbb C\setminus \R)$ and conclusion}\label{tightanalytic}

For each $N\in \mathbb{N}$, $\xi_N: z\mapsto  \Tr((zI_N- P(W_N, D_N)^{-1})-\mathbb{E}[\Tr((zI_N- P(W_N, D_N))^{-1})]$ is a random analytic function on $\mathbb{C}\setminus \mathbb{R}$. Let $K$ be a compact set  in $\mathbb{C}\setminus \R$.
According to Lemma \ref{Shirai},  there exists $\delta > 0$ such that $\overline{K_\delta}\subset \mathbb{C}\setminus \R$ and for any $r>0$,
$$\left\| \xi_N \right\|_K^r \leq (\pi \delta^2)^{-1} \int _{\overline{K_\delta}} \left| \xi_N(z)\right|^r m(dz).$$
Therefore 
\begin{align}
\mathbb{E}\left( \left\| \xi_N \right\|_K^r \right)&\leq (\pi \delta^2)^{-1} \int _{\overline{K_\delta}} \mathbb{E} \left( \left| \xi_N(z)\right|^r \right) m(dz)\\ &\leq  (\pi \delta^2)^{-1} \sup_{z\in \overline{K_\delta}}  \mathbb{E} \left( \left| \xi_N(z)\right|^r \right) m({\overline{K_\delta}}). \label{ineg1}
\end{align}

In order to prove the tightness of $\{\xi_N\}_{N\in \mathbb{N}}$ in  ${\mathcal H}( \mathbb C\setminus \R)$, using Proposition \ref{criterion}, \eqref{ineg0} and \eqref{ineg1}, it is sufficient to prove that
\begin{equation}\label{maj}\mathbb{E}\left( \left| \xi_N(z)\right|\right)=O(1).\end{equation}
This will readily follows from the following

\begin{prop}\label{variancebound2}
$$\Var[\Tr(R_N(z))]=O(1).$$
\end{prop}


\begin{proof}
	{From the decomposition \eqref{decomposition}, apply Lemma \ref{martingalevariance} to the martingale $(\mathbb{E}_{\leq k}[\Tr(zI_N-X_N)^{-1}])_{k\geq 1}$ and deduce that 
		$$\Var[\Tr(zI_N-X_N)^{-1}]=\sum_{k=1}^N \mathbb{E}\Big[\big|(\mathbb{E}_{\leq k}-\mathbb{E}_{\leq k-1})[\Tr(zI_N-X_N)^{-1}]\big|^2\Big].$$
		Recall from the preceding section that, setting $\beta:=ze_{11}-\gamma_0$, 
		\begin{align*}
			\Tr & (zI_N-X_N)^{-1}\\
			&=(\Tr\otimes \Tr)((e_{11}\otimes I_{N-1})R^{(k)}(\beta))\\
			&\hspace{0.5cm}+\Tr((e_{11}+\gamma_1\otimes C_k^{(k)*}R^{(k,1)}(\beta)\gamma_1\otimes C_k^{(k)})(\beta-W_{kk}\gamma_1-D_{kk}\gamma_2-\gamma_1\otimes C_k^{(k)*}R^{(k)}(\beta)\gamma_1\otimes C_k^{(k)})^{-1}).
		\end{align*}
		In the second term of the right-hand side, decompose 
		\[e_{11}+\gamma_1\otimes C_k^{(k)*}R^{(k,1)}(\beta)\gamma_1\otimes C_k^{(k)}=(e_{11}+\gamma_1 (\mathrm{id}_m\otimes \sigma_N^2\Tr)(R^{(k,1)}(\beta))\gamma_1)-\frac{\partial}{\partial z}\Phi_k(\beta)\]
		and
		\begin{align*}
		 (\beta-W_{kk}\gamma_1- & D_{kk}\gamma_2-\gamma_1\otimes C_k^{(k)*} R^{(k)}(\beta)\gamma_1\otimes C_k^{(k)})^{-1}\\
		  & =\hat{R}_k(\beta)
		 +\hat{R}_k(\beta)(W_{kk}\gamma_1+\Psi_k(\beta))(\beta-W_{kk}\gamma_1-D_{kk}\gamma_2-\gamma_1\otimes C_k^{(k)*}R^{(k)}(\beta)\gamma_1\otimes C_k^{(k)})^{-1},
		\end{align*}
		so that 
		\begin{align*}
			\Tr (zI_N-X_N)^{-1} & =(\Tr\otimes \Tr)\big((e_{11}\otimes I_{N-1})R^{(k)}(\beta)\big)\\
            & \hspace{0.5cm}+\Tr\Big((e_{11}+\gamma_1 (\mathrm{id}_m\otimes \sigma_N^2\Tr)(R^{(k,1)}(\beta))\gamma_1)\hat{R}_k(\beta)\Big)-\Tr\Big(\frac{\partial}{\partial z}\Phi_k(\beta)\hat{R}_k(\beta)\Big)\\
			& \hspace{0.5cm}+\Tr\Big((e_{11}+\gamma_1\otimes C_k^{(k)*}R^{(k,1)}(\beta)\gamma_1\otimes C_k^{(k)})\hat{R}_k(\beta)\\
			& \hspace{1.5cm} \times (W_{kk}\gamma_1+\Psi_k(\beta))(\beta-W_{kk}\gamma_1-D_{kk}\gamma_2-\gamma_1\otimes C_k^{(k)*}R^{(k)}(\beta)\gamma_1\otimes C_k^{(k)})^{-1}\Big).
		\end{align*}
		Observe that the first two terms satisfy
		\[(\mathbb{E}_{\leq k}-\mathbb{E}_{\leq k-1})[(\Tr\otimes \Tr)((e_{11}\otimes I_{N-1})R^{(k)}(\beta))+\Tr((e_{11}+\gamma_1 (\mathrm{id}_m\otimes \sigma_N^2\Tr)(R^{(k,1)}(\beta))\gamma_1)\hat{R}_k(\beta))]=0\]
		and denote by $T_1$ and $T_2$ the last two terms. Using Jensen's inequality (with respect to $\mathbb{E}_{\leq k}$) after writing $\mathbb{E}_{\leq k-1}=\mathbb{E}_{\leq k}\mathbb{E}_{k}$, 
		\begin{align*}
			\mathbb{E}[|(\mathbb{E}_{\leq k}-\mathbb{E}_{\leq k-1})[\Tr(zI_N-X_N)^{-1}]|^2]
			&= \mathbb{E}[|(\mathbb{E}_{\leq k}-\mathbb{E}_{\leq k-1})[T_1+T_2]|^2]\\
			&\leq\mathbb{E}[|T_1+T_2-\mathbb{E}_k[T_1+T_2]|^2]\\
			&\leq\mathbb{E}[\mathbb{E}_k[|T_1+T_2-\mathbb{E}_k[T_1+T_2]|^2]\\
			&\leq\mathbb{E}[\mathbb{E}_k[|T_1+T_2|^2]]\\
			&\leq2(\mathbb{E}[\mathbb{E}_k[|T_1|^2]]+\mathbb{E}[\mathbb{E}_k[|T_2|^2]]).
		\end{align*}
Bound on $\mathbb{E}[\mathbb{E}_k[|T_1|^2]]$: 
$$|T_1|\leq m\|\frac{\partial}{\partial z}\Phi_k(\beta)\|\|\hat{R}_k(\beta)\|
$$
and deduce from 
Lemma \ref{bornepdphik} 
and $\eqref{majorationunifhatRk}$ that 
\begin{equation}  \mathbb{E}[|T_1|^2]=O(N^{-1}).\end{equation}
		Bound on $\mathbb{E}[\mathbb{E}_k[|T_2|^2]]$:
		by traciality and using \eqref{majorationunifhatRk}, 
		\begin{align*}
			|T_2|
			&=\Big|\Tr\Big((\beta-W_{kk}\gamma_1-D_{kk}\gamma_2-\gamma_1\otimes C_k^{(k)*}R^{(k)}(\beta)\gamma_1\otimes C_k^{(k)})^{-1}\\
			& \hspace{3cm}\times (e_{11}+\gamma_1\otimes C_k^{(k)*}R^{(k,1)}(\beta)\gamma_1\otimes C_k^{(k)})\hat{R}_k(\beta)(W_{kk}\gamma_1+\Psi_k(\beta))\Big)\Big|\\
			&\leq m\big\|(\beta-W_{kk}\gamma_1-D_{kk}\gamma_2-\gamma_1\otimes C_k^{(k)*}R^{(k)}(\beta)\gamma_1\otimes C_k^{(k)})^{-1}\\
			& \hspace{3cm}\times (e_{11}+\gamma_1\otimes C_k^{(k)*}R^{(k,1)}(\beta)\gamma_1\otimes C_k^{(k)})\big\|\|\hat{R}_k(\beta)\|\|W_{kk}\gamma_1+\Psi_k(\beta)\|
		\end{align*}
		
		
		Then,
		\begin{align*}
			|T_2|^2
			&\leq 4m^2\|\hat R_k(\beta)\|^2\|R_N(\beta)\|^2(1+\|\gamma_1\|^4\|C_k^{(k)}\|^4\|R^{(k)}(\beta)\|^4)(|W_{kk}|^2\|\gamma_1\|^2+\|\Psi_k(\beta)\|^2),
		\end{align*}
and consequently, by H\"older's inequality with $q\in [1,2(1+\varepsilon))]$ and $p,r\geq 1$ such that $p^{-1}+q^{-1}+r^{-1}=1$.
		\begin{align*}
			\mathbb{E}[|T_2|^2]
			&\leq O(1)\mathbb{E}\big[\|R_N(\beta)\|^{2p}\big]^{1/p}\mathbb{E}\big[(1+\|\gamma_1\|^4\|C_k^{(k)}\|^4\|R^{(k)}(\beta)\|^4)^r\big]^{1/r}\\
			& \hspace{6cm} \times \mathbb{E}\big[(|W_{kk}|^2\|\gamma_1\|^2+\|\Psi_k(\beta)\|^2)^q\big]^{1/q}\\
			&\leq O(1)\Big(1+\|\gamma_1\|^4\mathbb{E}\big[\|C_k^{(k)}\|^{8r}\big]^{1/2r}\mathbb{E}\big[\|R^{(k)}(\beta)\|^{8r}\big]^{1/2r}\Big)\\
			& \hspace{6cm} \times \Big(\|\gamma_1\|^2\mathbb{E}\big[|W_{kk}|^{2q}\big]^{1/q}+\mathbb{E}\big[\|\Psi_k(\beta)\|^{2q}\big]^{1/q}\Big)\\
			&\leq O(1)(\|\gamma_1\|^2N^{-1}+O(N^{-1}))=O(N^{-1}),
		\end{align*}
	because of Lemma \ref{majRNRkLp}, Remark \ref{remmajRNRkLp}, Lemma \ref{normeqpsik} and assumptions on entries of $W_N$.} 
\end{proof}


\subsection*{Conclusion} 
It follows from the discussion above that $\{\xi_N\}_{N\in \mathbb{N}}$ is tight in  $\mathcal H(\C \setminus \R)$. Then, according to Theorem 5.1 in \cite{Bil}, it is relatively compact. According to 
Section \ref{CLT}, the finite dimensional distributions converge towards those of the Gaussian process $\mathcal{G}$ defined in Theorem \ref{cvfdim}.
Since the class of finite dimensional sets $\left\{\{f\in \mathcal H(\C \setminus \R), (f(z_1), \ldots, f(z_k))\in B\}, k\in \N, (z_1,\ldots,z_k)\in (\C \setminus \R)^k, B \in \mathcal{B}(\C^k)\right\}$ is a separating class, we can deduce Theorem \ref{cvfdim} by Theorem 2.6 in \cite{Bil}.

\appendix

\section{Tools}
\subsection{Linear algebra}
%
%
\begin{prop}[Schur inversion formula]\label{Schur} 
Let $\mathcal{A}$ be a unital complex algebra. For non-empty subsets $I,J$ of $\{1,\ldots ,n\}$ and $A\in M_n(\mathcal{A})$, we denote by $A_{I\times J}$ the submatrix of $A$ corresponding to rows indexed by $I$ 
and columns indexed by $J$. In the particular case where $I=J$, we will use the notation $A_I$.\\
Let $I$ be a non-empty subset of $\{1,\ldots ,n\}$ 
and $A\in M_n(\mathcal{A})$ such that $A_{I}$ is invertible, then $A$ is invertible if and only if $A_{I^c}-A_{I^c\times I}A_{I}^{-1}A_{I\times I^c}$ is invertible, 
in which case the following formulas hold:
\begin{align*}
 (A^{-1})_{I} & =(A_I)^{-1}+ (A_I)^{-1}A_{I\times I^c}(A_{I^c}-A_{I^c\times I}A_{I}^{-1}A_{I\times I^c})^{-1}A_{I^c\times I}(A_I)^{-1},\\
 (A^{-1})_{I \times I^c} & =-(A_I)^{-1}A_{I\times I^c}(A_{I^c}-A_{I^c\times I}A_{I}^{-1}A_{I\times I^c})^{-1},\\
 (A^{-1})_{I^c \times I} & =-(A_{I^c}-A_{I^c\times I}A_{I}^{-1}A_{I\times I^c})^{-1} A_{I^c\times I}(A_I)^{-1},\\
 (A^{-1})_{I^c} & =(A_{I^c}-A_{I^c\times I}A_{I}^{-1}A_{I\times I^c})^{-1}.
\end{align*}
\end{prop}
\begin{prop}\label{partialtranspose} (Theorem 2.9 in \cite{AS})
The transpose map on $M_N(\C)$ induces the
well-defined linear map $\Theta$ on $M_m(\C)\otimes M_N(\C)$, called the partial transpose map:
for $X=\sum_k A_k\otimes B_k$, 
$$\Theta(X) :=\sum_k A_k\otimes B_k^T.$$ For every unitarily invariant norm $\left\|\cdot \right\|$, 
$$\left\|\Theta(X) \right\|\leq \min(m,N) \left\|X \right\|.$$
\end{prop} 
\begin{lem}\label{inversionsomme}
Assume that an operator $A$ is invertible and 
$\Vert A^{-1}\Vert \leq K$.
Then if $\Vert \Delta\Vert \leq (2K)^{-1}$,  $A+\Delta$ is invertible and $\Vert (A+\Delta)^{-1}\Vert \leq 2K$.
\end{lem}
\begin{lem}\label{majnormerayon}
Let $A$ be in $M_n(\C)$ with spectral radius $\rho(A)$. Then for any $z\in \C \setminus \mbox{spect}(A)$, we have 
$$\| (z-A)^{-1}\| \leq \sum_{p=1}^n (d(z, \mbox{spect}(A))^{-p} \left[ \|A \| +\rho(A) \right]^{p-1}.$$
\end{lem}
\begin{proof}
Using Schur decomposition, 
$A=P(D+{\mathcal N}) P^*$, where $P$ is a unitary matrix, $D$ is a diagonal with the same spectrum as $A$ and 
${\mathcal N}$ is a strictly upper triangular matrix. Note that for any $z\in \C \setminus \mbox{spect}(A)$, $(zI_{n}-D)^{-1}{\mathcal N}$ is a nilpotent matrix so that $[(zI_{n}-D)^{-1}{\mathcal N}]^{n}=0$
\[(zI_{n}-A)^{-1}  =P\Big(\sum_{p=0}^{n-1}[(zI_n-D)^{-1}{\mathcal N}]^{p}\Big)(zI_{n}-D)^{-1} P^*\]
Hence
\begin{align*}
\Vert (zI_{n}-A)^{-1}\Vert &\leq\sum_{p=1}^{n}\Vert (zI_{n}-D)^{-1}\Vert^{p}\Vert {\mathcal N} \Vert^{p-1}\\
&\leq \sum_{p=1}^{n}(d(z, \mbox{spect}(A))^{-p} \left[ \|A \| +\rho(A) \right]^{p-1}.
\end{align*}
where we use  $\Vert {\mathcal N} \Vert \leq \Vert P^*  AP \Vert + \Vert D\Vert \leq \Vert A\Vert + \rho(A)$ in the last line. 
\end{proof}

\begin{lem}(Lemma 8.1 \cite{BC})\label{prelim} For any matrix $M =\sum_{i,j=1}^n M_{ij}\otimes E_{ij} \in M_m(\mathbb{C})\otimes M_n(\mathbb{C})$ 
and for any fixed $k$, 
\begin{equation}
\sum_{i=1}^n\|M_{ik} \|^2 \leq m\|M\|^2;\quad \sum_{j=1}^n\|M_{kj} \|^2 \leq m\|M\|^2.
\end{equation}
Hence, 
\begin{equation}\label{lp}
\sum_{i,j=1}^n\|M_{ij}\|^2 \leq mn\|M\|^2
\end{equation}

\end{lem}

\begin{lem}\label{technicalbound}
Let $m\in \mathbb{N}$, $1 \leq k\leq n$ and $A=\sum_{i,j=1}^nA_{ij}\otimes E_{ij},B=\sum_{i,j=1}^nB_{ij}\otimes E_{ij}\in \mathcal{M}_m(\mathbb{C})\otimes \mathcal{M}_n(\mathbb{C})$. Define $C\in \mathcal{M}_m(\mathbb{C})\otimes \mathcal{M}_m(\mathbb{C})\otimes \mathcal{M}_n(\mathbb{C})$ by $C= \sum_{i,j=1}^nC_{ij}\otimes E_{ij}$ where $C_{ij}=\sum_{l<k}A_{il}\otimes B_{lj}$, for $1 \leq i,j\leq n$.
Then $\|C\|\leq \| A\| \| B\|$. In particular, 
\[\|C_{ij}\|
\leq \| A\| \| B\|,\quad 1 \leq i,j\leq n;\]
\[\big(\sum_{i=1}^n\|C_{ij}\|^2\big)^{1/2}
\leq \| A\| \| B\|,\quad 1 \leq j\leq n.\]
\end{lem}
\begin{proof}
Let $\tilde{A}:=\sum_{i,l_1=1}^nA_{il_1}\otimes I_m\otimes E_{il_1}$, $P=\sum_{l<k}I_m\otimes I_m\otimes E_{ll}$ and $\tilde{B}:=\sum_{l_2,j=1}^nI_m\otimes B_{l_2j}\otimes E_{l_2j}$.
Then 
\begin{align*}
\tilde{A}P\tilde{B}
&=\sum_{i,l_1,l<k,l_2,j}A_{il_1}\otimes B_{l_2j}\otimes E_{il_1}E_{ll}E_{l_2j}\\
&=\sum_{i,j=1}^n\sum_{l<k}A_{il}\otimes B_{lj}\otimes E_{ij}\\
&=\sum_{i,j=1}^nC_{ij}\otimes E_{ij}\\
&=C.
\end{align*}
Hence $\Vert C\Vert \leq \Vert\tilde{A}\Vert \Vert P\Vert \Vert\tilde{B}\Vert$.
Observe that 
\[\Vert \tilde{B}\Vert=\Vert I_m\otimes B\Vert=\Vert B\Vert,\]
\[\Vert P\Vert=\Vert I_m\otimes I_m\otimes\sum_{l<k}E_{ll}\Vert=\Vert \sum_{l<k}E_{ll}\Vert=1\]
and
\[\Vert \tilde{A}\Vert=\Vert I_m\otimes \sum_{i,l_1=1}^nE_{il_1}\otimes A_{il_1}\Vert =\Vert \sum_{i,l_1=1}^n E_{il_1}\otimes A_{il_1}\Vert =\Vert A\Vert.\]

\end{proof}

\begin{lem}\label{lipschitz}
If $\varphi:\R \to \C$ is Lipschitz continuous, then, 
for $n\times n$ Hermitian matrices $M_1, M_2$ and $p\geq 1$,
\[|\Tr(\varphi(M_1))-\Tr(\varphi(M_2))|^p\leq \| \varphi \|_{\text{Lip}}^pn^{p-1}\| M_1-M_2\|_{S^p}^p,\]
where $\| A\|_{S^p}=\big(\sum_{\lambda \in \text{sp}(A)}|\lambda|^p\big)^{1/p}$ is the Schatten $p$-norm of the normal matrix $A$.
\end{lem}
\begin{proof}
Denote by $\lambda_1\geq \cdots \geq \lambda_n$ the eigenvalues of $M_1$ 
and $\mu_1\geq \cdots \geq\mu_n$ the eigenvalues of $M_2$. 
Then, \[|\Tr(\varphi(M_1))-\Tr(\varphi(M_2))|\leq \sum_{i=1}^n|\varphi(\lambda_i)-\varphi(\mu_i)|\leq \| \varphi \|_{\text{Lip}}\sum_{i=1}^n|\lambda_i-\mu_i|.\]
Using H\"older's and Hoffman-Wielandt inequalities, 
\[|\Tr(\varphi(M_1))-\Tr(\varphi(M_2))|^p\leq \| \varphi \|_{\text{Lip}}^pn^{p-1}\sum_{i=1}^n|\lambda_i-\mu_i|^p\leq \| \varphi \|_{\text{Lip}}^pn^{p-1}\| M_1-M_2\|_{S^p}^p.\]
\end{proof}

\begin{thm}\label{real}(\cite{E}Theorem 2.5)
Let $\Phi$ be a positive linear map on a finite-dimensional ${C}^*$-algebra $\mathcal{A}$. If $\rho$ is the spectral radius of $\Phi$, there is a non-zero positive element z in $\mathcal{A}$ such that $\Phi(z) =\rho z$.
\end{thm}

\begin{lem}\label{cp} 
Let $\Phi$ and $\Psi$ be positive linear maps on $M_m(\mathbb{C})$ such that $\Phi \leq \Psi$. Then their spectral radii satisfy $\rho(\Phi) \leq \rho(\Psi)$.
\end{lem}

\begin{proof}
The proof we give here follows the proof of Theorem 2.5 in \cite{E}. One may assume without loss of generality that $\Phi$ and $\Psi$ are irreducible. Indeed, given $\chi$ a fixed irreducible positive linear map on $M_m(\mathbb{C})$, $\Phi_n=\Phi+n^{-1}\chi$ and $\Psi_n=\Psi+n^{-1}\chi$ are irreducible positive linear maps on $M_m(\mathbb{C})$ such that $\Phi_n \leq \Psi_n$ and converging respectively to $\Phi$ and $\Psi$ in norm. If the result holds for $\Phi_n$ and $\Psi_n$, letting $n$ tend to $+\infty$ in $\rho(\Phi_n) \leq \rho(\Psi_n)$ gives the conclusion by continuity of the spectral radius in finite dimension.
Assume then that $\Phi$ and $\Psi$ are irreducible. According to Theorem 2.4 in \cite{E} and sentences below this Theorem 2.4 in \cite{E}, the spectral radius of irreducible positive linear maps $\chi$ on a finite dimensional $C^*$-algebra satisfy
\begin{equation}\label{radius} 
\rho(\chi)=\max_{y\geq 0} \inf \{ \alpha \in \mathbb{R}, \alpha y \geq \chi(y)\}.
\end{equation}
We have by assumption that for any $y\geq 0$, $\Psi(y)\geq \Phi(y).$
Thus, \eqref{radius} readily implies that $\rho(\Psi)\geq \rho(\Phi)$.
\end{proof}

\subsection{Concentration bounds for quadratic forms}
One can easily deduce the following result from Lemma 2.7 in \cite{BS98}:

\begin{lem}\label{lem_moment_quadratic_forms}
Let $m, n\in \mathbb{N}$, $\gamma\in \mathcal{M}_m(\mathbb{C})$ and $A=\sum_{i,j=1}^m e_{ij}\otimes A_{ij}\in \mathcal{M}_m(\mathbb{C})\otimes \mathcal{M}_n(\mathbb{C})$. Let $p \geq 2$ and $Y=(Y_1, \ldots , Y_n)$ be a $n$-tuple of independent identically distributed standard complex random variables with finite $2p$-th moment, then 
\begin{align*}
 \esp\big[\|\gamma\otimes Y^*A\gamma \otimes Y & -\gamma \mathrm{id}_m\otimes \Tr (A)\gamma \|^p\big]\\
 & \leq K_{m,p}\|\gamma\|^{2p} \Big(\big(\esp\big[|Y_1|^4\big]\max_{i,j}\Tr (A_{ij}A_{ij}^*)\big)^{p/2}+\esp\big[|Y_1|^{2p}\big]\max_{i,j}\Tr((A_{ij}A_{ij}^*)^{p/2})\Big).
\end{align*}
\end{lem}

\subsection{Martingales}
The proofs of our variance bounds and of our CLT rely on martingale theory.
\begin{lem}\label{martingalevariance}
Let $(M_k)_{k \in \N}$ be a martingale  with values in $\mathbb{C}$ and satisfying $\mathbb{E}[\vert M_k\vert^2]<+\infty$, $k\in \mathbb{N}$. 
Then
\[\mathbb{E}[\vert\sum_{k=1}^N (M_k-M_{k-1})\vert^2]=\sum_{k=1}^N \mathbb{E}[\vert(M_k-M_{k-1})\vert^2],\quad N\in \mathbb{N}.\]
\end{lem}

\begin{lem}\label{martingalelp}
Let $(M_k)_{k \in \N}$ be a $M_n(\mathbb{C})$-valued martingale and $p\in \mathbb{N}$ be an even integer. Then
\[\mathbb{E}[\Vert M_N-M_0\Vert^{p}]\leq n^{p/2}{p!}{{N+p-2}\choose{p-1}}\max_{k=1,\ldots ,N}\mathbb{E}[\Vert M_k-M_{k-1}\Vert^{p}],\quad N\in \mathbb{N}.\]
\end{lem}

\begin{proof}
We assume that $\max_{k=1,\ldots ,N}\mathbb{E}[\Vert M_k-M_{k-1}\Vert^{p}]<+\infty$. Recall that $\Vert \cdot \Vert\leq \Vert \cdot \Vert_{HS}\leq n^{1/2}\Vert \cdot \Vert$.
Observe that 
\begin{align*}
\Vert M_N-M_0\Vert_{HS}^{p}
&=\Vert\sum_{k=1}^N (M_k-M_{k-1})\Vert_{HS}^p\\
&=\sum_{i:\{1,\ldots ,p\}\to \{1,\ldots ,N\}}\langle M_{i_1}-M_{i_1-1},M_{i_2}-M_{i_2-1}\rangle \cdots \langle M_{i_{p-1}}-M_{i_{p-1}-1},M_{i_p}-M_{i_p-1}\rangle.
\end{align*}
Note that, using H\"older's inequality,
\begin{align*}
\mathbb{E}[|\langle M_{i_1}-M_{i_1-1},M_{i_2}-M_{i_2-1}\rangle \cdots \langle M_{i_{p-1}} & -M_{i_{p-1}-1},M_{i_p}-M_{i_p-1}\rangle|]\\
&\leq \mathbb{E}[\Vert M_{i_1}-M_{i_1-1}\Vert_{HS}\cdots \Vert M_{i_p}-M_{i_p-1}\Vert_{HS}]\\
&\leq \max_{k=1,\ldots ,N}\mathbb{E}[\Vert M_k-M_{k-1}\Vert_{HS}^{p}]\\
&\leq n^{p/2}\max_{k=1,\ldots ,N}\mathbb{E}[\Vert M_k-M_{k-1}\Vert^{p}].
\end{align*}
It follows that 
\[\mathbb{E}[\Vert M_N-M_0\Vert_{HS}^{p}]=\sum_{{i:\{1,\ldots ,p\}\to \{1,\ldots ,N\}}}\mathbb{E}[\langle M_{i_1}-M_{i_1-1},M_{i_2}-M_{i_2-1}\rangle \cdots \langle M_{i_{p-1}}-M_{i_{p-1}-1},M_{i_p}-M_{i_p-1}\rangle].\]
Consider a term indexed by $i$ such that $i^{-1}(\max i)$ is a singleton. Then
\begin{align*}
 \mathbb{E}  [\langle M_{i_1}- & M_{i_1-1},M_{i_2}-M_{i_2-1}\rangle \cdots \langle M_{i_{p-1}}-M_{i_{p-1}-1},M_{i_p}-M_{i_p-1}\rangle]\\
 & =\mathbb{E}[\mathbb{E}[\langle M_{i_1}-M_{i_1-1},M_{i_2}-M_{i_2-1}\rangle \cdots \langle M_{i_{p-1}}-M_{i_{p-1}-1},M_{i_p}-M_{i_p-1}\rangle|\mathcal{F}_{\max i-1}]]\\
 & =\mathbb{E}\Big[\langle M_{i_1}-M_{i_1-1},M_{i_2}-M_{i_2-1}\rangle \cdots \langle \mathbb{E}[M_{\max i}-M_{\max i-1}|\mathcal{F}_{\max i-1}],*\rangle \\
 & \hspace{7cm} \times \cdots \langle M_{i_{p-1}}-M_{i_{p-1}-1},M_{i_p}-M_{i_p-1}\rangle\Big]\\
 & =0.
\end{align*}
There are at most ${p!}{{N+p-2}\choose{p-1}}$ choices of indices $i$ such that $i^{-1}(\max i)$ is not a singleton. Indeed, for a map $i:\{1,\ldots ,p\}\to \{1,\ldots ,N\}$, there are $p!$ ways to rank  $i_1, \ldots, i_p$ in increasing order. Now, since $i^{-1}(\max i)$ is not a singleton, we know that at least the two last values of the increasing sequence are equal; since there are  ${{N+p-2}\choose{p-1}}$ choices of increasing sequence of $p-1$ numbers  in $\{1,\ldots,N\}$, the result follows.
\end{proof}


The following result may be deduced from its real-valued analogue (Theorem 35.12 in \cite{Billingsleybook}).

\begin{thm}\label{thm_CLT_martingale}
Suppose that, for all $N\geq 1$, $(M_k^{(N)})_{k \in \N}$ is a square integrable complex martingale and define, for $k \geq 1$, $\Delta_k^{(N)}:=M_k^{(N)}-M_{k-1}^{(N)}$. If
\begin{equation}\label{L}
\forall \varepsilon > 0,\, L(\varepsilon,N):=\sum_{k= 1}^N\mathbb{E}[|\Delta_k^{(N)}|^2\mathbf{1}_{|\Delta_k^{(N)}|\geq \varepsilon}]\underset{N\to +\infty}{\longrightarrow} 0,
\end{equation}
\begin{equation}\label{v}
V_N:=\sum_{k= 1}^N\mathbb{E}_{\leq k-1}[|\Delta_k^{(N)}|^2]\underset{N\to +\infty}{\longrightarrow} v\geq 0,
\end{equation}
and
\begin{equation}\label{w}
W_N:=\sum_{k= 1}^N\mathbb{E}_{\leq k-1}[(\Delta_k^{(N)})^2]\underset{N\to +\infty}{\longrightarrow} w \in \mathbb{C}
\end{equation}
(convergences in \eqref{v} and \eqref{w} have to be understood in probability), then 
\[\sum_{k=1}^N\Delta_k^{(N)}\Rightarrow_{N\to +\infty}\mathcal{N}_{\mathbb{C}}(0,v,w).\]
\end{thm}

\subsection{Complex analysis}
\begin{thm}\label{Vitali}[Vitali, see \cite{Sc05} Exercise 1.4.37]
Let $D \subset \mathbb{C}^l$ be a domain, $A\subset D$ a set of uniqueness (for instance an open set) and $(f_n)_{n\in \mathbb{N}}$ a bounded sequence in the set $\mathcal{H}(D)$ of holomorphic functions on $D$
(that is $\sup_{n\in \mathbb{N}}\sup_{x\in K} \Vert f_n(x)\Vert < +\infty$ for any compact  subset $K\subset D$), which converges pointwise on $A$. Then the sequence  $(f_n)_{n\in \mathbb{N}}$ converges towards a holomorphic function $f \in \mathcal {H}(D)$.
\end{thm}
\begin{thm}\label{Earle-Hamilton}[Earle-Hamilton]
Let $D$ be a nonempty domain in a
complex Banach space X and let $f : D \mapsto D$ be a bounded holomorphic
function. If $f(D)$ lies strictly inside $D$, then $f$ is a strict contraction in the Carath\'eodory-Riffen-Finsler metric  and thus has a unique fixed point in D. Furthermore,
 $(f^n(x_0))_{n\in \mathbb{N}}$ converges in norm, for any $x_0\in D$, to
this fixed point.
\end{thm}
We  recall here a criterion of tightness for random analytic functions from \cite{S}. Let $D \subset \mathbb{C}$ be an open set of the complex plane. Denote by ${\mathcal H} (D)$  the space of complex analytic functions on $D$, endowed with the uniform topology on compact sets. For $f \in {\mathcal H} (D)$ and $K$ a compact set of $D$, we denote $\Vert f \Vert_K =\sup_{z\in K} \left| f(z)\right|$. 
The space ${\mathcal H} (D)$ is equipped with the (topological)
Borel $\sigma$-field ${\mathcal B}({\mathcal H} (D))$ and the set of probability measures on $({\mathcal H}(D); {\mathcal B}({\mathcal H} (D)))$ is
denoted by ${\mathcal P}({\mathcal H}(D))$.
By a random analytic function on $D$ we mean an ${\mathcal H} (D)$-valued
random variable on a probability space. 
\begin{prop}\label{criterion}
	(Proposition 2.5. in \cite{S}) Let $(f_n)_{n\in \mathbb{N}}$ be a sequence of random analytic functions on $D$. If $\{\Vert f_n\Vert_K\}_{n\in \mathbb{N}}$ is tight for any compact set K, then $\{{\mathcal L}(f_n)\}_{n\in \mathbb{N}}$ is
	tight in ${\mathcal P}({\mathcal H}(D))$.
	
\end{prop}
Using that, by Markov's inequality,   for any $C>0$ and any $r>0$, 
\begin{equation}\label{ineg0}\mathbb{P} \left( \left\| f_n \right\|_K >C \right) \leq \frac{1}{C^r} \mathbb{E} \left( \left\| f_n \right\|^r_K \right),\end{equation}
the following  lemma turns out to be useful to prove  tightness results.

\begin{lem}\label{Shirai}(lemma 2.6 \cite{S}) 
	For any compact set K in D, there exists $\delta > 0$ such that
	$$\left\| f \right\|_K^r \leq (\pi \delta^2)^{-1} \int _{\overline{K_\delta}} \left| f(z)\right|^r m(dz), \; f \in {\mathcal H}(D),$$
	for any $r > 0$, where $\overline{K_\delta}\subset D$ is the closure of the $\delta$-neighborhood of K  and $m$ denotes the Lebesgue measure.
\end{lem}

\section{Norm of Wigner matrices}\label{Norm}

\begin{lem}\label{puissance}
Let, for each $N\in \mathbb{N}$, $W_N$ be a $N\times N$ Hermitian matrix such that entries $\{W_{ij}\}_{1\leq i\leq j\leq N}$ are random variables bounded by $\delta $ and such that for $i\neq j$, for some $\varepsilon>0$, $\mathbb{E}[|\sqrt{N}W_{ij}|^{6(1+\varepsilon)}]\leq C_6,$ and $\mathbb{E}[|\sqrt{N}W_{ii}|^{4(1+\varepsilon)}]\leq C_4.$ Then $$\forall p \in [2; 4(1+\varepsilon)], \;\mathbb{E}[|W_{ij}|^{2p}]=O(N^{-p/2-1}) \mbox{~ and ~} \mathbb{E}[|W_{ii}|^{p}]=O(N^{-p/2}).$$
\end{lem}

\begin{proof}
By Jensen's inequality, for $p \in [2; 4(1+\varepsilon)],$
$$ \mathbb{E}[|W_{ii}|^{p}]\leq \mathbb{E}[|W_{ii}|^{4(1+\varepsilon)}]^{p/4(1+\varepsilon)}\leq C_4^{p/4(1+\varepsilon)} N^{-p/2}=O(N^{-p/2}).$$
Similarly, by Jensen's inequality, for $p\in [2; 3(1+\varepsilon)[,$
$$\mathbb{E}[|W_{ij}|^{2p}]\leq \mathbb{E}[|W_{ij}|^{6(1+\varepsilon)}]^{p/3(1+\varepsilon)}\leq C_6^{p/3(1+\varepsilon)}N^{-p}=O(N^{-p})=O(N^{-p/2-1}),$$
the last equality following from the fact that $p\geq 1+p/2$ when $p\geq 2$.\\
Now, for $p\in [3(1+\varepsilon); 4(1+\varepsilon)],$  
$$\mathbb{E}[|W_{ij}|^{2p}]\leq \delta^{2p-6(1+\varepsilon)}\mathbb{E}[|W_{ij}|^{6(1+\varepsilon)}]\leq \delta^{2p-6(1+\varepsilon)}C_6N^{-3(1+\varepsilon)}=O(N^{-3(1+\varepsilon)})=O(N^{-p/2-1}),$$
the last equality following from the fact that $1+p/2\leq 3(1+\varepsilon)$ when $p\leq 4(1+\varepsilon)$.
\end{proof}

\begin{prop}\label{boundnorm}
There exists $C>0$ such that for every $p\geq 1$, $\mathbb{P}(\|W_N\|>C)=o(N^{-p})$. In particular, the sequence of random variables $(\|W_N\|)_{N\geq 1}$ is bounded in every $L^p,\, p\geq 1$.
\end{prop}

\begin{proof}
By assumption, entries of $W_N$ satisfy 
$$\mathbb{E}[\sqrt{N}W_{ij}]=0,\quad \mathbb{E}[|\sqrt{N}W_{ij}|^2]\leq \Sigma^2,\quad 
\mathbb{E}[|\sqrt{N}W_{ij}|^{\ell}]\leq b(\delta_N\sqrt{N})^{\ell-3},\, (\ell\geq 3).$$
For example, $\Sigma=\max(\sup_{N\in \mathbb{N}}\sqrt{N}\sigma_N,\sup_{N\in \mathbb{N}}\sqrt{N}\tilde{\sigma}_N)$ and $b=\max(C_6^{1/2(1+\varepsilon)},C_4^{3/4(1+\varepsilon)})$.
It then follows from Remark 5.7 in the book of Bai and Silverstein that $\mathbb{P}(\|W_N\|>C)=o(N^{-p})$ for any $C>2\Sigma$ and any $p\geq 1$. 
Then, for such $C,p$ and $N\in \mathbb{N}$,
\begin{align*}
\mathbb{E}[\|W_N\|^p]
&=\mathbb{E}[\|W_N\|^p\mathbf{1}_{\|W_N\|\leq C}]+\mathbb{E}[\|W_N\|^p\mathbf{1}_{\|W_N\|>C}]\\
&\leq C^p+(N\delta_N)^p\mathbb{P}(\|W_N\|>C)
\end{align*}
is bounded uniformly in $N$. 
\end{proof}

\section{Truncation and centering}\label{Truncation}

Fluctuations of the trace of the resolvent of $X_N$ were studied under the hypothesis that entries of $W_N$ are bounded by $\delta_N$, for a sequence $(\delta_N)_{N\in \mathbb{N}}$ slowly converging to $0$. 
		
For any bounded continuous function $\varphi: \R \to \C$, let 
$$\mathcal{N}_N(\varphi):=\Tr(\varphi(X_N))=\sum_{\lambda \in \text{sp}(X_N)}\varphi(\lambda)
.$$
In this section, we truncate and center the entries of $W_N$, in order to show that it is sufficient to study the fluctuations of $\mathcal{N}_N(\varphi)$ for matrices $W_N$ with entries bounded by $\delta_N$, where $(\delta_N)_{N\geq 1}$ is a sequence of positive numbers such that $\delta_N \underset{N\to +\infty}{\longrightarrow} 0$ at rate less than $N^{-\epsilon}$ for any $\epsilon >0$.
		
Define $\hat{X}_N=P(\hat{W}_N,D_N)$ by $$\hat{W}_{ij}:=W_{ij}\mathbf{1}_{|W_{ij}|\leq \delta_N/2},\quad 1\leq i, j\leq N,$$
and accordingly 
$$\hat{\mathcal{N}}_N(\varphi):=\Tr(\varphi(\hat{X}_N)).$$
By union bound, 
\begin{align*}
\mathbb{P}(\hat{\mathcal{N}}_N(\varphi)\neq \mathcal{N}_N(\varphi))&\leq \mathbb{P}(\hat{W}_N\neq W_N)\\
&\leq \sum_{1\leq i\leq j\leq N}\mathbb{P}(|W_{ij}| > \delta_N/2)\\
&\leq \sum_{1\leq i< j\leq N}(\delta_N/2)^{-6(1+\varepsilon)}\mathbb{E}[|W_{ij}|^{6(1+\varepsilon)}]+\sum_{1\leq i\leq N}(\delta_N/2)^{-4(1+\varepsilon)}\mathbb{E}[W_{ii}^{4(1+\varepsilon)}]\\
&\leq \sum_{1\leq i< j\leq N}(\delta_N/2)^{-6(1+\varepsilon)}\frac{C_6}{N^{3(1+\varepsilon)}}+\sum_{1\leq i\leq N}(\delta_N/2)^{-4(1+\varepsilon)}\frac{C_4}{N^{2(1+\varepsilon)}}\\
&\leq C_6(\delta_N/2)^{-6(1+\varepsilon)}N^{-1-3\varepsilon}+C_4(\delta_N/2)^{-4(1+\varepsilon)}N^{-1-2\varepsilon}=o(N^{-1}).
\end{align*}
Using the naive bound $|\hat{\mathcal{N}}_N(\varphi)- \mathcal{N}_N(\varphi)|\leq 2\|\varphi \|_{\infty}N$ yields: 
\begin{align*}
\mathbb{E}[|\hat{\mathcal{N}}_N(\varphi)- \mathcal{N}_N(\varphi)|]&=\mathbb{E}[|\hat{\mathcal{N}}_N(\varphi)- \mathcal{N}_N(\varphi)|\mathbf{1}_{\hat{\mathcal{N}}_N(\varphi)\neq \mathcal{N}_N(\varphi)}]\\
&\leq 2\|\varphi \|_{\infty}N\mathbb{P}(\hat{\mathcal{N}}_N(\varphi)\neq \mathcal{N}_N(\varphi))\\
&\leq o(1).
\end{align*}
Define then $\mathring{X}_N=P(\mathring{W}_N,D_N)$ by $\mathring{W}_N=\hat{W}_N-\mathbb{E}[\hat{W}_N]$
and accordingly
$$\mathring{\mathcal{N}}_N(\varphi):=\Tr(\varphi(\mathring{X}_N)).$$
Note that the entries of $\mathring{W}$ are independent, centred and bounded by $\delta_N$. Furthermore, the off-diagonal entries are independent and identically distributed, as well as entries on the diagonal.
Observe that, for $i\neq j$, 
\[|\mathbb{E}[\hat{W}_{ij}]|=|\mathbb{E}[W_{ij}\mathbf{1}_{|W_{ij}|\leq \frac{\delta_N}{2}}]|=|\mathbb{E}[W_{ij}\mathbf{1}_{|W_{ij}|>\frac{\delta_N}{2}}]|=O(\delta_N^{-5-6\varepsilon}N^{-3(1+\varepsilon)}).\]
Furthermore, 
\begin{align*}
\mathbb{E}[|\mathring{W}_{ij}|^2] & = \mathbb{E}[|\hat{W}_{ij}-\mathbb{E}[\hat{W}_{ij}]|^2]= \mathbb{E}[|\hat{W}_{ij}|^2]+R_2,
\end{align*}
with $|R_2| \leq 3|\mathbb{E}[\hat{W}_{ij}]|^2=O(\delta_N^{-10-12\varepsilon}N^{-6(1+\varepsilon)})$. Moreover
$\mathbb{E}[|\hat{W}_{ij}|^2]=\sigma_N^2-\esp[|W_{ij}|^2\mathbf{1}_{|W_{ij}|>\frac{\delta_N}{2}}]$, and $\esp[|W_{ij}|^2\mathbf{1}_{|W_{ij}|>\frac{\delta_N}{2}}]=O(\delta_N^{-4-6\varepsilon}N^{-3(1+\varepsilon)})$. Therefore, 
\[\mathbb{E}[|\mathring{W}_{ij}|^2]=\mathring{\sigma}_N^2\]
with $\mathring{\sigma}^2_N=\sigma_N^2+O(\delta_N^{-3-4\varepsilon}N^{-3(1+\varepsilon)})$. As a consequence, $N\mathring{\sigma}_N^2 \underset{N\to +\infty}{\longrightarrow} \sigma^2$.
		
We turn now to $\mathbb{E}[\mathring{W}_{ij}^2]$.
\[\mathbb{E}[\mathring{W}_{ij}^2]= \mathbb{E}[W_{ij}^2\mathbf{1}_{|W_{ij}|\leq \frac{\delta_N}{2}}]-\mathbb{E}[\hat{W}_{ij}]^2=\theta_N+\tilde{R}_2,\]
with $\tilde{R}_2=\mathbb{E}[W_{ij}^2\mathbf{1}_{|W_{ij}|>\frac{\delta_N}{2}}]-\mathbb{E}[\hat{W}_{ij}]^2=O(\delta_N^{-4-6\varepsilon}N^{-3(1+\varepsilon)})$. Therefore $\mathbb{E}[\mathring{W}_{ij}^2]=\mathring{\theta}_N$, with $\lim\limits_{N \to +\infty} N\mathring{\theta}_N =\lim\limits_{N\to +\infty}N\theta_N =\theta \in \R$. Note that, even if $\theta_N$ is supposed to be real for all $N$, $\mathbb{E}[\mathring{W}_{ij}^2]$ is not real anymore, but its imaginary part is negligible.
		
Similar bounds may be proved for $\esp[|\mathring{W}_{ij}|^4]$, $\esp[|\mathring{W}_{ij}|^{6(1+\varepsilon)}]$, $\esp[\mathring{W}_{ii}^2]$ and $\esp[|\mathring{W}_{ii}|^{4(1+\varepsilon)}]$ from which it may be shown that the entries of $\mathring{W}$ satisfy the same properties as the ones of $W_N$. In particular, one has for all $i \in \{1,\dots,N\}$, $N\esp[\mathring{W}_{ii}^2]
\underset{N\to +\infty}{\longrightarrow} \tilde{\sigma}^2$, and, for $1 \leq i \neq j \leq N$, $N^2(\esp[|\mathring{W}_{ij}|^4]-2\mathring{\sigma}_N^4-\mathring{\theta}_N^2)=N^2\mathring{\kappa}_N\underset{N\to +\infty}{\longrightarrow}\kappa \in \R$.

Assume now that $\varphi$ is a Lipschitz function. Note that this will be true in particular for functions in $\mathcal{C}^1_c(\R)$. Using first Cauchy-Schwarz inequality then Hoffman-Wielandt inequality (see for example \cite{AGZ} Section 2.1.5), we get
\begin{align*}
|\mathring{\mathcal{N}}_N(\varphi)-\hat{\mathcal{N}}_N(\varphi)|&=\big|\Tr\big(\varphi(\hat{X}_N)-\varphi(\mathring{X}_N)\big)\big|\\
&\leq\sqrt{N}\Big(\sum_{i=1}^N|\varphi(\lambda_i(\hat{X}_N))-\varphi(\lambda_i(\mathring{X}_N))|^2\Big)^{1/2}\\
&\leq\sqrt{N}\|\varphi\|_{\text{Lip}}\Big(\sum_{i=1}^N|\lambda_i(\hat{X}_N)-\lambda_i(\mathring{X}_N)|^2\Big)^{1/2}\\
&\leq\sqrt{N}\|\varphi\|_{\text{Lip}}\|\hat{X}_N-\mathring{X}_N\|_{HS}\\
&\leq N\|\varphi\|_{\text{Lip}}\|\hat{X}_N-\mathring{X}_N\|,
\end{align*}
where we have used in the last line the classical inequality $\|M_N\|_{HS}\leq \sqrt{N}\|M_N\|$ for $N \times N$ matrix $M_N$.
		
Now, in $\hat{X}_N=P(\hat{W}_N,D_N)$, decompose each $\hat{W}_N=\mathring{W}_N+\mathbb{E}[\hat{W}_N]$, so that $\hat{X}_N-\mathring{X}_N$ is a sum of a bounded number of monomials in $\mathring{W}_N$, $\mathbb{E}[\hat{W}_N]$ and $D_N$. All these monomials are of positive degree in $\mathbb{E}[\hat{W}_N]$. Recall that $(\|\mathring{W}_N\|)_{N\geq 1}$ is bounded in all $L^p$, $p\geq 1$ (see Proposition \ref{boundnorm}) and $(\|D_N\|)_{N\geq 1}$ is bounded (from Assumption \ref{hyp:deformation}). Furthermore, $(\|\mathbb{E}[\hat{W}_N]\|)_{N\geq 1}$ is $o(N^{-1})$, by the classical bound $\|\mathbb{E}[\hat{W}_N]\|^2 \leq \sum_{i,j}|\esp[\hat{W}_{ij}]|^2$. Consequently, we deduce that $\mathbb{E}[\|\hat{X}_N-\mathring{X}_N\|]=o(N^{-1})$. Therefore $\mathbb{E}[|\mathring{\mathcal{N}}_N(\varphi)-\hat{\mathcal{N}}_N(\varphi)|]=o(1)$.
		
From these controls of $\mathbb{E}[|\mathring{\mathcal{N}}_N(\varphi)-\hat{\mathcal{N}}_N(\varphi)|]$ and $\esp[|\mathcal{N}_N(\varphi)-\hat{\mathcal{N}}_N(\varphi)|]$, we conclude that 
\[\esp[|\mathcal{N}_N(\varphi)-\mathring{\mathcal{N}}_N(\varphi)|] \underset{N \to +\infty}{\longrightarrow} 0.\]
Hence 
\[\big(\mathcal{N}_N(\varphi)-\mathbb{E}[\mathcal{N}_N(\varphi)]\big)-\big(\mathring{\mathcal{N}}_N(\varphi) -\mathbb{E}[\mathring{\mathcal{N}}_N(\varphi)]\big) \overset{\P}{\underset{N \to +\infty}{\longrightarrow}} 0.\]
Therefore, by Slutsky's Lemma, assuming that $\mathring{\mathcal{N}}_N(\varphi)-\esp[\mathring{\mathcal{N}}_N(\varphi)]$ converges to a Gaussian variable yields that $\mathcal{N}_N(\varphi)-\esp[\mathcal{N}_N(\varphi)]$ converges to the same Gaussian variable.
		
		
As a consequence, for our purposes, we may suppose that the entries of $W_N$ are bounded almost surely by $\delta_N$, as long as $\varphi$ is a Lipschitz function.

\bibliographystyle{alpha}
\bibliography{biblio-stat-lin-polynomes}

\end{document}